\documentclass[10pt,reqno,final]{amsart}
\usepackage{epsfig,amssymb,amsmath,version}
\usepackage{amssymb,version,graphicx,fancybox,mathrsfs}
\usepackage[notcite,notref]{showkeys}
\usepackage{url,hyperref}
\usepackage{subfigure}
\usepackage{color}
\usepackage{stmaryrd}
\usepackage{multirow}
\usepackage{booktabs,siunitx}
\usepackage{bm} 
\usepackage {verbatim}
\usepackage{booktabs}
\usepackage{multirow}
\usepackage[title]{appendix}
\usepackage[ruled,vlined]{algorithm2e}
\allowdisplaybreaks[4]
\catcode`\@=11 \theoremstyle{plain}
\@addtoreset{equation}{section}   
\renewcommand{\theequation}{\arabic{section}.\arabic{equation}}
\@addtoreset{figure}{section}
\renewcommand\thefigure{\thesection.\@arabic\c@figure}
\newtheorem{thm}{\bf Theorem}

\newtheorem{cor}{\bf Corollary}
\newtheorem{prop}{Proposition}[section]

\theoremstyle{lemma}
\newtheorem{lm}{\bf Lemma}[section]
\theoremstyle{remark}
\newtheorem{rem}{\bf Remark}[section]

\usepackage{amsthm} 
\newtheoremstyle{example}{}{}{}{}{\bfseries}{\smallskip}{\newline}{}
\theoremstyle{example}
\newtheorem{example}{Example}

\def \epsilon {{\varepsilon}}

\definecolor{bgblue}{rgb}{0.04,0.39,0.54}
\definecolor{lired}{rgb}{0.3, 0.0, 0.0}
\definecolor{ligreen}{rgb}{0.0, 0.3, 0.0}
\definecolor{liblue}{rgb}{0.9, 1.0, 1.0}
\definecolor{gray}{rgb}{0.6, 0.6, 0.6}
\definecolor{sky}{rgb}{0.3, 1.0, 1.0}
\definecolor{bunhong}{rgb}{1.0, 0.3, 1.0}
\definecolor{yellow}{rgb}{0.97, 1, 0.0}
\definecolor{liyellow}{rgb}{0.9, 0.8, 0.0}
\definecolor{cengse}{rgb}{0.00,0.40,0.29}

\newcommand{\mathleft}{\@fleqntrue\@mathmargin0pt}

\newcommand{\bs}[1]{\boldsymbol{#1}}

\renewcommand \wedge \times

\graphicspath{{../figs/} }

\begin{document}
\bibliographystyle{plain}

{\title[Exact divergence-free spectral method for MHD system] {An exact divergence-free spectral method for incompressible and resistive magneto-hydrodynamic equations in two and three dimensions}
\author[
  L. Qin,\,    H. Li\,  $\&$\,  Z. Yang
  ]{
	  \;\; Lechang Qin${}^{1}$,   \;\;  Huiyuan Li${}^{2}$ \;\; and\;\; Zhiguo Yang${}^{*,1}$
	  }
\thanks{ \noindent  ${}^{*}$ Corresponding author.
  \\
\noindent${}^{1}$ School of Mathematical Sciences, MOE-LSC and CMA-Shanghai, Shanghai Jiao Tong University, Shanghai 200240, China. Email: tony$\_$qin0225@sjtu.edu.cn (L. Qin), yangzhiguo@sjtu.edu.edu (Z. Yang).
  \\
  \noindent${}^{2}$ State Key Laboratory of Computer Science/Laboratory of Parallel Computing, Institute of
Software, Chinese Academy of Sciences, Beijing 100190, China. Emails: huiyuan@iscas.ac.cn (H. Li).
	  }
	  }
\keywords{Magneto-hydrodynamic equations, spectral method, generalised Jacobi polynomial, divergence-free condition, unconditional stability.} \subjclass[2000]{65N35, 65N22, 65F05, 35J05}

\begin{abstract}
In this paper, we present exact divergence-free spectral method for solving the incompressible and resistive magneto-hydrodynamic (MHD) equations in two and three dimensions, as well as the efficient solution algorithm and unconditionally energy-stable fully-discretized numerical schemes. We introduce new ideas of constructing two families of exact divergence-free vectorial spectral basis functions on domains diffeomorphic to squares or cubes. These bases are obtained with the help of orthogonality and derivative relation of generalised Jacobi polynomials, several de Rham complexes, as well as the property of contravariant Piola transformation. They are well-suited for discretizing the velocity and magnetic fields, respectively, thereby ensuring point-wise preservation of the incompressibility condition and the magnetic Gauss's law.  With the aid of these bases,  we propose a family of exact divergence-free implicit-explicit $k$-step backward differentiation formula (DF-BDF-$k$) fully-discretized schemes for the MHD system. These schemes naturally decouple the pressure field from the velocity field.  Consequently, the stability of the space-time fully-discretized numerical schemes based on these bases are significantly enhanced. These schemes exhibit unconditional stability for $k=1,2$, and demonstrate exceptional stability and accuracy for $k=3,4$, verified with extensive numerical results for long time simulations using large time step sizes.  Moreover, these schemes can be solved efficiently as they require only the solutions of two symmetric positive-definite equations, instead of solving a coupled saddle-point system. Furthermore, we present efficient solution algorithms for these two decoupled equations for the velocity and magnetic fields, respectively, by exploiting the sparsity and structure of the resultant linear algebraic systems. Ample numerical examples in two and three dimensions are provided to demonstrate  the distinctive accuracy, efficiency and stability of our proposed method. 
\end{abstract}
\maketitle

\section{Introduction}
We are devoted to an exact divergence-free spectral approximation of the incompressible and resistive magnetohydrodynamic (MHD) system governing the interaction between the electrically conducting fluid and the magnetic field:
\begin{subequations}\label{eq: mhd}
\begin{align}
&\partial_t \bs u+\bs u\cdot \nabla \bs u -\nu \nabla^2 \bs u+{\rho}^{-1}\nabla p -(\mu \rho)^{-1}(\nabla \times \bs B) \times \bs B= {\bs f}, \label{eq: umom}\\
& \partial_t \bs B-\nabla \times (\bs u\times \bs B)+\eta \nabla \times(\nabla \times \bs B)=\bs 0, \label{eq: Bmom}\\
& \nabla \cdot \bs u=0, \label{eq: divu}\\
& \nabla \cdot \bs B=0, \label{eq: divB}
\end{align}
\end{subequations}
where $\bs u$, $p$ and $\bs B$ and $\bs f$ are the velocity, pressure, magnetic and certain body force fields, respectively. $\rho$ is the density of the fluid, $\nu$ is the kinematic viscosity, $\mu$ is the magnetic permeability, $\eta={1}/{(\mu \sigma)}$ is the magnetic diffusivity coefficient and $\sigma$ is the electric conductivity, which are all assumed to be constants throughout the paper. This system is prescribed on domain $\Omega \subset \mathbb{R}^d,$ $d\in \{ 2, 3\}$ with boundary $\partial \Omega$ and supplemented with the following
initial and boundary conditions
\begin{equation}\label{eq: initial}
\bs u(\bs x,t=0)=\bs u_{\rm in}(\bs x),\quad  \bs B(\bs x,t=0)=\bs B_{\rm in}(\bs x)\;\; {\rm in}\;\; \Omega,
\end{equation}
and
\begin{equation}\label{eq: mhdbc}
\bs u|_{\partial \Omega}=\bs 0,\quad \bs n \cdot \bs B|_{\partial \Omega}=0,\quad \bs n \times (\nabla \times \bs B)|_{\partial \Omega}=\bs 0\;\;\;{\rm at}\;\;\partial \Omega,
\end{equation}
where $\bs n$ is the outward unit normal vector along $\partial \Omega$. The initial data needs to satisfy the compatible conditions for the well-posedness of the problem
\begin{equation}\label{eq: compatible}
\nabla \cdot \bs u_{\rm in}=\nabla \cdot \bs B_{\rm in}=0, \quad \bs u_{\rm in}|_{\partial \Omega}=\bs 0,\quad \bs n \cdot \bs B_{\rm in}|_{\partial \Omega}=0, \quad \bs n \times \nabla \times \bs B_{\rm in}|_{\partial \Omega}=\bs 0\;\; {\rm at}\;\; \partial \Omega.
\end{equation}

The challenges of numerically simulating the above problem arise from at least three-folds: (i) the divergence-free conditions (i.e. $\nabla \cdot \bs u=\nabla \cdot \bs B=0$); (ii) the saddle-point nature of the system caused by the Lagrange multiplier $p$ to explicitly impose the constraint $\nabla \cdot \bs u=0$; (iii) the nonlinear coupling of the system. On the one hand, one needs to design structure-preserving numerical discretization such that the discrete divergence-free constraints are preserved point-wisely and lead to efficient and accurate solution algorithm. On the other hand, it is crucial to adopt stable scheme suitable for long-time simulations while keeping the resultant discrete linear system easy to solve at each time step. In this paper, we shall address all these issues. 

It is noteworthy that though the divergence-free constraints for the velocity and magnetic fields appear the same at the first glance, they possess different physical meanings and fulfill different roles in the constrained PDE system.  Consequently, they have to be approximated differently in numerical algorithms.  $\nabla \cdot \bs u=0$ encodes the incompressibility of the fluid or, equivalently, the mass conservation, and is imposed explicitly via the Lagrange multiplier $p$ (cf. \cite{white2011fluid}). It is well-known that a pair of approximation space for the velocity and pressure field that does not satisfy the  discrete inf-sup condition will lead to severe numerical instability \cite{boland1983stability}. While the magnetic Gauss's law $\nabla \cdot \bs B=0$ represents the non-existence of magnetic monopoles in nature and is implicitly implied by equation \eqref{eq: Bmom} given that $\nabla \cdot \bs B^0=0$. As is shown in \cite{brackbill1980effect}, the violation of $\nabla \cdot \bs B=0$ induces an unphysical parallel force such that even a tiny deviation in satisfying this constraint will lead to large numerical error in long-time simulation of the MHD system. 

Much effort has been devoted to developing numerical algorithms that preserve these divergence-free constraints for the velocity and magnetic fields at the discrete level.  The common methods used for ensuring divergence-free solutions include, but are not limited to, the divergence-cleaning technique \cite{brackbill1980effect}, the scalar or vector potential method \cite{cyr2013new, jardin2010computational, shadid2010towards}, the constrained transport method \cite{evans1988simulation,stone1992zeus}, the Lagrange multiplier method \cite{Monk03} and the mixed finite element method \cite{gawlik2022finite, hiptmair2018fully, hu2017stable, hu2019structure,LAAKMANN2023112410}. These methods often involve solving additional equations to project the numerical solution onto a divergence-free space or introducing new variables to augment or reformulate the original system. As a result, equations with higher spatial order or large coupled systems need to be solved.
In contrast to the aforementioned methods, the method of constructing divergence-free bases offers a straightforward approach to preserving the divergence-free constraint without any additional effort. 
    In \cite{ye1997discrete}, a weakly divergence-free finite element basis consisting of 9 velocity nodes per element was constructed for the Stokes equation. Guo and Jiao (\cite{guo2016spectral}) proposed an abstract framework for constructing divergence-free basis functions for $n$-dimensional Navier-Stokes equation. Cai et. al in \cite{cai2013divergence} constructed a divergence-free $\bs H({\rm div})$ hierarchical basis for the MHD system. Local divergence-free methods based on discontinuous-Garlerkin (DG) or weak-Galerkin (WG) frameworks were also proposed in \cite{hiptmair2018fully, li2005locally, mu2018discrete}. To the best of the authors' knowledge, there is no exact divergence-free spectral approximation in existing literatures that can preserve the incompressibility condition and  magnetic Gauss's law point-wisely at the same time.   

In this paper, we propose a systematic method adhering to de Rham complexes to construct exact divergence-free $\bs H^1$- and $\bs H({\rm div})$-conforming spectral basis functions on domains diffeomorphic to squares or cubes. This is accomplished through the orthogonal property and derivative relation of the generalized Jacobi polynomials and the property of contravariant Piola transformation. Due to the orthogonality of the generalized Jacobi polynomials, the proposed bases possess a notable merit that the coefficient matrices of the resultant linear algebraic systems are either highly sparse and band-limited or well-structured. By exploiting the properties of these matrices, we can propose corresponding efficient algorithms. 

With the help of these bases, the stability of the space-time fully-discretized numerical schemes are significantly enhanced.
We propose a family of exact divergence-free BDF-$k$ fully-discretized schemes, i.e. DF-BDF-$k$ schemes for the MHD system. These schemes exhibit unconditional stability for $k=1,2$, and demonstrate distinctive feature of stability and accuracy for $k=3,4$, with ample numerical results for long time simulations using large time step sizes. The proposed schemes have drawn inspirations from the energy-stable schemes proposed in \cite{choi2016efficient}. However, a naive extension of the schemes therein to $k\geq 3$ will cause severe stability issue, without the proposed divergence-free bases.  As for efficiency, the proposed methods are linear-coupled schemes and can be solved quite efficiently by a sub-iteration method. In each iteration step, one only needs to solve two decoupled linear equations with constant coefficients for the velocity and magnetic fields, respectively. Numerical experiments show that it often needs only a few sub-iterations (1$\sim$3) to converge for $k=3,4$, even for large time step sizes.

The rest of the paper is organized as follows. In Section \ref{sect: divfree}, we present a systematic way to construct exact divergence-free $\bs H^1$-and $\bs H({\rm div})$-conforming spectral bases in two and three dimensions. The associated efficient solution algorithms are presented. In Section \ref{sect: scheme}, we propose a family of exact divergence-free implicit-explicit BDF-$k$ fully-discretized schemes for the MHD system. The unconditional stability of the scheme for $k=1,2$ is rigorously proved. In Section \ref{sect: numer}, we provide extensive numerical results to demonstrate the accuracy, efficiency and exceptional stability of the proposed methods. Section \ref{sect: conclude} then concludes the discussions with some closing remarks. 

\section{Construction of the divergence-free spectral bases}\label{sect: divfree} 
In this section, we propose spectral  approximations for velocity and magnetic fields such that divergence-free constraints are held exactly at the discrete level. We restrict our attention to computational domains diffeomorphic to squares or cubes, i.e. there exists a diffeomorphism $\mathcal{F}: {\bs \xi} \in [-1,1]^d\rightarrow  {\bs x}\in \Omega \subset \mathbb{R}^d$ for $d=2,3$. We start by introducing some basic notations of Sobolev spaces and briefly review the definition of generalized Jacobi polynomials, as well as their orthogonality and derivative properties.
\subsection{Notations for some Sobolev spaces}
Let $\bs L^2(\Omega,\mathbb{R}^n)$ be the space of square integral $n$-dimensional vector functions on domain $\Omega$ with the associated inner product $(\cdot, \cdot)_{\Omega}$  and norm $\|\cdot \|_{0,\Omega}$ as usual. $\bs H^m(\Omega, \mathbb{R}^n)=\big\{\bs u \in \bs L^2(\Omega,\mathbb{R}^n)| \partial_{x_1}^{\alpha_1}\cdots\partial _{x_d}^{\alpha_d} (\bs u)_j \in  L^2(\Omega, \mathbb{R}), \;  \sum_{i=1}^d|\alpha_i|\leq m, j=1,\cdots,n  \big\}$ and its associated norm $\| \cdot \|_{m,\Omega}$ are defined as in \cite{Adam75}. For clarity, we denote $L^2(\Omega,\mathbb{R})=\bs L^2(\Omega,\mathbb{R})$ and $H^m(\Omega,\mathbb{R})=\bs H^m(\Omega,\mathbb{R})$ for scalar functions. We omit the reference to $\mathbb{R}^n$ and $\Omega$ if no ambiguity occurs. 

In the two-dimensional Cartesian coordinates,  for any vector function $\bs v(\bs x)=(v_1(\bs x),v_2(\bs x))^{\intercal}$ and scalar function $v(\bs x)$,  the curl operator for vector and scalar functions, and the divergence operator are defined respectively as
\begin{equation}\label{eq: curldiv2d}
\nabla \times \bs v=\frac{\partial v_2}{\partial x_1}- \frac{\partial v_1}{\partial x_2},\qquad \nabla \times v=\Big( \frac{\partial v}{\partial x_2}, -\frac{\partial v}{\partial x_1}  \Big)^{\intercal},\qquad \nabla \cdot \bs v=\frac{\partial v_1}{\partial x_1}+ \frac{\partial v_2}{\partial x_2}.
\end{equation}
While for the three-dimensional Cartesian coordinates, the curl and divergence operators for an arbitrary vector $\bs v(\bs x)=\big(v_1(\bs x),v_2(\bs x),v_3(\bs x) \big)^{\intercal}$ take the form
\begin{equation}\label{eq: curl3d}
\nabla \times \bs v=\Big( \frac{\partial v_3}{\partial x_2}- \frac{\partial v_2}{\partial x_3}, \; \frac{\partial v_1}{\partial x_3}-\frac{\partial v_3}{\partial x_1},\; \frac{\partial v_2}{\partial x_1}- \frac{\partial v_1}{\partial x_2} \Big)^{\intercal},\qquad \nabla \cdot \bs v=\frac{\partial v_1}{\partial x_1}+\frac{\partial v_2}{\partial x_2}+\frac{\partial v_3}{\partial x_3}.
\end{equation}
Accordingly, we denote the spaces  
\begin{equation}\label{eq:Hcurl}
\begin{aligned}
 \bs{H}({\rm curl},\Omega)&=
\begin{cases}
 \Big\{\bs v \in \bs L^2(\Omega,\mathbb{R}^2)\big| \nabla \times \bs v\in L^2(\Omega),\;\; \Omega \in \mathbb{R}^2 \Big \},\\[5pt]
 \Big\{\bs v \in \bs L^2(\Omega,\mathbb{R}^3)\big| \nabla \times \bs v\in \bs L^2(\Omega,\mathbb{R}^3),\;\; \Omega \in \mathbb{R}^3 \Big \},
\end{cases} \\
\bs{H}({\rm div},\Omega)&=\Big\{\bs v \in \bs L^2(\Omega)\big|\, \nabla \cdot \bs v\in L^2(\Omega)  \Big \}.
\end{aligned}
\end{equation}
The corresponding Sobolev spaces indowed with essential boundary conditions are defined as $ \bs H^1_0(\Omega)=\big\{ \bs v\in \bs H^1(\Omega),\; \bs v|_{\partial \Omega}=\bs 0 \big\}$ and
\begin{equation}\label{eq:Hcurl0}
\bs H_0({\rm curl},\Omega)=\big\{\bs v\in \bs H({\rm curl},\Omega),\; \bs n \times \bs v |_{\partial \Omega}=0 \big\},\;\;  \bs H_0({\rm div},\Omega)=\big\{\bs v\in \bs H({\rm div},\Omega),\;\; \bs n \cdot \bs v |_{\partial \Omega}=0 \big\}.
\end{equation}
Besides, we denote following Sobolev spaces  with divergence-free constraints
\begin{equation}\label{eq:div0}
\bs H^1_0({\rm div}0,\Omega)=\big\{ \bs v\in \bs H_0^1(\Omega),\;\; \nabla \cdot \bs v=0 \big\},\quad \bs H_0({\rm div}0,\Omega): =\big \{ \bs v \in \bs H_0({\rm div},\Omega), \;\;\nabla \cdot \bs v=0 \big \}.
\end{equation}
For clarity, we sometimes denote $D_{\bs \xi}$, $D_{\bs x}$ ($D=\nabla \times$ or $\nabla \cdot$ ) as differential operators on the reference and physical domains, respectively. 

\subsection{Generalised Jacobi polynomials} 
The classical Jacobi polynomials $P_{n}^{(\alpha,\beta)}(\xi)$ of order $n$ with hyper parameters $\alpha,\beta>-1$ are orthogonal polynomials  with respect to the weight function $\omega^{\alpha, \beta}(\xi)=(1-\xi)^{\alpha}(1+\xi)^{\beta}$ on $\Lambda=(-1,1)$, that is
\begin{equation}\label{GUPeq9}
\int_{-1}^1P_n^{(\alpha,\beta)}(\xi)P_m^{(\alpha,\beta)}(\xi)\omega^{\alpha, \beta}(\xi)\,{\rm d}\xi=\frac{2^{\alpha+\beta+1}\Gamma(n+\alpha+1)\Gamma(n+\beta+1)}{(2n+\alpha+\beta+1)\Gamma(n+1)\Gamma(n+\alpha+\beta+1)}\delta_{nm},
\end{equation}
where $\delta_{nm}$ is the Kronecker delta symbol and $\Gamma(\cdot)$ is the Gamma function. 
The classical Jacobi polynomials can be generalized to cases with general $\alpha,\beta\in\mathbb{R}$, that is, the generalised Jacobi functions (cf.\,\cite{GUO20091011, ShenTangWang2011}). The flexibility offered by the parameters $\alpha,\beta$ allows us to design suitable bases for exact preservation of the divergence-free constraint.  Hereafter, we adopt the following definition of generalized Jacobi polynomials of parameters $\alpha=\beta=-1,-2$:
\begin{equation}\label{eq: gJp}
P_n^{(-1,-1)}(\xi)=
\begin{cases}
(1-\xi)/2,\;\; & n=0,\\[2pt]
(1+\xi)/2,\;\; & n=1,\\[2pt]
\dfrac{\xi^2-1}{4} P_{n-2}^{(1,1)}(\xi), \;\; & n \geq 2;
\end{cases}
\quad
P_n^{(-2,-2)}(\xi)=
\begin{cases}
 {(1-\xi)^2(2+\xi)}/{4},\;\; & n=0,\\[2pt]
 {(1-\xi)^2(1+\xi)}/{4},\;\; & n=1,\\[2pt]
 {(1+\xi)^2 (2-\xi)}/{4}, \;\;& n=2,\\[2pt]
 {(1+\xi)^2(\xi-1)}/{4},\;\; & n=3, \\[2pt]
 \Big(\dfrac{\xi^2-1}{4}\Big)^2P_{n-4}^{(2,2)}(\xi),\;\;&  n \geq 4.
\end{cases}
\end{equation}
The following derivative relation and orthogonality of the generalised Jacobi polynomials play crucial roles for the construction of the desired bases: 

\begin{itemize}
\item derivative relations:
\begin{equation}\label{eq: derivativep}
 P_n^{(-1,-1)'}(\xi)=
 \dfrac{n-1}{2}P_{n-1}^{(0,0)}(\xi) ,\;\;  n\geq 2,
 \quad
P_n^{(-2,-2)'}(\xi)=
 \dfrac{n-3}{2}P^{(-1,-1)}_{n-1}(\xi),\;\;  n\geq 4,
\end{equation}
where $P_n^{(0,0)}(\xi)=L_n(\xi)$ is the Legendre polynomial of order $n$;
\item orthogonality:
\begin{subequations}\label{eq: orth}
\begin{align}
& \int_{-1}^1 L_{n}(\xi) L_{m}(\xi) {\rm d}\xi=\frac{2}{2n+1}\delta_{mn},\quad m,n\geq 0, \label{eq: L1orth}\\
& \int_{-1}^1 P_{n}^{(-1,-1)}(\xi) P_{m}^{(-1,-1)}(\xi) (1-\xi^2)^{-1}{\rm d}\xi=\frac{n-1}{2n(2n-1)}\delta_{mn}, \quad m,n\geq 2, \label{eq: J1orth}\\
& \int_{-1}^1 P_{n}^{(-2,-2)}(\xi) P_{m}^{(-2,-2)}(\xi)(1-\xi^2)^{-2} {\rm d}\xi= \frac{(n-2)(n-3)}{8n(n-1)(2n-3)}\delta_{mn},\quad m,n\geq 4. \label{eq: J2orth}
\end{align}
\end{subequations}
\end{itemize}

\subsection{Divergence-free  bases in two dimensions}
Now we are ready to present a systematic procedure to construct high-order exact divergence-free spectral bases for $\bs H_0^1(\Omega)$ and $\bs H_0({\rm div},\Omega)$  spaces in two dimensions. Besides the derivative relations and orthogonal properties of generalised Jacobi polynomials, another crucial tools for these constructions are several de Rham complexes and their associated commuting diagrams. 

\subsubsection{$\bs H_0({\rm div},\Omega)$-conforming divergence-free basis in 2D}
To fix the idea, we start with the $\bs H_0({\rm div})$-conforming divergence-free basis on a reference square $\Lambda^2$. 
Consider the two-dimensional de Rham complex with minimal $L^2$-smoothness (see \cite{arnold2010finite, Monk03}) and its corresponding commuting diagram
\begin{equation}\label{eq: Bcomp2d}
\begin{split}
&\mathbb{R}\xrightarrow{\rm id}H^1(\Omega)\;\;\;\xrightarrow{\nabla\times  }\;\;\; \bs H({\rm div},\Omega)\;\;\;\xrightarrow{\nabla \cdot} \;\;\; L^2(\Omega)\xrightarrow{0} \{ 0\}\\
& \qquad \;\; \mathcal{P}_1 \downarrow \qquad \;\;\;\;\; \;\;\;\;\;\;\;\mathcal{P}_{\rm div} \downarrow \qquad \;\;\;\;\;\;\;\;\; \;\;\mathcal{P}_{0}\downarrow \\
&\mathbb{R}\xrightarrow{\rm id}H_N^1(\Omega)\;\;\xrightarrow{\nabla\times  }\;\; \bs H_N({\rm div},\Omega)\;\; \xrightarrow{\nabla \cdot} \;\;L_{N-1}^2(\Omega)\xrightarrow{0} \{ 0\}.
\end{split}
\end{equation}
The finite dimensional spaces $H_N^1$, $\bs H_N({\rm div})$ and $L_{N-1}^2$ form discrete approximate spaces for $H^1$, $\bs H({\rm div})$ and $L^2$ spaces, respectively and $\mathcal{P}_1$, $\mathcal{P}_{\rm div}$ and $\mathcal{P}_{0}$ are the corresponding projections onto the discrete spaces. The continuous and discrete complexes respectively form exact sequences,  implying that for each operator, the range of the operator coincides with the null space of the successive operator in the sequence. 

Define 
\begin{equation}\label{eq: Wn2d}
H_{N}^1(\Lambda^2):={\rm span}\big( \big\{P_{m}^{(-1,-1)}(\xi_1)P_{n}^{(-1,-1)}(\xi_2),\;\; 0\leq m,n\leq N  \big \} \big).
\end{equation}
We note from \cite{Shen94b} that $H_{N}^1(\Lambda^2)$
serves as a conforming basis for $H^1(\Lambda^2)$. Inspired by the exact sequence \eqref{eq: Bcomp2d}, one 
can construct the following approximation spaces $ \bs H_N({\rm div},\Lambda^2)$ for $\bs H({\rm div},\Lambda^2)$ as follows
\begin{equation}\label{eq: Yn2d}
\begin{aligned}
\bs H_N({\rm div},\Lambda^2)={\rm span}\Big(& 
 \Big\{P_{m}^{(-1,-1)}(\xi_1)L_{n}(\xi_2)\bs e_1,\;\;0\leq m\leq N,\; 0\leq n \leq N-1 \Big\},\\
 &\Big\{ L_{m}(\xi_1)P_{n}^{(-1,-1)}(\xi_2)\bs e_2, \;\; 0\leq m \leq N-1, \; 0\leq n \leq N \Big\} \Big),
 \end{aligned}
\end{equation} 
where $\bs e_i,\, i=1,2$ are the canonical bases of two-dimensional Cartesian coordinates. 
It is direct to verify that
\begin{equation}\label{eq: exHdiv2d}
 \nabla \times H_N^1(\Lambda^2)={\rm ker}(\nabla \cdot):=\big\{\bs v\in \bs H_N({\rm div},\Lambda^2) |\nabla \cdot  \bs v=0 \big\},
\end{equation}
which mimics the continuous relation \eqref{eq: Bcomp2d} at the discrete level.  

Note that $\bs H_N({\rm div},\Lambda^2)$ can be categorized into Raviart-Thomas $Q_{k,k-1}\times Q_{k-1,k}$ approximation basis for $\bs H({\rm div},\Lambda^2)$ (see \cite[p.~97]{boffi2013mixed}), thus  any function $\bs v\in \bs H_0({\rm div},\Lambda^2)\cap \bs H({\rm div0},\Lambda^2)$ can be well-approximated by $\bs v_N\in \bs H_N({\rm div},\Lambda^2)$ as
\begin{equation}\label{eq: v2d}
\begin{aligned}
\bs v_N(\bs \xi)=& \sum_{m,n=1}^{N-1} \Big( \hat v_{m,n}^1P_{m+1}^{(-1,-1)}(\xi_1)L_{n}(\xi_2)\bs e_1+ \hat v_{m,n}^2  L_{m}(\xi_1)P_{n+1}^{(-1,-1)}(\xi_2)\bs e_2 \Big)\\
                   & + \sum_{m=1}^{N-1} \hat v_{e,m}^1 P_{m+1}^{(-1,-1)}(\xi_1)\bs e_1 + \sum_{n=1}^{N-1} \hat v_{e,n}^{2} P_{n+1}^{(-1,-1)}(\xi_2)\bs e_2,
\end{aligned}
\end{equation}
where the terms involving $P_0^{(-1,-1)}$ and $P_1^{(-1,-1)}$ vanish due to the boundary condition $\bs n \cdot \bs v|_{\partial \Lambda^2}=0.$
With the help of the derivative property in equation \eqref{eq: derivativep}, it is straightforward to evaluate that
\begin{equation}
 \nabla \cdot \bs v_N(\bs \xi)=\sum_{m,n=1}^{N-1}  \frac{1}{2}\big( m\hat v_{m,n}^1 +n \hat v_{m,n}^2  \big)L_{m}(\xi_1)L_{n}(
\xi_2)+\sum_{m=1}^{N-1} \frac{m}{2} \hat v_{e,m}^1 L_m(\xi_1)+\sum_{n=1}^{N-1} \frac{n}{2} \hat v_{e,n}^2 L_n(\xi_2). 
\end{equation}
Thus, by the orthogonality of Legendre polynomials under the $L^2$ inner product, one obtains a sequence of linear equations from $\nabla \cdot \bs u_N=0$
\begin{equation}\label{eq: lineqB}
m\hat v^1_{m,n}+n\hat v_{m,n}^2=0, \;\; 1\leq m,n\leq N-1, \quad \hat v_{e,i}^1=\hat v_{e,i}^2=0,\;\; i=1,\cdots, N-1.
\end{equation}
Define 
\begin{equation}\label{eq: psiphi}
\begin{aligned}
&\psi_{m}(\xi)=\frac{\sqrt{2(2m-1)}}{m-1}P_{m}^{(-1,-1)}(\xi),\;\; m\geq 2, \quad \phi_n(\xi)=\sqrt{\frac{2n+1}{2}}L_n(\xi),\;\;n\geq 0.
\end{aligned}
\end{equation}
By combining equations \eqref{eq: v2d} and \eqref{eq: lineqB} and adopting the notations in equation \eqref{eq: psiphi}, it directly leads to the divergence-free spectral basis for $\bs H_0({\rm div}0;\Lambda^2)$ space in the following proposition. 
\begin{prop}
${\bs H}_{N,0}({\rm div0}; \Lambda^2)$ is a conforming divergence-free approximation space for $\bs H_0({\rm div}0;\Lambda^2)$ taking the form
\begin{equation}\label{eq: div0b2d}
{\bs H}_{N,0}({\rm div0}; \Lambda^2)={\rm span} \big \{\bs \Phi_{m,n}(\bs \xi),\; 1\leq m,n\leq N-1\big \},
\end{equation}
where $ {\bs \Phi}_{m,n}(\bs \xi)=\big( \psi_{m+1}(\xi_1)\phi_n(\xi_2),-\phi_m(\xi_1)\psi_{n+1}(\xi_2)    \big)^{\intercal}$. 
\end{prop}
\begin{proof}
With the help of  the derivative property \eqref{eq: derivativep}, one readily verifies that
\begin{equation}\label{eq:psirelaphi}
\psi_{m+1}'(\xi)=\phi_m(\xi),
\end{equation}
from where we obtain $\nabla_{\bs \xi} \cdot \bs\Phi_{m,n}(\bs \xi)=0.$ 
\end{proof}
In order to extend the above exact divergence-free basis from the reference square to physical domain $\Omega$ that is diffeomorphic to  $\Lambda^2$, one needs to resort to the following property of contravariant Piola transformation (see \cite[p.39]{ciarlet1988three}).
\begin{lm}
Let $\mathcal{F}$ be a diffeomorphism from $\Lambda^d$ onto $\mathcal{F}(\Lambda^d)=\Omega \subset \mathbb{R}^d$, the contravariant Piola transformation is given by
\begin{equation}\label{eq: piola}
\mathcal{F}^{\rm div}[\bs \Phi](\bs x)=\frac{{\partial_{\bs \xi} \bs x}}{{\rm det}(\partial_{\bs \xi} \bs x)}  \bs \Phi \circ \mathcal{F}^{-1}(\bs x),\;\; \bs x=\mathcal{F}(\bs \xi),
\end{equation}
which satisfies the identity
\begin{equation}\label{eq: piolaid}
{\rm det}(\partial_{\bs \xi} \bs x)   \nabla_{\bs x} \cdot \mathcal{F}^{\rm div}[\bs \Phi]=\nabla_{\bs \xi}\cdot  \bs \Phi .
\end{equation}
\end{lm}
As a direct consequence, we arrive at the spectral approximation basis for $\bs H_0({\rm div}0;\Omega)$-conforming basis as follows.
\begin{prop}
${\bs H}_{N,0}({\rm div0};\Omega)$ is a conforming divergence-free approximation space for $\bs H_0({\rm div}0;\Omega)$ taking the form
\begin{equation}\label{eq: div0b2d}
{\bs H}_{N,0}({\rm div0}; \Omega)={\rm span} \big \{\tilde{ \bs \Phi}_{m,n}(\bs x),\; 1\leq m,n\leq N-1\big \},
\end{equation}
where $ \tilde{\bs \Phi}_{m,n}(\bs x)=\mathcal{F}^{\rm div}[{\bs \Phi}_{m,n}](\bs x)$, the basis function ${\bs \Phi}_{m,n}$ and the Piola mapping  $\mathcal{F}^{\rm div}$   are given in equations \eqref{eq: psiphi} and \eqref{eq: piola}, respectively. 
\end{prop}
\begin{proof}
It is direct to verify that  $\nabla_{\bs x} \cdot  \tilde{\bs \Phi}_{m,n}(\bs x)=0 $ by using the fact that $\nabla_{\bs \xi}\cdot {\bs \Phi}_{m,n}=0$ and identity in equation \eqref{eq: piolaid}.
\end{proof}


\subsubsection{$\bs H_0^1(\Omega)$-conforming divergence-free basis in 2D}\label{sec: ho1div2d}

Next, we turn to the construction of suitable $\bs H^1$-conforming divergence-free basis.  Let us resort to the two-dimensional smoothed de Rham complex, known as the 2D ``Stokes complex'' \cite{falkstokes2013, guzman2014conforming, mardal2002robust}:
\begin{equation}\label{eq: stokescom2d}
\mathbb{R}\xrightarrow{\rm id}H^2(\Omega)\xrightarrow{\nabla\times  }  \bs H^1(\Omega)\xrightarrow{\nabla \cdot} L^2(\Omega) \xrightarrow{0} \{ 0\}.
\end{equation}
This complex is exact provided that $\Omega$ is simply connected \cite{girault2012finite}.  One can start from the conforming basis for $H^2(\Lambda^2)$: 
\begin{equation}\label{eq:h2N}
H^2_N(\Lambda^2)={\rm span}\big( \big\{P_m^{(-2,-2)}(x_1)P_n^{(-2,-2)}(x_2),\;\;0\leq m,n\leq N  \big\}  \big),
\end{equation}
and the derivation is quite similar with $\bs H_0({\rm div},\Omega)$-conforming divergence-free basis in the previous section. For the sake of concise explanation, we omit the detailed derivation and summarize the corresponding basis as follows.

\begin{prop}\label{prop: h12d}
Define for $m,n\geq 1,$
	\begin{equation}\label{eq: chi}
	 \bs \chi_{m,n}(\bs \xi)=\big(  \varphi_{m+3}(\xi_1)\psi_{n+2}(\xi_2)    ,-\psi_{m+2}(\xi_1)\varphi_{n+3}(\xi_2)    \big)^{\intercal},\qquad \tilde{ \bs \chi}_{m,n}(\bs x)=\mathcal{F}^{\rm div}[{\bs \chi}_{m,n}](\bs x),
	\end{equation}
where $\mathcal{F}^{\rm div}$ is the contravariant Piola transformation and $\varphi_{m}(\xi)$ is given by
	\begin{equation}\label{eq:varphipsi}
	\varphi_{m}(\xi) = \frac{\sqrt{8(2m-3)}}{(m-3)(m-2)} P_{m}^{(-2,-2)}(\xi), \;\; m \geq 4.
	\end{equation}
The conforming divergence-free approximation space for $\bs H^1_0({\rm div}0;\Omega)$ takes the form
\begin{equation}\label{eq: div0b2d}
{\bs H}^1_{N,0}({\rm div0},\Omega)={\rm span} \big \{\tilde{ \bs \chi}_{m,n}(\bs x),\; 1\leq m,n\leq N-3\big \}.
\end{equation}
\end{prop}
\begin{proof}
From the derivative relation of generalised Jacobi polynomials in \eqref{eq: derivativep}, we can obtain that
$\varphi_{m+1}'(\xi)=\psi_m(\xi)$. Thus, it is straightforward to derive that $\nabla_{\bs \xi}\cdot {\bs \chi}_{m,n}=0,$ which together with the property of the Piola transformation \eqref{eq: piolaid} leads to  $\nabla_{\bs x} \cdot  \tilde{\bs \chi}_{m,n}(\bs x)=0 $.
\end{proof}


\subsection{Divergence-free  bases in three dimensions}
Following the same spirit of the two dimensional case in the previous subsections, we can obtain the three dimensional  $\bs H_0({\rm div},\Omega)$- and $\bs H_0^1(\Omega)$-conforming divergence-free  bases, which are accomplished by realizing the three-dimensional de Rham complex with $L^2$-smoothness \cite{Monk03}
\begin{equation}\label{eq: 3dcompl2}
\mathbb{R}\xrightarrow{\rm id}H^1(\Omega)\xrightarrow{\nabla } \bs H({\rm curl},\Omega)\xrightarrow{\nabla \times} \bs H({\rm div},\Omega) \xrightarrow{\nabla \cdot} L^2(\Omega)\xrightarrow{0} \{ 0\}
\end{equation}
and  the 3D Stokes complex \cite{neilan2015discrete, tai2006discrete}
\begin{equation}\label{eq: stokescom3d}
\mathbb{R}\xrightarrow{\rm id}H^2(\Omega)\xrightarrow{\nabla  } \bs H^1({\rm curl},\Omega)\xrightarrow{\nabla \times} \bs H^1(\Omega)\xrightarrow{\nabla \cdot }L^2(\Omega)\xrightarrow{0} \{ 0\}.
\end{equation}
The derivations are much involved, especially for the $\bs H_0^1$-conforming divergence-free  bases in 3D.  For conciseness, we postpone them  in appendices and state the results in the following two propositions.
\begin{prop}\label{prop: p3}
The conforming divergence-free approximation space for $\bs H_0({\rm div}0;\Omega)$ in three dimensions takes the form
\begin{equation}\label{eq: div0b3d}
{\bs H}_{N,0}({\rm div0}; \Omega)= {\rm span}\Big\{ \tilde{\bs\Phi}_{m,n,l}^1(\bs x),\tilde{\bs\Phi}_{m,n,l}^2(\bs x), \tilde{\bs \Phi}^{x}_{n,l}(\bs x), \tilde{\bs\Phi}^{y}_{m,l}(\bs x),\tilde{\bs\Phi}^{z}_{m,n}(\bs x) \Big\}_{m,n,l=1}^{N-1},
\end{equation}
where $\tilde{ \bs \Phi}(\bs x)=\mathcal{F}^{\rm div}[{\bs \Phi}](\bs x)$ for $ \bs \Phi=\bs \Phi_{m,n,l}^1,\bs \Phi_{m,n,l}^2,  \bs \Phi^{x}_{n,l}, \bs \Phi^{y}_{m,l},\bs \Phi^{z}_{m,n}$, and $\mathcal{F}^{\rm div}$ is the contravariant Piola transformation defined in equation \eqref{eq: piola}. The basis functions on  the reference cube $\Lambda^3$  for the interior modes take the form
\begin{equation}\label{eq:divBb1}
\begin{aligned}
&\bs \Phi_{m,n,l}^1(\bs \xi)=\Big( \psi_{m+1}(\xi_1)\phi_{n}(\xi_2)\phi_{l}(\xi_3), \;\; -\phi_m(\xi_1)\psi_{n+1}(\xi_2)\phi_{l}(\xi_3),\;\;0\Big)^{\intercal},\\
&\bs \Phi_{m,n,l}^2(\bs \xi)=\Big( \psi_{m+1}(\xi_1)\phi_{n}(\xi_2)\phi_{l}(\xi_3),\;\; 0,\;\;   -\phi_{m}(\xi_1)\phi_{n}(\xi_2)\psi_{l+1}(\xi_3)\Big)^{\intercal},
\end{aligned}
\end{equation}
and the basis functions for the face modes, which are of type $x$, $y$, and $z$, respectively, take the form
\begin{align}
&  \bs \Phi^{x}_{n,l}(\bs \xi)=\big( 0,\; \; \psi_{n+1}(\xi_2) \phi_l(\xi_3),\;\; -\phi_n(\xi_2) \psi_{l+1}(\xi_3)  \big)^{\intercal},  \label{eq:hdive1}\\
&  \bs \Phi^{y}_{m,l}(\bs \xi)=\big(  \psi_{m+1}(\xi_1) \phi_l(\xi_3),\;\;0,\;\; -\phi_m(\xi_1) \psi_{l+1}(\xi_3)  \big)^{\intercal}, \label{eq:hdive2}\\
&  \bs \Phi^{z}_{m,n}(\bs \xi)=\big(  \psi_{m+1}(\xi_1) \phi_n(\xi_2),\;\; -\phi_m(\xi_1) \psi_{n+1}(\xi_2) ,\;\;0 \big)^{\intercal},\label{eq:hdive3}
\end{align}
where $\{ \psi_m\}_{m=2}^{N}$ and $\{\phi_m\}_{m=1}^{N-1}$ are defined in equation \eqref{eq: psiphi}. 
\end{prop}
\begin{proof}
See detailed derivations in Appendix \ref{ap: b}.
\end{proof}


\begin{prop}\label{prop: prop4}
The conforming divergence-free approximation space for $\bs H^1_0({\rm div}0;\Omega)$ in three dimensions reads
\begin{equation}\label{eq: div0u3d}
{\bs H}^1_{N,0}({\rm div0},\Omega)= {\rm span}\Big\{ \tilde{\bs\chi}_{m,n,l}^1(\bs x),\,\tilde{\bs\chi}_{m,n,l}^2(\bs x),\, \tilde{\bs \chi}^{x}_{n,l}(\bs x),\, \tilde{\bs\chi}^{y}_{m,l}(\bs x),\,\tilde{\bs\chi}^{z}_{m,n}(\bs x) \Big\}_{m,n,l=1}^{N-3},
\end{equation}
where $\tilde{ \bs \chi}(\bs x)=\mathcal{F}^{\rm div}[{\bs \chi}](\bs x)$ for $ \bs \chi=\bs \chi_{m,n,l}^1,\,\bs \chi_{m,n,l}^2,\,  \bs \chi^{x}_{n,l},\, \bs \chi^{y}_{m,l},\, \bs \chi^{z}_{m,n}$ and $\mathcal{F}^{\rm div}$ is the contravariant Piola transformation.  Here, $\big\{\bs \chi_{m,n,l}^1(\bs \xi),\bs \chi_{m,n,l}^2(\bs \xi) \big\}$ stand for the interior modal basis functions defined by 
	\begin{equation}
	\label{eq:div0u_in}
	\begin{aligned}
	& \bs \chi_{m,n,l}^1(\bs \xi)
	=\big( \varphi_{m+3}(\xi_1)\psi_{n+2}(\xi_2)\psi_{l+2}(\xi_3),\;\;  -\psi_{m+2}(\xi_1)\varphi_{n+3}(\xi_2)\psi_{l+2}(\xi_3),\;\; 0\big)^{\intercal},
	\\
	&\bs \chi_{m,n,l}^2(\bs \xi)
	=\big( \varphi_{m+3}(\xi_1)\psi_{n+2}(\xi_2)\psi_{l+2}(\xi_3),\;\; 0,\;\;   -\psi_{m+2}(\xi_1)\psi_{n+2}(\xi_2)\varphi_{l+3}(\xi_3)\big)^{\intercal},
	\end{aligned}
	\end{equation}
	and $\big\{\bs \chi_{n,l}^x (\bs \xi), \bs \chi_{m,l}^y (\bs \xi), \bs \chi_{m,n}^z (\bs \xi) \big\}$ 
	stand for the face modal basis functions of type $x$, $y$, and $z$, respectively, defined by
	\begin{equation}
	\label{eq:div0u_fa}
	\begin{aligned}
	&  \bs \chi_{n,l}^x(\bs \xi)
	=\big( 0, \; \psi_2(\xi_1)\varphi_{n+3}(\xi_2) \psi_{l+2}(\xi_3),\; - \psi_2(\xi_1)\psi_{n+2}(\xi_2) \varphi_{l+3}(\xi_3)  \big)^{\intercal},  
	\\
	&  \bs \chi_{m,l}^y(\bs \xi)
	=\big(  \varphi_{m+3}(\xi_1) \psi_2(\xi_2) \psi_{l+2}(\xi_3),\;0,\; -\psi_{m+2}(\xi_1) \psi_2(\xi_2) \varphi_{l+3}(\xi_3)  \big)^{\intercal}, 
	\\
	&  \bs \chi_{m,n}^z(\bs \xi)
	=\big(  \varphi_{m+3}(\xi_1) \psi_{n+2}(\xi_2)\psi_2(\xi_3),\; -\psi_{m+2}(\xi_1) \varphi_{n+3}(\xi_2)\psi_2(\xi_3) ,\;0 \big)^{\intercal},
	\end{aligned}
	\end{equation}
	where $\big\{\varphi_m(\xi) \big\}_{m=4}^{N}$ and $\big\{\psi_m(\xi)\big\}_{m=2}^{N-1}$ are given by equation \eqref{eq:varphipsi} and \eqref{eq: psiphi}, respectively.
\end{prop}
\begin{proof}
We postpone the derivations in Appendix \ref{ap:d}.
\end{proof}

\subsection{Efficient solution algorithms}
In order to facilitate the development of efficient solution algorithms for the MHD system \eqref{eq: mhd}, we need to resort to the divergence-free spectral discretizations of two model problems on reference domain $\Omega=\Lambda^{d}$ ($d=2,3$), i.e. the curl-curl problem of the magnetic field 
 \begin{subequations}\label{eq:Bmodelsystem}
\begin{align}
&\nabla\times\nabla\times {\bs B}+\kappa_0 \, {\bs B}={\bs g},\;\;\;\;\;\;\qquad {\rm in}\;\;\Omega, \label{eq: 2dsys1Bmodel}\\
& {\bs n}\cdot \bs B=0,\;\; \bs n\times (\nabla\times {\bs B})=\bs 0,\;\quad {\rm at}\;\;\partial\Omega, \label{eq: 2dsys1bcBmodel}
\end{align}
\end{subequations}
and the Oseen problem of the velocity field
 \begin{subequations}\label{eq:umodelsystem}
\begin{align}
&-\nabla^2 {\bs u}+\kappa_1 \, {\bs u}+\nabla p={\bs f},\;\;\;\; {\rm in}\;\;\Omega, \label{eq: 2dsys1umodel}\\
& \nabla \cdot \bs u=0,  \qquad \qquad \qquad \qquad  {\rm in}\;\;\Omega, \label{eq: 2dsys1udiv0}    \\
& \bs u= \bs 0,\qquad \qquad\qquad \qquad \quad \;\; {\rm at}\;\;\partial\Omega, \label{eq: 2dsys1bcumodel}
\end{align}
\end{subequations}
where $\kappa_0,\kappa_1$ are given constants, and ${\bs g}$ is assumed to be solenoidal. 

Let us employ the $\bs H_0({\rm div})$- and $\bs H_0^1$-conforming divergence-free bases to discretize these two model problems, respectively, and the related approximation schemes take the form
\begin{itemize}
\item find $\bs B_N \in {\bs H}_{N,0}({\rm div0}; \Lambda^d)$ such that
\begin{equation}\label{eq: Bwkbih1}
(\nabla \times \bs B_N, \nabla \times \bs \Phi)+ \kappa_0 (\bs B_N, \bs \Phi)  = (\bs g, \bs \Phi)    ,\;\; \forall \bs \Phi \in {\bs H}_{N,0}({\rm div0}; \Lambda^d), 
\end{equation}
\item and find $\bs u_{N}\in {\bs H}^1_{N,0}({\rm div0},\Lambda^d)$ such that
\begin{equation}\label{eq: VelwkLap1}
(\nabla \bs u_N, \nabla \bs \chi) + \kappa_1 (\bs u_N, \bs \chi)=(\bs f,\bs \chi),\;\; \forall \bs \chi \in {\bs H}^1_{N,0}({\rm div0},\Lambda^d).
\end{equation}
\end{itemize}

The model problem \eqref{eq: Bwkbih1} can be solved efficiently by a matrix diagonalisation algorithm, which has been reported by the authors in a separate paper (see a detailed description of the algorithm in \cite{qin2023highly}). Thus, we focus on the solution algorithm for \eqref{eq: VelwkLap1}. The following proposition is the key to the efficiency of our algorithms. 
\begin{prop}
\label{Prop5-GHKMat}
Denote symmetric matrices $\mathbf{G}=(G_{mn})$, $\mathbf{H}=(H_{mn})$ and $\mathbf{I}=(I_{mn})$, 
where $1\leq{m,n}\leq{N-3}$ and the entries are given by
\begin{equation}
\label{eq:GHKdef}
G_{mn} = \big( \varphi_{n+3}, \varphi_{m+3} \big)_{\Lambda}, \;\; 
H_{mn} =  \big( \psi_{n+2}, \psi_{m+2} \big)_{\Lambda} , \;\;
I_{mn} = \big( \psi'_{n+2},\psi'_{m+2} \big)_{\Lambda}
= \big( \phi_{n+1},\phi_{m+1} \big)_{\Lambda}.
\end{equation}
Then there hold
\begin{align}
& G_{mn} = 
\begin{cases}
\dfrac{1}{2n+3}\Big( \dfrac{1}{(2n-1)(2n+1)^2} + \dfrac{2(2n+3)}{(2n+1)^2(2n+5)^2} + \dfrac{1}{(2n+5)^2(2n+7)}  \Big), \; & m=n,  
\\[6pt]
- \dfrac{1}{\sqrt{2n+3}} \dfrac{1}{\sqrt{2n+7}}\dfrac{2}{(2n+5)^2}\Big( \dfrac{1}{2n+1} + \dfrac{1}{2n+9} \Big) , \; & m=n+2, 
\\[6pt]
\dfrac{1}{\sqrt{2n+3}}\dfrac{1}{\sqrt{2n+11}}\dfrac{1}{2n+5}\dfrac{1}{2n+7}\dfrac{1}{2n+9},
\; & m=n+4,
\\[6pt]
0, \; & \text{otherwise};
\end{cases}
\label{eq:Gmn}
\\
& H_{mn} =
\begin{cases}
\dfrac{1}{2n+3}\big( \dfrac{1}{2n+1}+\dfrac{1}{2n+5}  \big),       \quad & m=n,
\\[6pt]
- \dfrac{1}{\sqrt{2n+3}}\dfrac{1}{\sqrt{2n+7}}\dfrac{1}{2n+5} ,  \quad & m=n+2, 
\\[6pt]
0,\quad & {\rm otherwise};
\end{cases}
\qquad \; 
I_{mn} = \delta_{{m}{n}}.
\label{eq:MmnKmn}
\end{align}
In addition, denote the column vector $\widetilde{\bs H}=(\widetilde{H}_m)$ with $1\leq{m}\leq{N-3}$, given by
\begin{equation}\label{eq: tHm}
\widetilde{H}_m=(\psi_2,\psi_{m+2})_{\Lambda}=
\begin{cases}
 -\dfrac{\sqrt{3}}{15\sqrt{7}}, & m=2,\\
 0, & {\rm otherwise}.
\end{cases}
\end{equation}

\begin{proof}
Denote $c_{n}=\sqrt{2(2n+1)}/2.$ It is known that the generalised Jacobi polynomials with negative integer parameters  can be expressed as compact combination of Legendre polynomials. Indeed, one has (see \cite[p.~203]{ShenTangWang2011})
\begin{align}
& \varphi_{m+3}(\xi)=\frac{1}{4c_{m+1}}\Big(\frac{1}{c_m^2}L_{m-1}(\xi) -2\Big( \frac{c_{m+1}}{c_m c_{m+2}} \Big)^2 L_{m+1}+\frac{1}{c_{m+2}^2}L_{m+3}(\xi)  \Big),\\
& \psi_{m+2}(\xi)=\frac{1}{2c_{m+1}} \Big( L_{m+2}(\xi)-L_m(\xi)  \Big).
\end{align}
Moreover, $\psi_2(\xi)$ can be expressed in terms of Legendre polynomials as $\psi_2(\xi)= \big(L_2(\xi)-L_0(\xi)\big)/\sqrt{6}.$
By the analytic expressions of $\varphi_{m+3},$ $\psi_{m+2}$ and $\psi_2$ and the orthogonality of Legendre polynomials in equation \eqref{eq: L1orth}, we can derive equations \eqref{eq:Gmn}-\eqref{eq: tHm} with a direct calculation.
\end{proof}
\end{prop}

 Let us start with the algorithm for two dimensional case.  Expanding the numerical solution $\bs u_N(\bs x)$ of the approximation scheme \eqref{eq: VelwkLap1} in terms of ${\bs H}^1_{N,0}({\rm div0},\Lambda^2)$ in proposition \ref{prop: h12d} as $\bs u_N=\sum_{m,n=1}^{N-3} u_{mn}\, \bs \chi_{mn},$
and denote
\begin{equation*}\label{eq: umatrix}
f_{mn}=\big(\bs f,\bs \chi_{mn}\big),\quad \mathbf{F}=(f_{mn})_{m,n=1,\cdots, N-3},\quad \mathbf{U}=(u_{mn})_{m,n=1,\cdots, N-3}.
\end{equation*}
One can take $\bs \chi=\bs \chi_{\hat m \hat n}$ in equation \eqref{eq: VelwkLap1} for $\hat m,\hat n=1,\cdots, N-1$,  and aware of the definitions of matrices $\mathbf{G}$, $\mathbf{H}$ and $\mathbf{I}$ in equations \eqref{eq:GHKdef}-\eqref{eq:MmnKmn} to obtain the following matrix equation 
\begin{equation}\label{eq: 2dvelaxb}
2\mathbf{H} \mathbf{U}  \mathbf{H} +   \mathbf{G}\mathbf{U} +   \mathbf{U} \mathbf{G}+   \kappa_1 \big( \mathbf{G}\mathbf{U} \mathbf{H}+\mathbf{H}\mathbf{U} \mathbf{G}  \big)=\mathbf{F}.
\end{equation}
Since $\mathbf{G}$ and $\mathbf{H}$ are respectively  nona-diagonal and penta-diagonal, the coefficient matrix of the above linear algebraic system \eqref{eq: 2dvelaxb}  is highly sparse. Thus, efficient solvers based on factorization \cite{golub2013matrix, trefethen2022numerical} or matrix-free preconditioned conjugate gradient method (cf. \cite{kelley1995iterative, saad2003iterative}) can be adopted to solve this system.

Next, we turn to the solution algorithm for the three dimensional case.  We employ the divergence-free basis functions of ${\bs H}^1_{N,0}({\rm div0},\Lambda^3)$   in proposition \ref{prop: prop4} to expand $\bs u_N$  as
\begin{equation}\label{eq:uN_3d}
\bs u_N = \sum_{m,n,l=1}^{N-3}
\Big\{  u_{mnl}^1 {\bs\chi}_{m,n,l}^1 + u_{mnl}^2 {\bs\chi}_{m,n,l}^2 \Big\} +\sum_{n,l=1}^{N-3}  u_{nl}^x {\bs\chi}_{n,l}^x + \sum_{m,l=1}^{N-3} u_{ml}^y {\bs\chi}_{m,l}^z  + \sum_{m,n=1}^{N-3} u_{mn}^z {\bs\chi}_{m,n}^z.
\end{equation}
and denote the source terms by
\begin{align}
& f_{mnl}^i = \big( \bs{f},{\bs\chi}_{m,n,l}^i \big),\;\; i=1,2,\;\;  f_{nl}^x = \big( \bs{f},{\bs \chi}_{n,l}^x \big),\;\; 
f_{ml}^y = \big( \bs{f}, {\bs\chi}_{m,l}^y \big),\;\; f_{mn}^z = \big( \bs{f},{\bs\chi}_{m,n}^z \big).
\label{eq:Hmatrix3d}
\end{align}
Let us remap the multidimensional variables $\{u^1_{mnl}, u^2_{mnl}, u^x_{nl},u^y_{ml}, u^z_{mn}  \}$ and the source terms $\{F^1_{mnl},$ $F^2_{mnl},$ $F^x_{nl},$ $F^y_{ml}, F^z_{mn}\}$ into vector forms $\{ \vec{u}^1, \vec{u}^2,\vec{u}^3, \vec{u}^x,\vec{u}^y,\vec{u}^z   \}$ and $\{ \vec{f}^1, \vec{f}^2,\vec{f}^3, \vec{f}^x,\vec{f}^y,\vec{f}^z   \}$ following the column-wise linear index order.
%
Taking $\bs \chi=\big\{{\bs\chi}_{\hat{m},\hat{n},\hat{l}}^1\,,\; {\bs\chi}_{\hat{m},\hat{n},\hat{l}}^2\,,\; {\bs\chi}_{\hat{n},\hat{l}}^x,\; {\bs\chi}_{\hat{m},\hat{l}}^y,\; {\bs\chi}_{\hat{m},\hat{n}}^z   \big \}$ in the approximation scheme \eqref{eq: VelwkLap1}, one arrives at
\begin{equation}\label{eq:StokesEq3d}
\begin{aligned}
& 
\mathbf{A}_{11} \vec{u}^1 + \mathbf{A}_{12} \vec{u}^2 + \mathbf{A}_{13} \vec{u}^x + \mathbf{A}_{14} \vec{u}^y + \mathbf{A}_{15} \vec{u}^z  = \vec{f}^1, 
\\
&
\mathbf{A}_{21} \vec{u}^1 + \mathbf{A}_{22} \vec{u}^2 + \mathbf{A}_{23} \vec{u}^x + \mathbf{A}_{24} \vec{u}^y + \mathbf{A}_{25} \vec{u}^z  = \vec{f}^2, 
\\
&
\mathbf{A}_{31} \vec{u}^1 + \mathbf{A}_{32} \vec{u}^2 + \mathbf{A}_{33} \vec{u}^x + \mathbf{A}_{34} \vec{u}^y + \mathbf{A}_{35} \vec{u}^z  = \vec{f}^x, 
\\
&
\mathbf{A}_{41} \vec{u}^1 + \mathbf{A}_{42} \vec{u}^2 + \mathbf{A}_{43} \vec{u}^x + \mathbf{A}_{44} \vec{u}^y + \mathbf{A}_{45} \vec{u}^z  = \vec{f}^y, 
\\
&
\mathbf{A}_{51} \vec{u}^1 + \mathbf{A}_{52} \vec{u}^2 + \mathbf{A}_{53} \vec{u}^x + \mathbf{A}_{54} \vec{u}^y + \mathbf{A}_{55} \vec{u}^z  = \vec{f}^z,
\end{aligned}
\end{equation}
where the coefficient matrices $\mathbf{A}_{11}-\mathbf{A}_{55}$ are given by
\begin{subequations}
\mathleft
\begin{equation}\label{eq:Stokes3d_A11} 
\begin{aligned}
\mathbf{A}_{11} :=& 2 \mathbf{H}\otimes{ \mathbf{H}}\otimes{\mathbf{H}} + \mathbf{I}\otimes{\mathbf{H}}\otimes{\mathbf{G}} +  \mathbf{H}\otimes{\mathbf{I}}\otimes{\mathbf{G}} + \mathbf{I}\otimes{\mathbf{G}}\otimes{\mathbf{H}} + \mathbf{H}\otimes{\mathbf{G}}\otimes{\mathbf{I}}\\
& + \kappa_1 \big( \mathbf{H}\otimes{\mathbf{H}}\otimes{\mathbf{G}} + \mathbf{H}\otimes{\mathbf{G}}\otimes{\mathbf{H}} \big),
\end{aligned}
\end{equation}
\begin{equation}\label{eq:Stokes3d_A22} 
\begin{aligned}
\mathbf{A}_{22} :=& 2 \mathbf{H}\otimes{\mathbf{H}}\otimes{\mathbf{H}} + \mathbf{I}\otimes{\mathbf{H}}\otimes{\mathbf{G}} +  \mathbf{H}\otimes{\mathbf{I}}\otimes{\mathbf{G}} + \mathbf{G}\otimes{\mathbf{I}}\otimes{\mathbf{H}} + \mathbf{G}\otimes{\mathbf{H}}\otimes{\mathbf{I}}\\
&+ \kappa_1 \big( \mathbf{H}\otimes{\mathbf{H}}\otimes{\mathbf{G}} + \mathbf{G}\otimes{\mathbf{H}}\otimes{\mathbf{H}} \big),
\end{aligned}
\end{equation}
\begin{equation}\label{eq:Stokes3d_A334455} 
\begin{aligned}
\mathbf{A}_{33}=\mathbf{A}_{44}=\mathbf{A}_{55} :=& \dfrac{4}{5} \mathbf{H}\otimes{\mathbf{H}} +  \big({\mathbf{H}}\otimes{\mathbf{G}} + {\mathbf{G}}\otimes{\mathbf{H}}\big)+  \dfrac{2}{5}\big(  {\mathbf{I}}\otimes{\mathbf{G}}  +{\mathbf{G}}\otimes{\mathbf{I}} \big) \\
&+ \dfrac{2}{5}\kappa_1 \big( {\mathbf{H}}\otimes{\mathbf{G}} + {\mathbf{G}}\otimes{\mathbf{H}} \big),
\end{aligned}
\end{equation}
\begin{equation}\label{eq:Stokes3d_A12} 
\mathbf{A}_{12} = \mathbf{A}_{21}^{\intercal} := \mathbf{H}\otimes{\mathbf{H}}\otimes{\mathbf{H}} + \mathbf{I}\otimes{\mathbf{H}}\otimes{\mathbf{G}} + \mathbf{H}\otimes{\mathbf{I}}\otimes{\mathbf{G}}+ \kappa_1 \mathbf{H}\otimes{\mathbf{H}}\otimes{\mathbf{G}},
\end{equation}
\begin{equation}\label{eq:Stokes3d_A13}
\mathbf{A}_{13} = \mathbf{A}_{31}^{\intercal} := - \big( \mathbf{H}\otimes{\mathbf{H}}\otimes{\widetilde{\bs H}} + \mathbf{I}\otimes{\mathbf{G}}\otimes{\widetilde{\bs H}}+ \kappa_1  \mathbf{H}\otimes{\mathbf{G}}\otimes{\widetilde{\bs H}} \big), 
\end{equation}
\begin{equation}\label{eq:Stokes3d_A14}
\mathbf{A}_{14} = \mathbf{A}_{41}^{\intercal} :=  \mathbf{H}\otimes{\widetilde{\bs H}}\otimes{\mathbf{H}} + \mathbf{I}\otimes{\widetilde{\bs H}}\otimes{\mathbf{G}}+ \kappa_1  \mathbf{H}\otimes{\widetilde{\bs H}}\otimes{\mathbf{G}}, 
\end{equation}
\begin{equation}\label{eq:Stokes3d_A15}
\mathbf{A}_{15} = \mathbf{A}_{51}^{\intercal} := 2 \widetilde{\bs H}\otimes{\mathbf{H}}\otimes{\mathbf{H}} + \widetilde{\bs H}\otimes{\mathbf{I}}\otimes{\mathbf{G}} + \widetilde{\bs H}\otimes{\mathbf{G}}\otimes{\mathbf{I}}+ \kappa_1 \big( \widetilde{\bs H}\otimes{\mathbf{H}}\otimes{\mathbf{G}} + \widetilde{\bs H}\otimes{\mathbf{G}}\otimes{\mathbf{H}} \big) , 
\end{equation}
\begin{equation}\label{eq:Stokes3d_A23}
\mathbf{A}_{23} = \mathbf{A}_{32}^{\intercal} :=  \mathbf{H}\otimes{\mathbf{H}}\otimes\widetilde{\bs H} + \mathbf{G}\otimes{\mathbf{I}}\otimes\widetilde{\bs H}+ \kappa_1 \mathbf{G}\otimes{\mathbf{H}}\otimes\widetilde{\bs H},
\end{equation}
\begin{equation}\label{eq:Stokes3d_A24}
\mathbf{A}_{24} =\mathbf{A}_{42}^{\intercal} := 2 \mathbf{H}\otimes{\widetilde{\bs H}}\otimes{\mathbf{H}} + \mathbf{I}\otimes{\widetilde{\bs H}}\otimes{\mathbf{G}} + \mathbf{G}\otimes{\widetilde{\bs H}}\otimes{\mathbf{I}}+ \kappa_1 \big( \mathbf{H}\otimes{\widetilde{\bs H}}\otimes{\mathbf{G}} + \mathbf{G}\otimes{\widetilde{\bs H}}\otimes{\mathbf{H}} \big) , 
\end{equation}
\begin{equation}\label{eq:Stokes3d_A25}
\mathbf{A}_{25} = \mathbf{A}_{52}^{\intercal} :=  \widetilde{\bs H}\otimes{\mathbf{H}}\otimes{\mathbf{H}} + \widetilde{\bs H}\otimes{\mathbf{I}}\otimes{\mathbf{G}}+ \kappa_1 \widetilde{\bs H}\otimes{\mathbf{H}}\otimes{\mathbf{G}} , 
\end{equation}
\begin{equation}
\mathbf{A}_{34} = \mathbf{A}_{43}^{\intercal} := \mathbf{H}\otimes{\widetilde{\bs H}}\otimes{\widetilde{\bs H}^{\intercal}}+ \kappa_1 \mathbf{G}\otimes\widetilde{\bs H}\otimes{\widetilde{\bs H}^{\intercal}},
\end{equation}
\begin{equation}
\mathbf{A}_{35} = \mathbf{A}_{53}^{\intercal} := -\big( {\widetilde{\bs H}}\otimes{\mathbf{H}}\otimes{\widetilde{\bs H}^{\intercal}}+ \kappa_1 {\widetilde{\bs H}}\otimes{\mathbf{G}}\otimes{\widetilde{\bs H}^{\intercal}} \big),
\end{equation}
\begin{equation}
\mathbf{A}_{45} = \mathbf{A}_{54}^{\intercal} := {\widetilde{\bs H}}\otimes{\widetilde{\bs H}^{\intercal}}\otimes{\mathbf{H}}+ \kappa_1 {\widetilde{\bs H}}\otimes{\widetilde{\bs H}^{\intercal}}\otimes{\mathbf{G}}.
\end{equation}
\end{subequations}
Here, $\mathbf{G}$, $\mathbf{H}$, $\widetilde{\bs H}$ and $\mathbf{I}$ are defined in proposition \ref{Prop5-GHKMat} and $\otimes$ is the Kronecker product. In addition to the sparsity of $\mathbf{G}$, $\mathbf{H}$ and $\mathbf{I}$, $ \widetilde{\bs H}$ has only one nonzero entry, thus the coefficient matrix of the linear system \eqref{eq:StokesEq3d} is highly sparse, which can be solved efficiently as the two dimensional case in a similar manner.

\begin{rem}
Although we focus on the solution algorithms of the model problems \eqref{eq:Bmodelsystem} and \eqref{eq:umodelsystem}  on $\Lambda^d$,  general computational domain diffeomorphic to $\Lambda^d$ can be solved via Richardson iteration \cite{concus1973use} or matrix-free preconditioned Krylov subspace iterative method, with the help of efficient solution algorithms developed for $\Omega=\Lambda^d$ as preconditioners. 
\end{rem}

\section{DF-BDF-$k$ fully-discretized scheme}\label{sect: scheme}

In this section, we employ the proposed divergence-free bases to discretize the MHD system \eqref{eq: mhd}-\eqref{eq: compatible}. We emphasize that with the aid of these bases, pressure field is fully decoupled with the velocity field and unconditionally energy-stable schemes can be easily obtained.

\subsection{Formulation of numerical scheme}
Let us partition the time interval $[0,T]$ into $N_t$ equispaced sub intervals with $\tau=T/N_t$ being the time step size.  Denote $ t_n=n\tau,\, n=0,1,\cdots,N_t$ and $\chi^n$ the approximation of a generic variable $\chi$ at time $t_n$. The initial fields $\bs u^0$ and $\bs B^0$ are defined by
\begin{equation}\label{eq: inischeme}
\bs u^0(\bs x)=\bs u_{\rm in}(\bs x),\quad \bs B^0(\bs x)=\bs B_{\rm in}(\bs x).
\end{equation}
Given $(\bs u^n, \bs B^n)\in \bs H^1_{N,0}({\rm div}0; \Omega)\times \bs H_{N,0}({\rm div}0; \Omega)$, we seek  $(\bs u^{n+1}, \bs B^{n+1})\in \bs H^1_{N,0}({\rm div}0; \Omega)\times \bs H_{N,0}({\rm div}0; \Omega)$ such that
\begin{subequations}\label{eq: scheme}
\mathleft
\begin{equation} \label{eq: scheme01}
\qquad
\begin{aligned}
 \frac{1}{\tau} \big({\gamma {\bs u}^{n+1}-\hat{\bs u}}, \bs \chi \big)+\big( \tilde{\bs u}^{n+1}\cdot \nabla \bs u^{n+1},\bs \chi  \big)&+\nu \big(\nabla \bs u^{n+1},\nabla \bs \chi  \big)\\
&= \frac{1}{{ \mu \rho}}\big( {\nabla \times \bs B^{n+1} \times \tilde {\bs B}^{n+1}}, \bs \chi \big) +(\bs f^{n+1}, \bs \chi),
\end{aligned}
\end{equation}
\begin{equation}\label{eq: scheme02}
\qquad
\frac{1}{\tau} \big({\gamma \bs B^{n+1} -\hat {\bs B}},\bs \Phi \big)+\eta \big(\nabla \times \bs B^{n+1}, \nabla \times \bs \Phi \big)-\big(\nabla \times (\bs u^{n+1}\times \tilde{\bs B}^{n+1}),\bs \Phi \big)=0, 
\end{equation}
\end{subequations}
for  $\forall\; (\bs \chi, \bs \Phi) \in \bs H^1_{N,0}({\rm div}0; \Omega)\times \bs H_{N,0}({\rm div}0; \Omega).$ 

In the above, the notation $({\gamma \chi^{n+1}-\hat \chi})/{\tau}$ represents a $k$-th order backward differentiation formula and ${\tilde \chi}^{n+1} $ represents the $k$-th order explicit approximation of $\chi$ at time $t_{n+1}$. For $k=1,2$, $\hat \chi$, $\gamma$ and ${\tilde \chi}^{n+1} $ are given by
\begin{equation}
\hat \chi=
\begin{cases}
 \chi^n,   & k=1\\[5pt]
 2\chi^n-\chi^{n-1}/2, & k=2 
 \end{cases},\quad 
\gamma=
\begin{cases}
1, & k=1\\[5pt]
{3}/{2}, & k=2
\end{cases}, \quad 
\tilde{\chi}^{n+1}=
\begin{cases}
 \chi^n, & k=1\\
 2\chi^n-  \chi^{n-1}, & k=2
\end{cases} \,.
\end{equation}

\subsection{Discrete energy dissipation law}
The proposed numerical scheme satisfies a discrete energy dissipative law, thus are suitable for long time simulations with large time step sizes. We present the discrete energy law as follows.

\begin{thm}
In the absence of the external force $\bs f^{n+1}$, the numerical scheme consisting of equations \eqref{eq: scheme01}-\eqref{eq: scheme02} satisfies the following property:
\begin{equation}\label{eq: dislaw}
\mathcal{E}^{n+1}-\mathcal{E}^n=-\mathcal{D}^{n+1} -\nu \| \nabla \bs u^{n+1}\| -\frac{\eta}{\mu \rho}\|\nabla\times \bs B^{n+1} \|^2,
\end{equation}
where $\mathcal{E}^n$ is the discrete energy
\begin{equation}\label{eq: Elaw}
\mathcal{E}^n=
\begin{cases}
\dfrac{1}{2\tau}\Big(  \|\bs u^{n} \|^2 +\dfrac{1}{\mu\rho}\|\bs B^{n} \|^2    \Big), & k=1,\\[5pt]
\dfrac{1}{2\tau}\Big( { \|\bs u^{n} \|^2 +\| 2\bs u^{n}-\bs u^{n-1}\|^2}  +\dfrac{1}{\mu\rho}{ \|\bs B^{n} \|^2 +\dfrac{1}{\mu\rho} \| 2\bs B^{n}-\bs B^{n-1}\|^2}    \Big), & k=2;
\end{cases}
\end{equation}
and $\mathcal{D}^{n+1}$ is the artificial dissipation term
\begin{equation}\label{eq: Dlaw}
\mathcal{D}^{n+1}=
\begin{cases}
 \dfrac{1}{2\tau} \Big(    \|\bs u^{n+1}-\bs u^n \|^2 +  \dfrac{1}{\mu\rho} \|\bs u^{n+1}-\bs u^n \|^2 \Big), & k=1,\\
 \dfrac{1}{2\tau}\Big(   {\big\| \bs u^{n+1}-2\bs u^n+\bs u^{n-1} \big\|^2} +\dfrac{1}{\mu\rho }  {\big\| \bs B^{n+1}-2\bs B^n+\bs B^{n-1} \big\|^2}    \Big), & k=2.
\end{cases}
\end{equation}
\end{thm}

\begin{proof}
Taking $\bs \chi=\bs u^{n+1}$ and $\bs \Phi={(\mu \rho)}^{-1} \bs B^{n+1}$ in equations \eqref{eq: scheme}, summing up the resultant equations and using the following identities 
\begin{align}
& \big(  \tilde{\bs u}^{n+1}\cdot \nabla \bs u^{n+1}, \bs u^{n+1}   \big)=0,\;\;{\rm with}\;\; \nabla \cdot \tilde {\bs u}^{n+1}=0,\;\; \bs n \cdot \tilde{\bs u}^{n+1}\big|_{\partial \Omega}=0; \label{eq: id1}\\
&\big(\nabla \times \bs B^{n+1} \times \tilde{\bs B}^{n+1}, \bs u^{n+1} \big)+ \big(\nabla \times(\bs u^{n+1}\times \tilde {\bs B}^{n+1}  ), \bs B^{n+1} \big)=0, \;\; \bs n \cdot \tilde{\bs B}^{n+1}\big|_{\partial \Omega} =\bs n\cdot \bs u^{n+1}\big|_{\Omega}=0,  \label{eq: id2}
\end{align}
we arrives at
\begin{equation}\label{eq: en1}
 \big(\frac{\gamma {\bs u}^{n+1}-\hat{\bs u}}{\tau}, \bs u^{n+1} \big)+ \big(\frac{\gamma \bs B^{n+1} -\hat {\bs B}}{\mu\rho\tau},\bs B^{n+1} \big)+\nu \|\nabla \bs u^{n+1} \|^2 +\frac{\eta}{\mu\rho} \|\nabla \times \bs B^{n+1} \|^2=(\bs f^{n+1}, \bs u^{n+1}). 
\end{equation}
With the help of the identities for a generic variable $\chi$
\begin{subequations}
\begin{equation}\label{eq: id3}
\chi^{n+1}(\chi^{n+1}-\chi^n)=\frac{1}{2}(\chi^{n+1})^2-\frac{1}{2}(\chi^n)^2 +\frac{1}{2}(\chi^{n+1}-\chi^n)^2,
\end{equation}
\begin{equation}\label{eq: id4}
\begin{aligned}
\chi^{n+1}(3\chi^{n+1}-4\chi^n+\chi^{n-1})&=\frac{1}{2}\Big( {(\chi^{n+1})^2+(2\chi^{n+1}-\chi^n)^2}  \Big)-\frac{1}{2} \Big( {(\chi^n)^2+(2\chi^n-\chi^{n-1})^2}\Big) \\
&\quad+\frac{1}{2}  {\big(\chi^{n+1}-2\chi^n+\chi^{n-1} \big)^2},
\end{aligned}
\end{equation}
\end{subequations}
we can transform equation \eqref{eq: en1} into
\begin{equation}
\begin{aligned}
&\frac{1}{2\tau}\Big(  \|\bs u^{n+1} \|^2 +\frac{1}{\mu\rho}\|\bs B^{n+1} \|^2    \Big)-\frac{1}{2\tau}\Big(  \|\bs u^{n} \|^2 +\frac{1}{\mu\rho}\|\bs B^{n} \|^2    \Big)+\nu \| \nabla \bs u^{n+1}\| +\frac{\eta}{\mu \rho}\|\nabla\times \bs B^{n+1} \|^2  \\
&=- \frac{1}{2\tau} \Big(    \|\bs u^{n+1}-\bs u^n \|^2 +  \frac{1}{\mu\rho} \|\bs u^{n+1}-\bs u^n \|^2 \Big)+(\bs f^{n+1}, \bs u^{n+1}),\;\;\;\; k=1;
\end{aligned}
\end{equation}
and
\begin{equation}
\begin{aligned}
& \frac{1}{2\tau}\Big( { \|\bs u^{n+1} \|^2 +\| 2\bs u^{n+1}-\bs u^n\|^2}  +\frac{1}{\mu\rho}{ \|\bs B^{n+1} \|^2 +\frac{1}{\mu\rho} \| 2\bs B^{n+1}-\bs B^n\|^2}    \Big) +\nu \| \nabla \bs u^{n+1}\|\\
&-\frac{1}{2\tau}\Big( { \|\bs u^{n} \|^2 +\| 2\bs u^{n}-\bs u^{n-1}\|^2}  +\frac{1}{\mu\rho}{ \|\bs B^{n} \|^2 +\frac{1}{\mu\rho} \| 2\bs B^{n}-\bs B^{n-1}\|^2}    \Big) +\frac{\eta}{\mu \rho}\|\nabla\times \bs B^{n+1} \|^2   \\
&= -\frac{1}{2\tau}\Big(   {\big\| \bs u^{n+1}-2\bs u^n+\bs u^{n-1} \big\|^2} +\frac{1}{\mu\rho }  {\big\| \bs B^{n+1}-2\bs B^n+\bs B^{n-1} \big\|^2}    \Big)+(\bs f^{n+1}, \bs u^{n+1}),\;\;\;\; k=2.
\end{aligned}
\end{equation}
The discrete energy dissipation law \eqref{eq: dislaw} can be obtained directly from the above by omitting the boundary force term $(\bs f^{n+1}, \bs u^{n+1})$ and noting the discrete energy and artificial dissipation in equations \eqref{eq: Elaw}-\eqref{eq: Dlaw}.
\end{proof}

%
The linear-coupled system \eqref{eq: scheme} can be effectively solved using a sub-iteration method. In the initial iteration step, we employ the substitution $\bs u^{n+1,(0)}=\tilde{\bs u}^{n+1}, \,\bs B^{n+1,(0)}=\tilde{\bs B}^{n+1}$ to replace $\bs u^{n+1}$ and $\bs B^{n+1}$ in the non-linear terms. This allows us to decouple the problem and solve for $\bs u^{n+1,(1)}$ and $\bs B^{n+1,(1)}$. Subsequently, we update $\bs u^{n+1,(0)}$ and $\bs B^{n+1,(0)}$ by $\bs u^{n+1,(1)}$ and $\bs B^{n+1,(1)}$ respectively, and solve the resultant decoupled equations. In each iteration step, it suffices to solve the following two decoupled equations:
\begin{itemize}
\item Find $\bs u^{n+1,(r+1)}\in \bs H^1_{N,0}({\rm div}0; \Omega)$ such that
\begin{equation}\label{eq: usolve}
 \big(\nabla \bs u^{n+1,(r+1)},\nabla \bs \chi  \big)+\kappa_1 \big({\bs u}^{n+1,(r+1)}, \bs \chi \big)=(\bs F,\bs \chi), \;\; \forall \bs \chi \in \bs H^1_{N,0}({\rm div}0; \Omega),
\end{equation}
where $\kappa_1:={\gamma}/{(\nu\tau)}$ and $ \bs F={\nu}^{-1}\Big(\bs f+{\tau}^{-1} \hat{\bs u}- \tilde{\bs u}^{n+1}\cdot \nabla \bs u^{n+1,(r)}+ {{( \mu \rho)}^{-1}} {\nabla \times \bs B^{n+1,(r)} \times \tilde {\bs B}^{n+1}}  \Big).$

\item Find $\bs B^{n+1,(r+1)}\in \bs H_{N,0}({\rm div}0; \Omega)$ such that
\begin{equation}\label{eq: Bsolve}
 \big(\nabla \times \bs B^{n+1,(r+1)}, \nabla \times \bs \Phi \big)+\kappa_0 \big(\bs B^{n+1,(r+1)},\bs \Phi \big)=(\bs J,\bs \Phi), \;\; \bs \Phi \in \bs H_{N,0}({\rm div}0; \Omega),
\end{equation}
where $\kappa_0={\gamma}/{(\eta\tau)}$ and $ \bs J={(\eta \tau)}^{-1} \hat{\bs B} +\eta^{-1} \nabla \times (\bs u^{n+1,(r)}\times \tilde{\bs B}^{n+1})  .
$
\end{itemize}
These two decoupled equations can be solved efficiently by the solution algorithms described in the previous section. We repeat the iterations until the following stopping criterion is satisfied:
\begin{equation}
{\rm max}\Big\{ \frac{\| \bs u^{n+1,(r+1)}-\bs u^{n+1,(r)}\|^2}{\| \bs u^{n+1,(r+1)}\|^2 }, \frac{\| \bs B^{n+1,(r+1)}-\bs B^{n+1,(r)}\|^2}{  \| \bs B^{n+1,(r+1)}\|^2 } \Big
\}< \epsilon.
\end{equation}
This sub-iteration scheme is quite efficient as it usually takes a few iterations (1$\sim$3) to converge for each time step for the proposed DF-BDF-$k$ schemes for high order schemes, even for very large time step sizes (see the comprehensive numerical experiments in Section \ref{sect: numer}).  

\begin{rem}
The proposed scheme \eqref{eq: scheme} can be readily extended to $k$-step BDF schemes ($k=3,\cdots, 6$) by defining $\hat \chi$, $\gamma$ and ${\tilde \chi}^{n+1} $ in Appendix \ref{ap: bdfk}. Though the general DF-BDF-$k$ schemes are not unconditionally energy-stable for $k=3,\cdots, 6$, numerical experiments demonstrate that these schemes are quite stable for large time step sizes when equipped with the divergence-free spectral bases, especially for $k=3,4$. We refer the readers to Section \ref{sect: numer} for the supportive examples. 
\end{rem}

\section{Representative numerical experiments}\label{sect: numer}
In this section, we conduct several representative numerical experiments in two and three dimensions to verify the accuracy, efficiency and stability of the proposed DF-BDF-$k$ schemes. We first employ contrived analytic solutions to demonstrate the spatial and temporal convergence rates of the proposed methods.  Simulation results of DF-BDF-3 to DF-BDF-6 using large time steps are provided to show the enhanced numerical  stability and accuracy of these schemes,  when equipped with the exact divergence-free basis functions. The efficiency of the proposed methods is shown with the iteration numbers of the sub-iteration method in each time step against large time step sizes.  We then use the driven cavity flow problem as a benchmark to demonstrate that the proposed methods can produce stable and physically accurate results.

\subsection{Two dimensional case}

\begin{example}[\bf Convergence test for DF-BDF-1 and -2 schemes in 2D.] \label{ex: ex1}

 We start with the spatial and temporal convergence tests of the proposed DF-BDF-$k$ schemes developed herein using a contrived analytic solution. Let us fix the computational domain $\Omega=\Lambda^2.$ The manufactured solution of the MHD system is given by
\begin{equation}\label{examp1}
\begin{split}
\bs u(\bs x,t) &= \nabla \times \Big[\big(1-x_1^2 \big)^3 \big(1-x_2^2 \big)^3\big(1+\sin(\pi x_1)\cos(\pi x_2)\big)\exp\big(\sin(\pi x_1)\cos(\pi x_2)/10 \big)\cos(t) \Big],\\
\bs B(\bs x,t) &= \nabla \times \Big[\big(1-x_1^2 \big)^3 \big(1-x_2^2 \big)^3\big(1+\sin(\pi x_1)\cos(\pi x_2)\big)\exp\big(\sin(\pi x_1)\cos(\pi x_2)/10 \big)\cos(t) \Big],
\end{split}
\end{equation}
such that $\bs u(\bs x,t)$ and $\bs B(\bs x,t)$ satisfy the divergence-free constraints \eqref{eq: divu}-\eqref{eq: divB} and the boundary condition \eqref{eq: mhdbc}, and the initial condition $\bs u_{\rm in}(\bs x):=\bs u(\bs x,t=0)$ and $\bs B_{\rm in}(\bs x):=\bs B(\bs x,t=0)$ satisfy compatible requirements \eqref{eq: compatible}. The pressure field $p$ is set with
\begin{equation}\label{eq: pressure2d}
p(\bs x,t)=10\big[(x_1-1/2)^3 \,x_2^2+(1-x_1)^3\,(x_2-1/2)^3\big].
\end{equation}
To guarantee that the above contrived solution satisfy the MHD system, a body force term $\bs g(\bs x,t)$ is added to equation \eqref{eq: Bmom} and the forcing terms $\bs f$ and $\bs g$ are chosen such that the manufactured solution satisfy
 \begin{subequations}\label{eq: mhdnew}
\begin{align}
&\partial_t \bs u+\bs u\cdot \nabla \bs u -\nu \nabla^2 \bs u+{\rho}^{-1}\nabla p -(\mu \rho)^{-1}(\nabla \times \bs B) \times \bs B= {\bs f}, \label{eq: umomforce}\\
& \partial_t \bs B-\nabla \times (\bs u\times \bs B)+\eta \nabla \times(\nabla \times \bs B)=\bs g. \label{eq: Bmomforce}
\end{align}
\end{subequations}
The parameters are prescribed as  $\nu=\rho=\mu=\eta=1$. We employ the proposed DF-BDF-$k$ scheme to numerically integrate the MHD system from $t=0$ to $t=T$. We prescribe the stopping tolerance $\epsilon=10^{-10}$ throughout the paper. The numerical errors between the numerical solution and the analytic solution of $\bs u$ and $\bs B$ are respectively measured under $\bs H^{1}$- and $\bs H({\rm curl})$-norms at $t=T$. 

We first conduct convergence tests for DF-BDF-$1$-$2$ schemes. In Figure \ref{fig_BDF12err_N2d} (a), we plot the numerical errors as a function of polynomial order $N$ with a fixed time step size $\tau=3.125\times 10^{-3}.$ It can be observed that the errors decay exponentially first and then level off at around $5\times 10^{-4}$ and $5\times 10^{-8}$, respectively for DF-BDF-$1$ and DF-BDF-$2$ schemes, showing a saturation due to temporal discretization error. Next, we fix the polynomial order $N=50$ and depict the errors as a function of $\tau$ in Figure \ref{fig_BDF12err_N2d} (b). First- and second-order convergence rates in time respectively for DF-BDF-$1$ and DF-BDF-$2$ schemes can be clearly observed. 


\end{example}

\begin{example}[\bf Convergence test for DF-BDF-$3$-$6$ schemes in 2D.] \label{ex: ex2}

It is known that implicit-explicit BDF-$k$ schemes for the discretization of MHD system based on projection method will encounter severe stability issue and often blow up even for very small time step sizes when $k\geq 3$. In the forthcoming two numerical experiments, we show that with the aid of the exact divergence-free basis, the stability of high-order BDF schemes are greatly enhanced, especially for $k=3,4$.  In Figure \ref{fig_BDF36err_N2d} (a)-(b), we fix the polynomial order $N=50$ and plot the history of errors against time step size $\tau$ for DF-BDF-$k$ schemes with $k=3,\cdots, 6$. It can be seen that DF-BDF-$3$-$4$ schemes exhibit clearly the third- and forth-order convergence rate in time. However, for  DF-BDF-$5$ scheme, the convergence rate is fifth-order for $\tau\geq 1.25\times 10^{-2}$ and slows down afterwards. As for DF-BDF-$6$ scheme, the errors first increase with the decrease of step size before $\tau=6.25\times 10^{-3}$ and then decrease to $10^{-11}$ as step size decreases, due to numerical instability.  


\end{example}

\begin{example}[\bf Efficiency of the sub-iteration method for DF-BDF-$k$ schemes in 2D.]\label{ex: ex3} 

In this example, we look into the efficiency of the sub-iteration algorithm for DF-BDF-$k$ schemes. We depict in Figure \ref{fig_iterStep_N2d} the iteration numbers as a function of time steps with a fixed polynomial order $N=50$. We adopt a large time step size $\tau=0.1$ first.  It can be seen in Figure \ref{fig_iterStep_N2d} (a) that the iteration numbers decrease as $k$ increase from 1 to 4. It is worthwhile to note that it usually takes only 1 iteration for DF-BDF-4 scheme to converge at such a large time step, demonstrating an exceptional efficiency of this method.  As for $k=5,6$, the iteration numbers abruptly increase due to numerical instability. We slightly decrease the step size to $\tau=0.01$ and conduct the same numerical experiment. It can be seen in Figure \ref{fig_iterStep_N2d} (b) that the iteration numbers of the DF-BDF-1 scheme are around 4$\sim$6, due to the lack of accuracy of the BDF-1 temporal discretization. While the sub-iteration method for DF-BDF-2-5 schemes are quite efficient, as they usually take only 1 iteration to converge. As for DF-BDF-6 scheme, the iteration numbers increase gradually with the time step, implying that the numerical solution obtained by this scheme gradually deviates from the true solution due to stability issue. 


\end{example}

\begin{figure}[htbp]
\begin{center}
  \subfigure[ Errors vs $N$]{ \includegraphics[scale=.38]{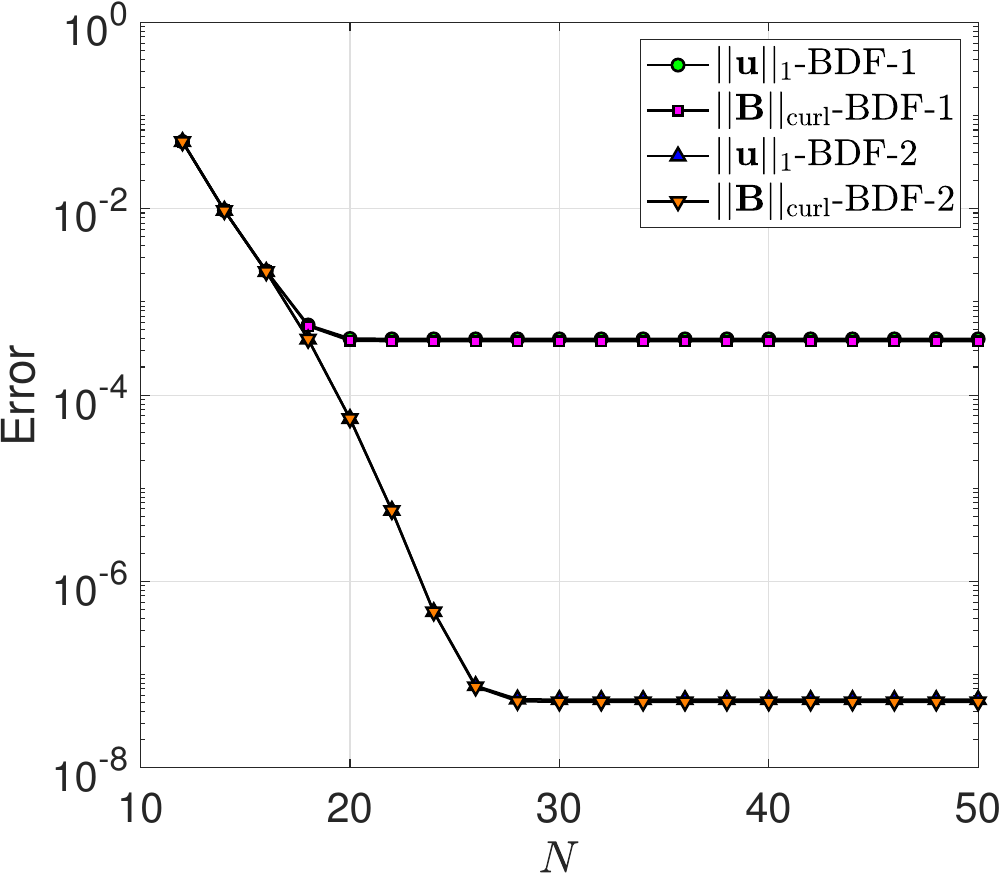}}\qquad \quad
  \subfigure[Errors vs $\tau$ ]{ \includegraphics[scale=.38]{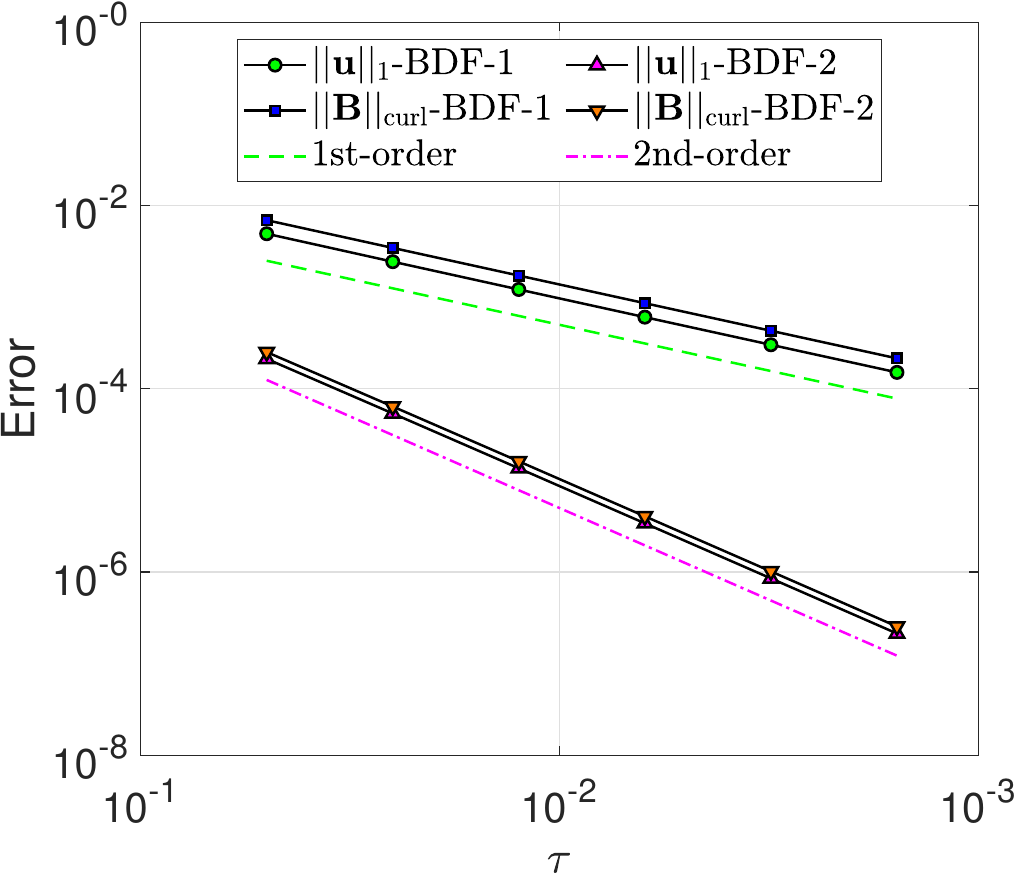}}  
 \caption{\small Spatial and temporal convergence test of DF-BDF-$k$ ($k=1,2$) schemes for solving the MHD system in 2D: (a) $\bs H^1$-errors of $\bs u$ and $\bs H({\rm curl})$-errors of $\bs B$ versus polynomial order $N$ (fixed $\tau=3.125\times 10^{-3}$ and $T=0.1$); (b) errors versus time step sizes $\tau$ (fixed polynomial order $N=50$ and $T=1$). } 
   \label{fig_BDF12err_N2d}
\end{center}
\end{figure}

 \begin{figure}[htbp]
\begin{center}
  \subfigure[Errors of $\bs u$ vs $\tau$]{ \includegraphics[scale=.38]{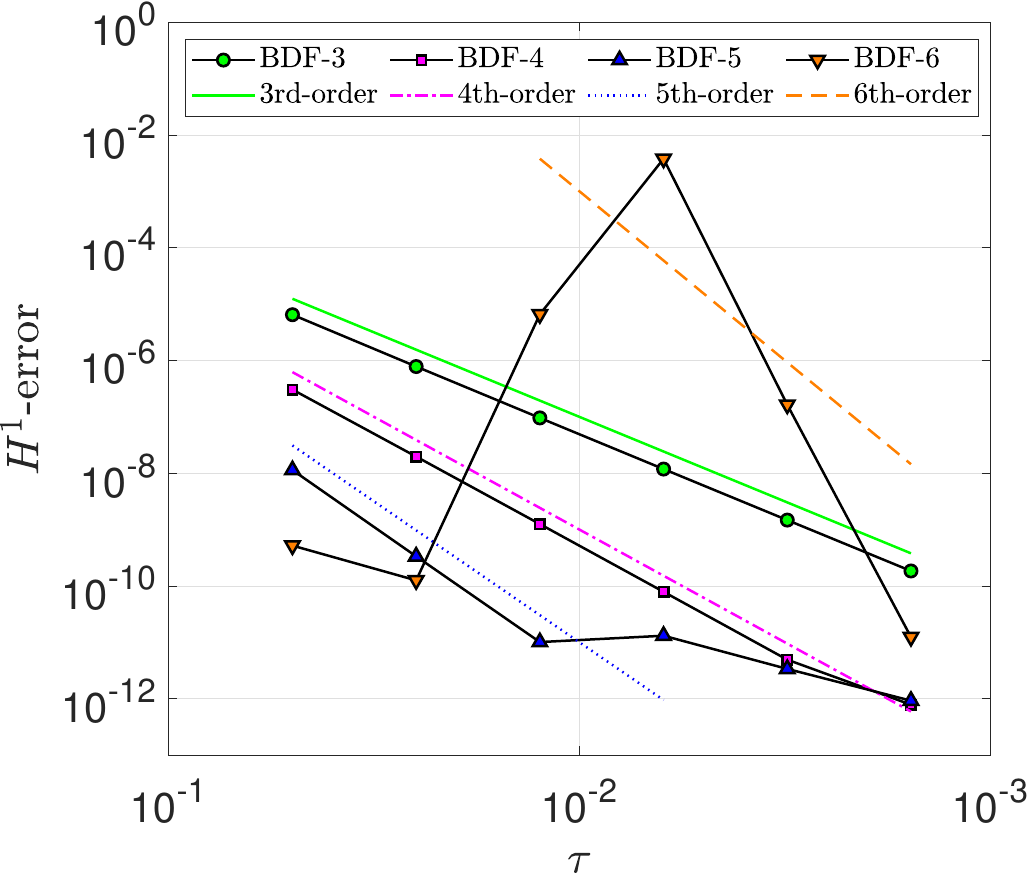}}\qquad \quad
  \subfigure[Errors of $\bs B$ vs $\tau$ ]{ \includegraphics[scale=.38]{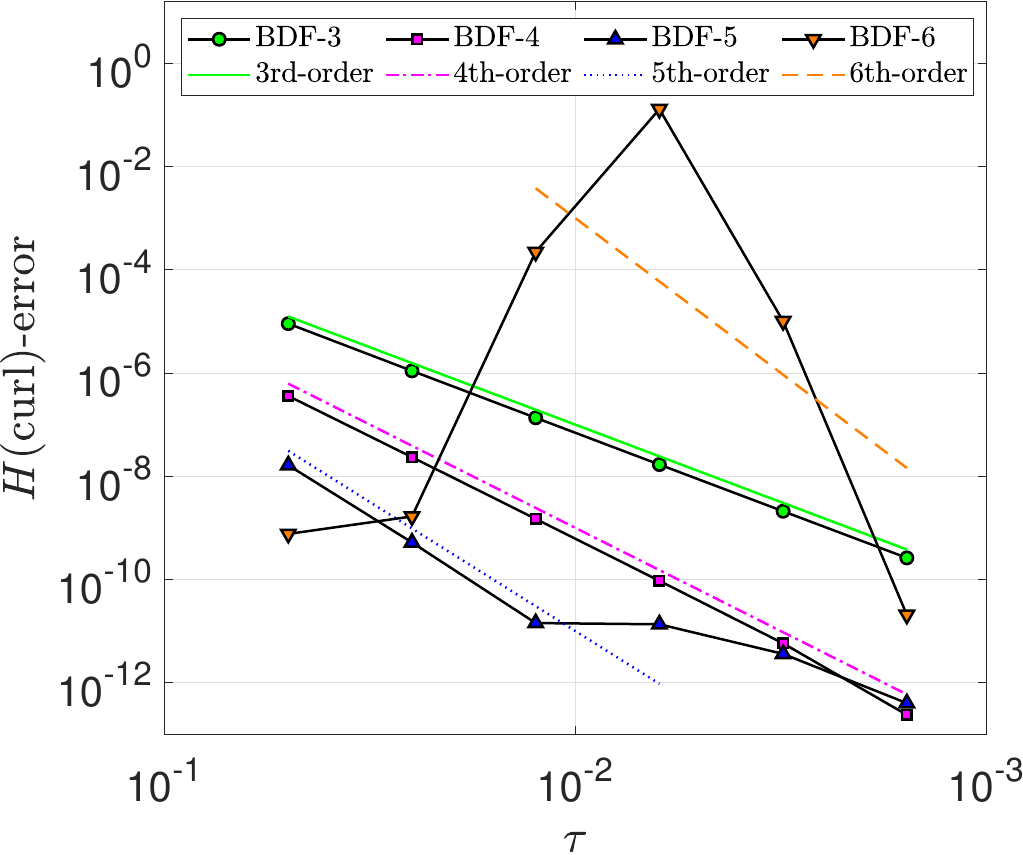}} 
\caption{\small Temporal convergence of DF-BDF-$k$ ($k=3,\cdots, 6$) schemes for solving the MHD system in 2D: (a) $\bs H^1$-errors of $\bs u$ and (b) $\bs H({\rm curl})$-errors of $\bs B$ versus time step sizes $\tau$ (fixed polynomial order $N=50$ and $T=1$.} 
   \label{fig_BDF36err_N2d}
\end{center}
\end{figure}

\begin{figure}[htbp]
\begin{center}
 \subfigure[$\tau=0.1$]{ \includegraphics[scale=.38]{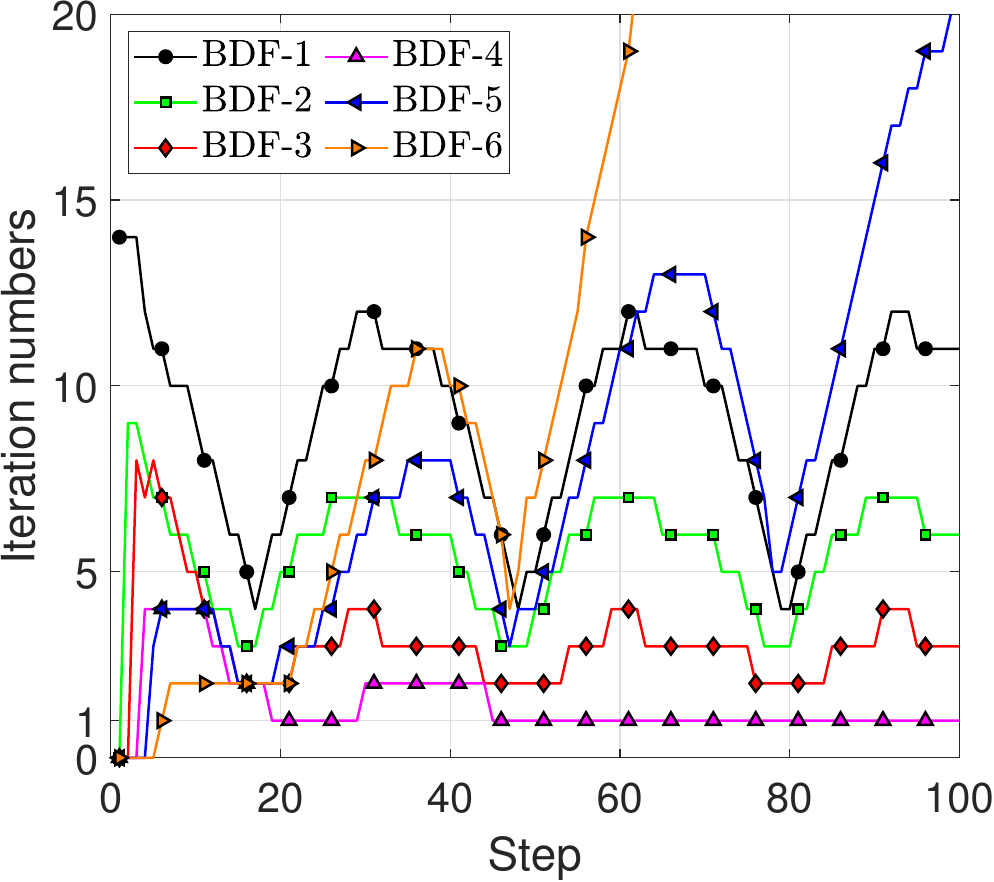}} \qquad
  \subfigure[$\tau=0.01$]{ \includegraphics[scale=.38]{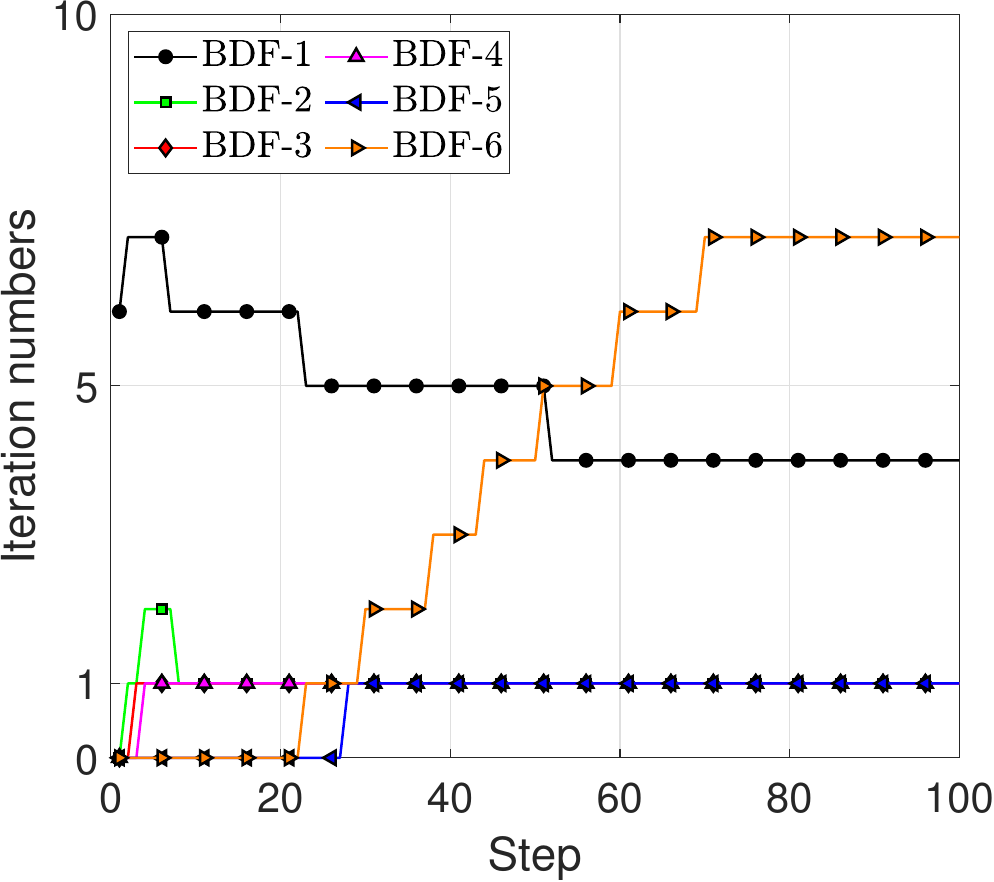}} 
\caption{\small Number of iterations of the sub-iteration method for DF-BDF-$k$ schemes as a function of time step in 2D: (a) $\tau=0.1$ and (b) $\tau=0.01$. The simulations are obtained with $N=50$ and 100 time steps.} 
   \label{fig_iterStep_N2d}
\end{center}
\end{figure}

\begin{example}[\bf Long time simulation for DF-BDF-3-6 schemes with large step size in 2D.]\label{ex: ex4} 

To further investigate the stability of high-order DF-BDF-$k$ schemes, we employ a large time step size $\tau=0.1$ with a fixed $N=50$ and conduct long time simulations to $T=100.$ It is observed in Figure \ref{fig_stable_N2d} that DF-BDF-5-6 schemes blows up before $t=13$ at such a large time step. In the meanwhile, DF-BDF-$3$-4 schemes produce results with numerical errors less than $10^{-4}$ and $10^{-5}$, respectively. In view of previous numerical experiments for the proposed DF-BDF-$k$ schemes, we advocate to adopt DF-BDF-$4$ scheme for numerical simulation, as it has distinctive feature of accuracy, efficiency and stability.  

\begin{figure}[!htbp]
\begin{center}
  \subfigure[$H^1$-error long time stability]{ \includegraphics[scale=.53]{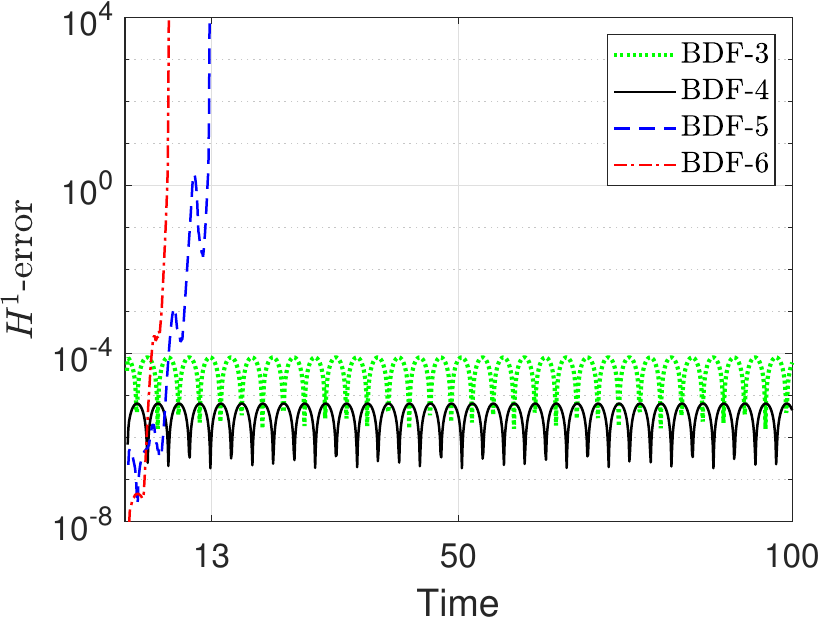}} \quad
  \subfigure[$H(\mathrm{curl})$-error long time stability]{ \includegraphics[scale=.523]{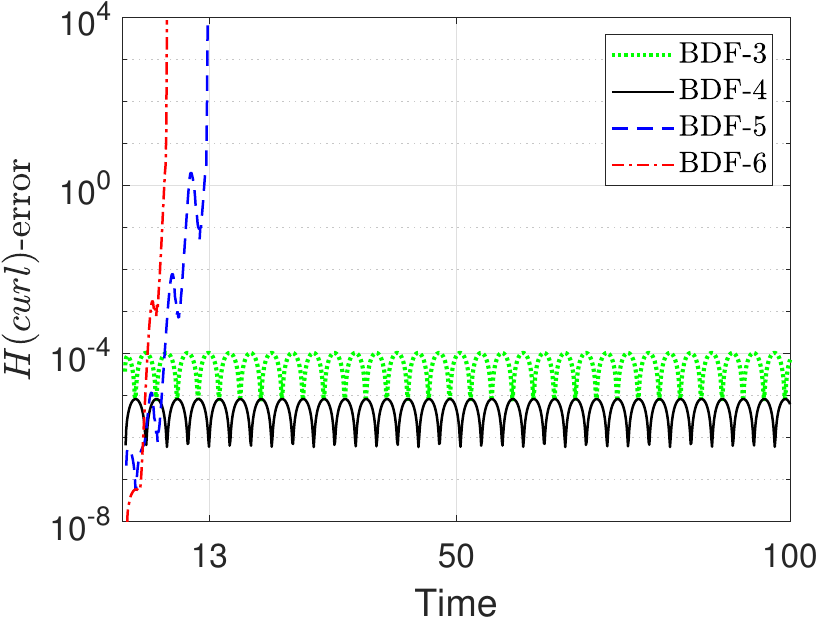}} 
\caption{\small (a)-(b) Long time stability tests of DF-BDF-$k$ schemes in 2D. Time history of errors with large time step size $\tau=0.1$ for (a) $\bs u$ and (b) $\bs B$, respectively. The simulations are obtained with polynomial order $N=50$. } 
   \label{fig_stable_N2d}
\end{center}
\end{figure}

\end{example}

\begin{example}[\bf Driven cavity flow problem in 2D.]\label{ex: ex5}

Next, we focus on the driven cavity flow problem in 2D to demonstrate that the proposed schemes can produce physically accurate simulation results. We fix the computational domain to be $\Omega = [0,1]^2$.  The boundary condition of this problem is given by
\begin{equation}\label{eq: Drivenbd}
 \bs u=
\begin{cases}
(1,0)^\intercal, & {\rm at}\; \Gamma_1, \\
(0,0)^{\intercal}, & {\rm at}\; \partial \Omega \backslash \Gamma_1,
\end{cases}
\quad
 \bs n \cdot \bs B=\bs n \cdot \bs B_0,\quad \bs n \times (\nabla \times \bs B)=0\;\;\;{\rm at}\;\;\partial \Omega,
\end{equation}
where $\Gamma_1:=\big\{(x_1,x_2)^{\intercal} \big| 0\leq x_1\leq 1, x_2=1\big \}$ and $\bs B_0 = (0,1)^\intercal$. The source term is set as $\bs f=\bs 0$ in system \eqref{eq: mhd}. We also denote ${\rm Re}$, ${\rm Rem}$ and ${\rm Ha}$ as the Renolds number, magnetic Renolds number and Hartmann number, respectively given by
\begin{equation}\label{eq: hartmann}
{\rm Re}=\frac{1}{\nu}, \quad {\rm Rem}=\frac{1}{\eta},\quad {\rm Ha}=\sqrt{\frac{{\rm Re}\times {\rm Rem}}{\mu \rho}}.
\end{equation}
We fix ${\rm Rem}=100$, ${\rm Ha}=\sqrt{10}$ and vary ${\rm Re}=100,500,1500,10000.$ We employ the proposed DF-BDF-4 scheme with polynomial order $N=100$ and time step size $\tau=10^{-2}$  to numerical integrate the MHD system until steady state solutions are obtained.  In Figure \ref{figs: Velocity2D} (a)-(d), we plot the streamline of velocity field with the contour of the vorticity field ($\nabla \times \bs u$) as background.  We observe that the center of the streamline is close to the upper right corner with relatively small Renolds numbers. As ${\rm Re}$ increases, the center of the streamline gradually moves to the center of the region and there forms new eddies at the lower corner. Moreover, we can find that the contour of vorticity field  become sharper as ${\rm Re}$ increases. Meanwhile, we depict in Figure \ref{figs: Velocity2D} (e)-(h) the associated streamlines of magnetic field. These figures indicate that the magnetic fields remain quite close to each other as Reynolds number changes. These expected phenomena demonstrate that the proposed DF-BDF-4 scheme can produce physically accurate simulation results with relatively large time step size and coarse grids. 

%

\end{example}

\begin{figure}[htbp]
\begin{center}
    \subfigure[Re=100]{ \includegraphics[scale=.19]{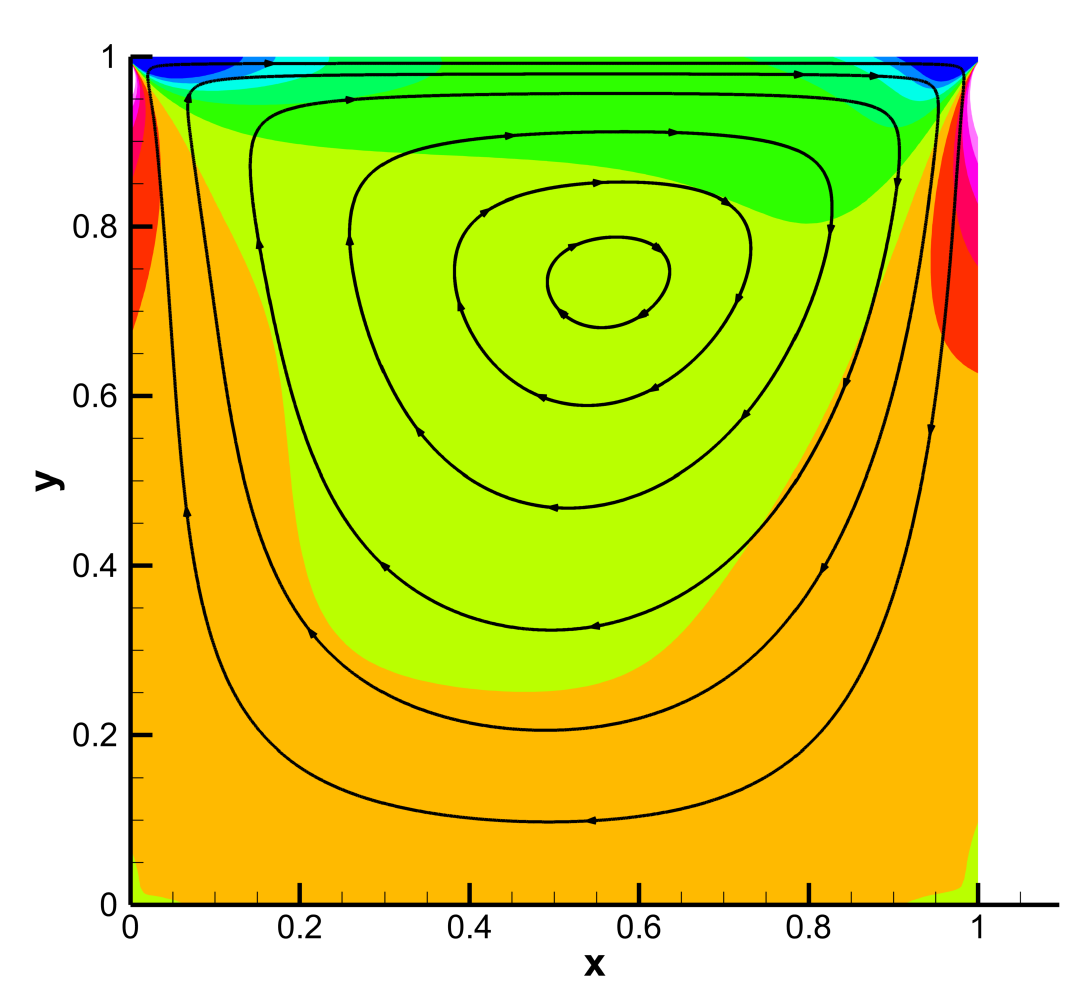}}
      \subfigure[Re=500]{ \includegraphics[scale=.19]{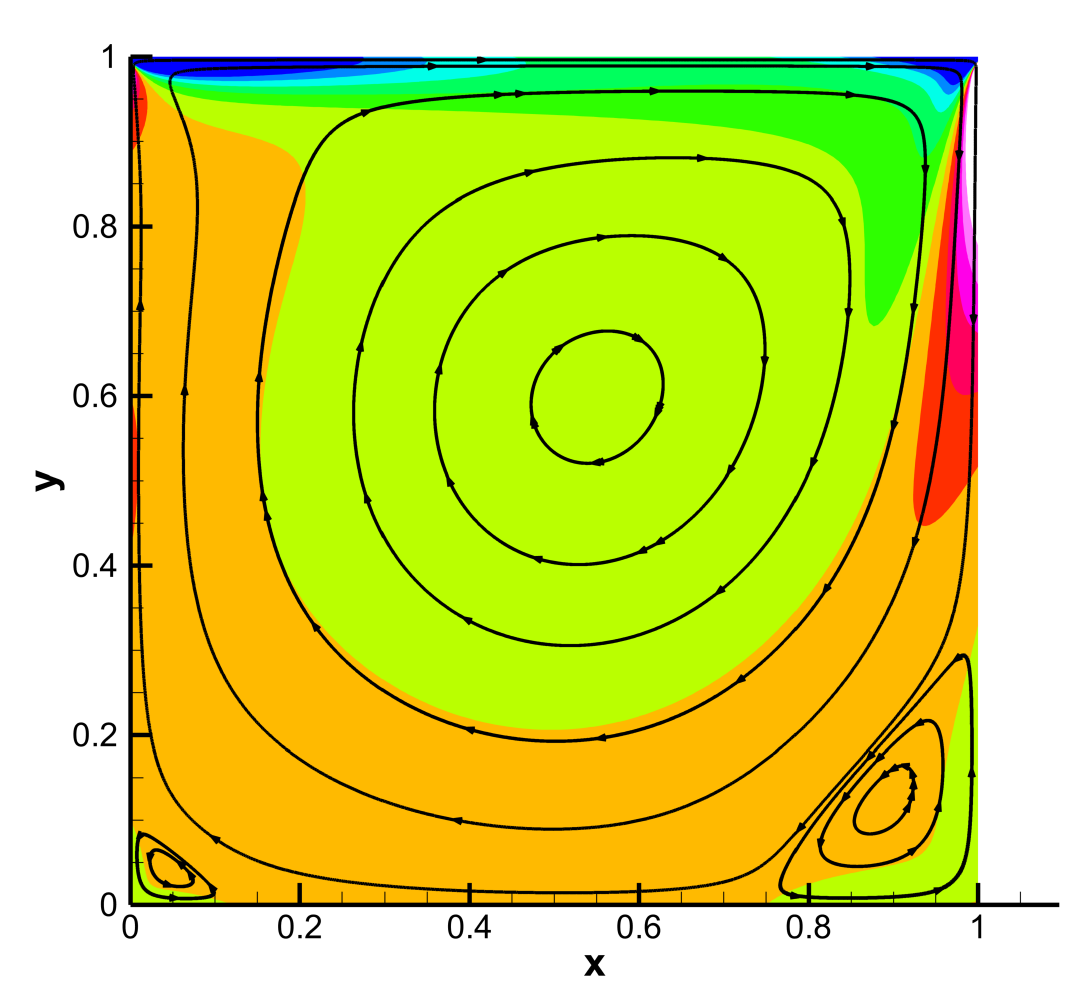}}
       \subfigure[Re=1500]{ \includegraphics[scale=.19]{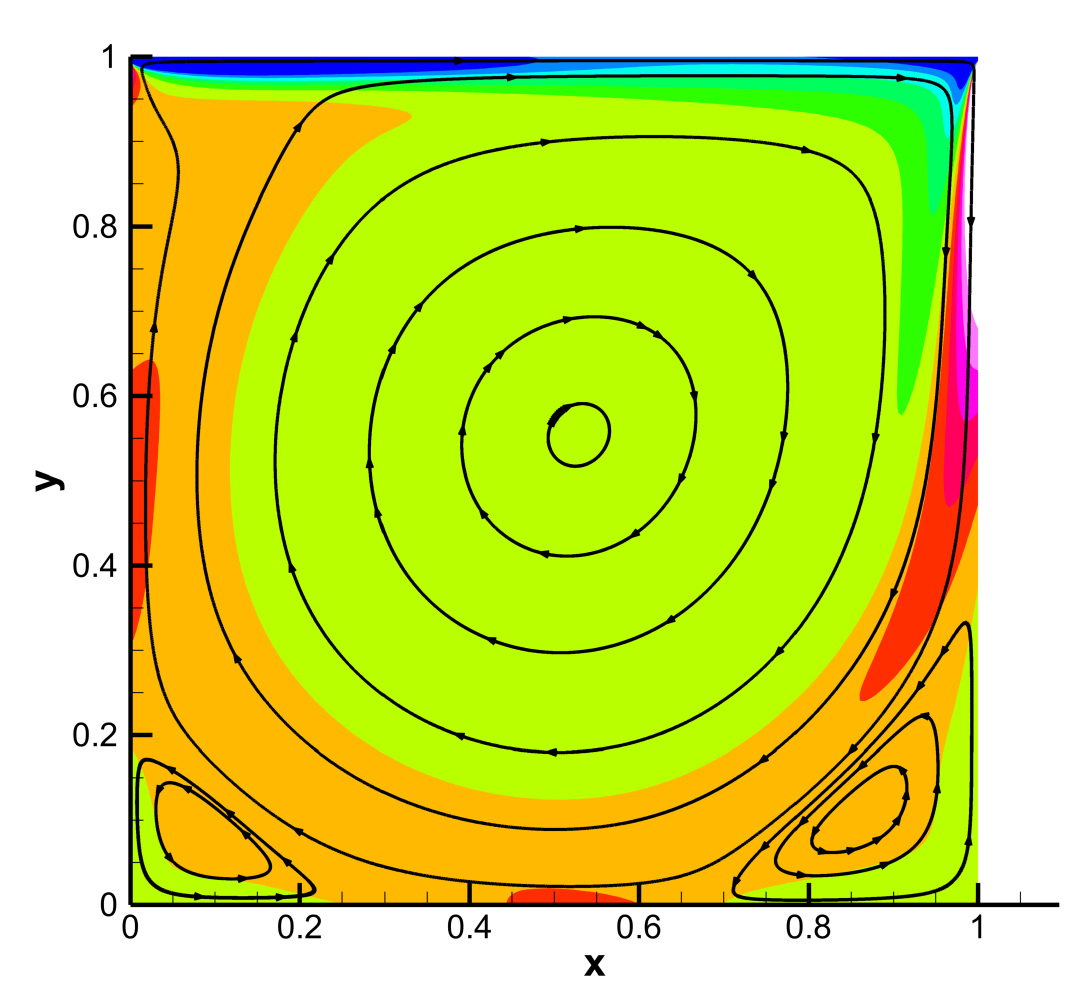}}
         \subfigure[Re=10000]{ \includegraphics[scale=.19]{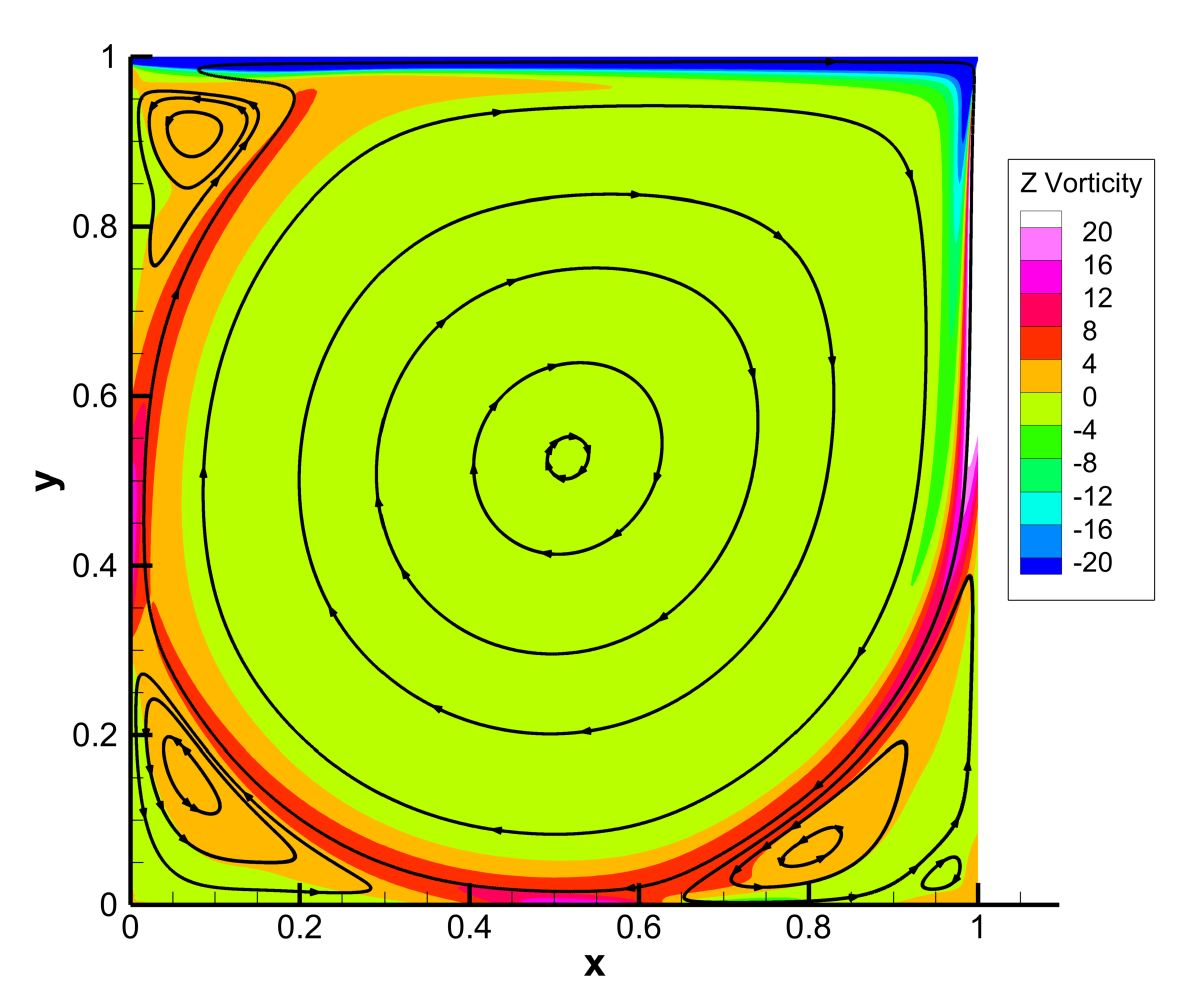}}\\
     \subfigure[Re=100]{ \includegraphics[scale=.2]{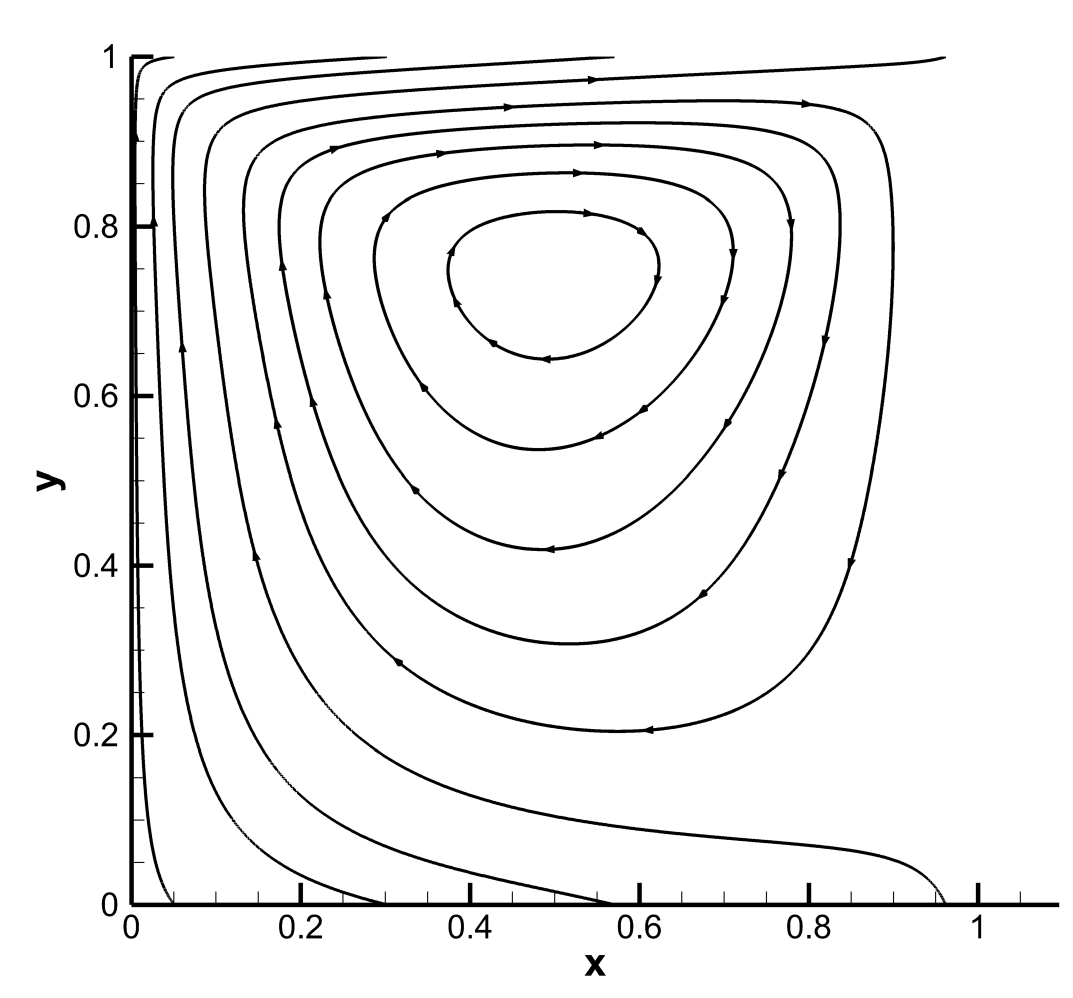}}
    \subfigure[Re=500]{ \includegraphics[scale=.2]{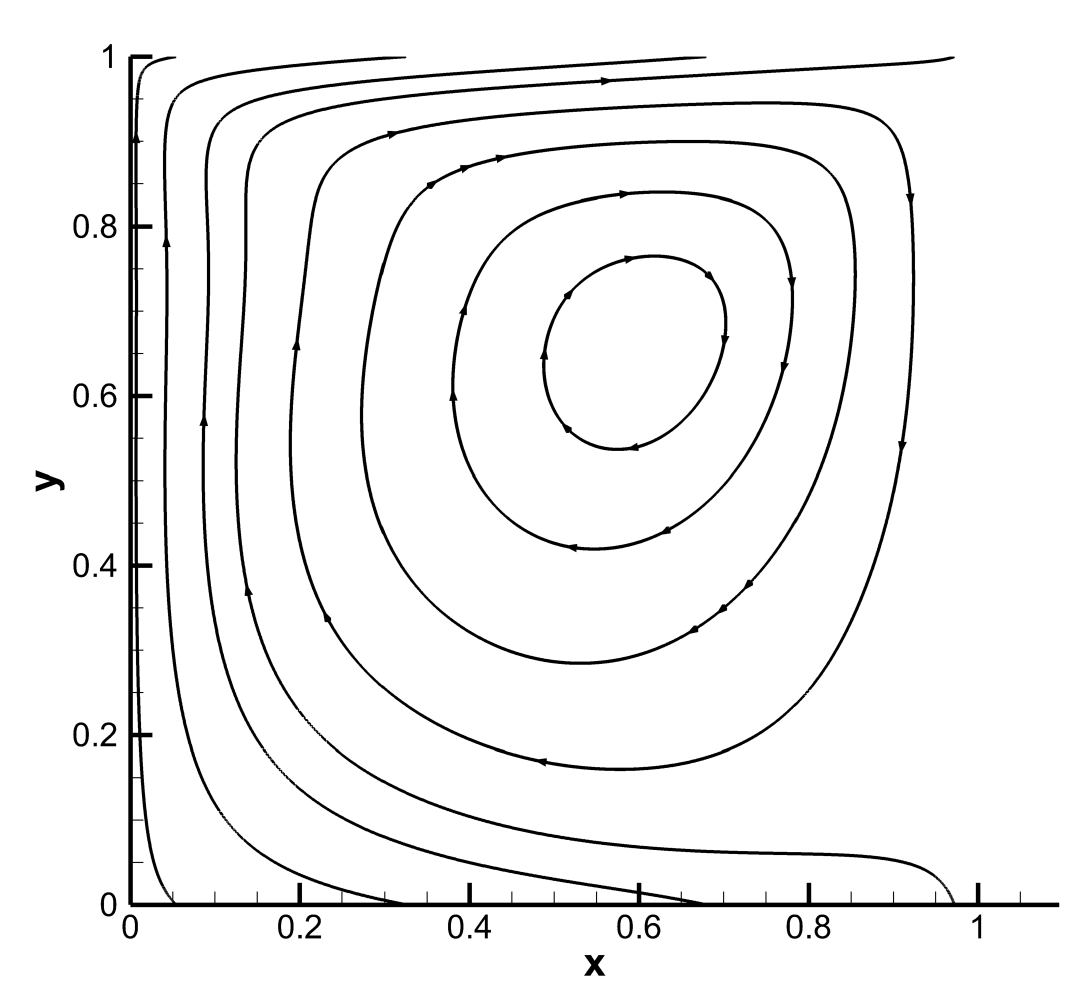}}
  \subfigure[Re=1500]{ \includegraphics[scale=.2]{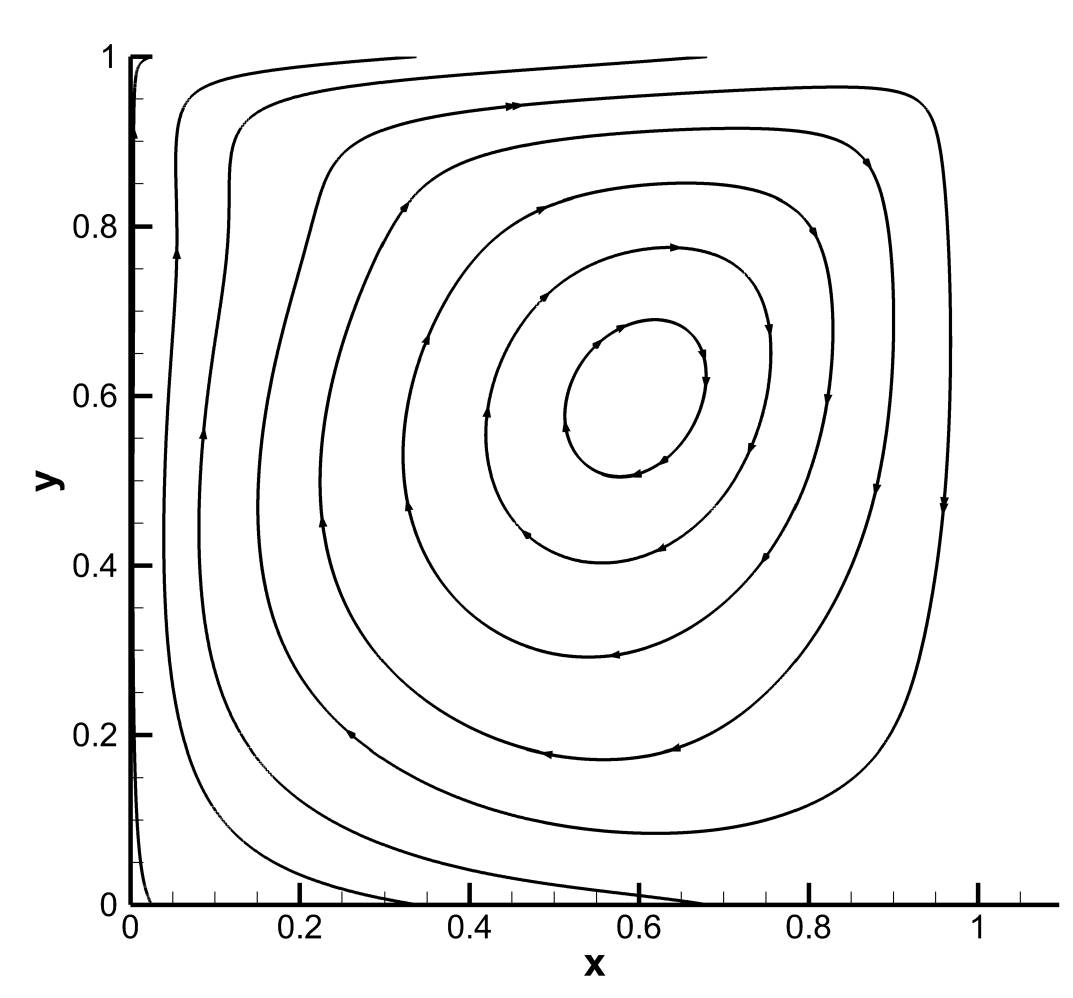}} 
    \subfigure[Re=10000]{ \includegraphics[scale=.2]{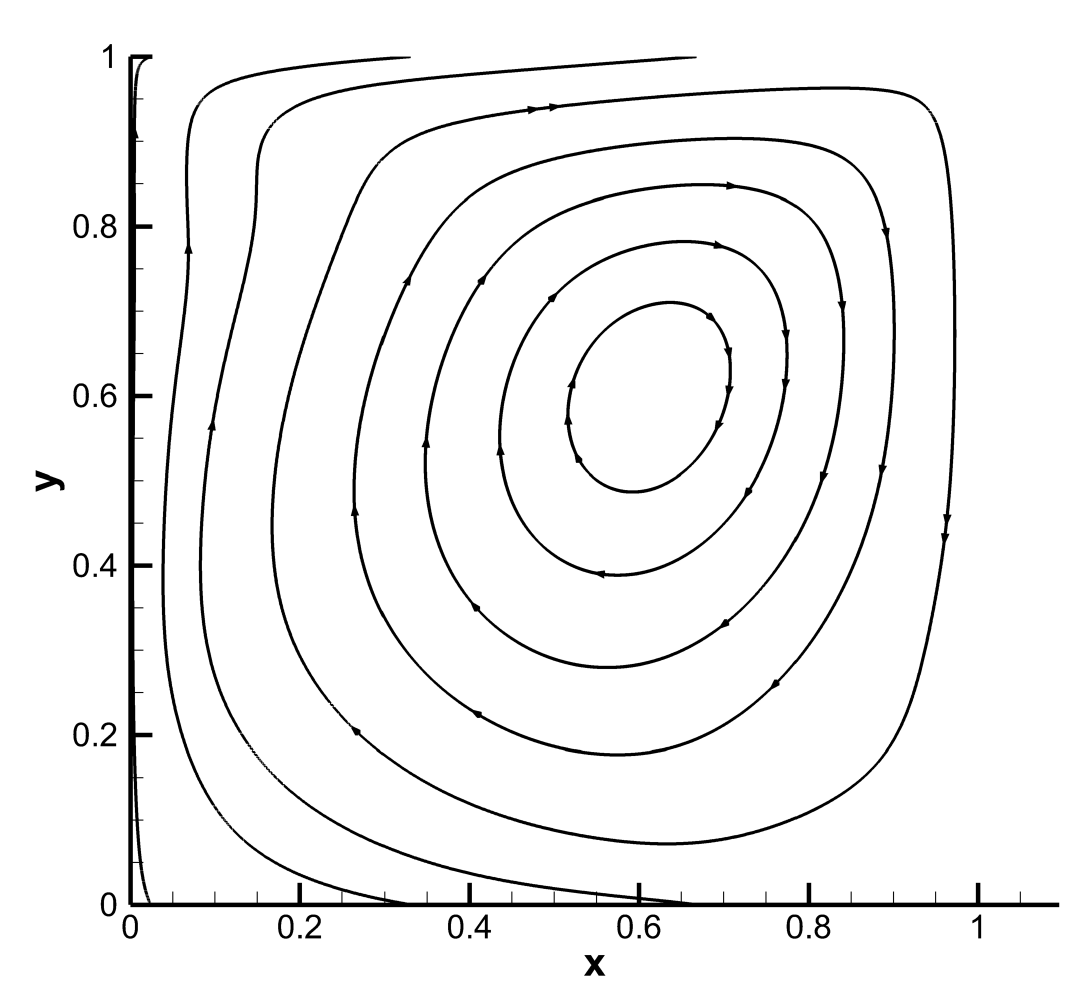}} 
    \caption{ Steady state solutions of driven cavity flow problem in 2D. The streamline of velocity field with the contour of vorticity as background  for (a) ${\rm Re}=100$ (b) ${\rm Re}=500$ (c) ${\rm Re}=1500$ and (d) ${\rm Re}=10000$ and the corresponding streamline of magnet fields in (e)-(h), respectively. The simulation results are obtained with DF-BDF-4 scheme with $N=100$, $\tau=10^{-2}$, ${\rm Rem}=100$ and ${\rm Ha}=\sqrt{10}$. The parameters are obtained by $\nu = {1}/{\rm Re}$,  $\eta=1/{\rm Rem}=0.01$, $\mu \rho =({\rm Re}\times {\rm Rem})/{\rm Ha}^2= 10{\rm Re}$. 
    } 
\label{figs: Velocity2D}
\end{center}
\end{figure}

\subsection{Three dimensional case}

\begin{example}[\bf Convergence test for DF-BDF-$k$ schemes in 3D.]\label{ex: ex6}

Next, we turn to numerical experiments in three dimensions. We conduct convergence tests of the proposed DF-BDF-$k$ schemes using the following manufactured solution in $\Omega=\Lambda^3$:
\begin{equation}\label{eq: analytic3d}
\begin{aligned}
&\beta_i(x_i)=(1-x_i^2)\sin\big({\pi}(x_i+1)/2 \big), \quad \gamma_i(x_i)=\sin({\pi}(x_i+1)/{2}),  \;\; i=1,\cdots,3,\\
& \bs u(\bs x,t)=\Big( 2\beta_1(x_1) \beta_{2}'(x_2) \beta_{3}'(x_3), -\beta_1'(x_1)\beta_2(x_2) \beta'_{3}(x_3),  -\beta_{1}'(x_1)\beta_{2}'(x_2)\beta_3(x_3) \Big)^{\intercal} \cos(t),\\
& \bs B(\bs x,t)=\Big(\frac{2}{\pi} \Big)^2  \Big( 2\gamma_1(x_1) \gamma_{2}'(x_2) \gamma_{3}'(x_3), -\gamma_1'(x_1)\gamma_2(x_2) \gamma'_{3}(x_3),  -\gamma_{1}'(x_1)\gamma_{2}'(x_2)\gamma_3(x_3) \Big)^{\intercal} \cos(2t),\\
&p(\bs x,t)=10\big[ (x_1-1/2)^3\, x_2^2+(1-x_1)^3\,(x_2-1/2)^3+(1-x_1)^2x_2(x_3-1/2)^3 \big],
\end{aligned}
\end{equation}
which satisfy the boundary condition \eqref{eq: mhdbc}. Again, the initial conditions are set as $\bs u_{\rm in}(\bs x)=\bs u(\bs x,t=0)$ and $\bs B_{\rm in}(\bs x)=\bs B(\bs x,t=0)$ to be compatible with the divergence-free constraints and boundary conditions.
Accordingly, the forcing terms $\bs f$, $\bs g$ are chosen such that the analytic expression \eqref{eq: analytic3d} satisfies equation \eqref{eq: mhdnew}. The parameters are chosen as  $\nu=\rho=\mu=\eta=1$. 

The convergence rates are quite similar with the 2D case. We find exponential convergence of numerical errors in space before saturation caused by the temporal discretization error in Figure \ref{fig_BDFerr_N3d} (a). We also observe clearly $k$-th order convergence rate in time for DF-BDF-$k$ schemes with $k=1,\cdots,4$ in Figure \ref{fig_BDFerr_N3d} (b) and Figure \ref{fig_BDF36err_N3d}. As for $k=5,6$, the convergence rates in time are not as expected due to numerical instability, see Figure \ref{fig_BDF36err_N3d}.

\begin{figure}[!htbp]
\begin{center}
  \subfigure[ Errors vs $N$]{ \includegraphics[scale=.38]{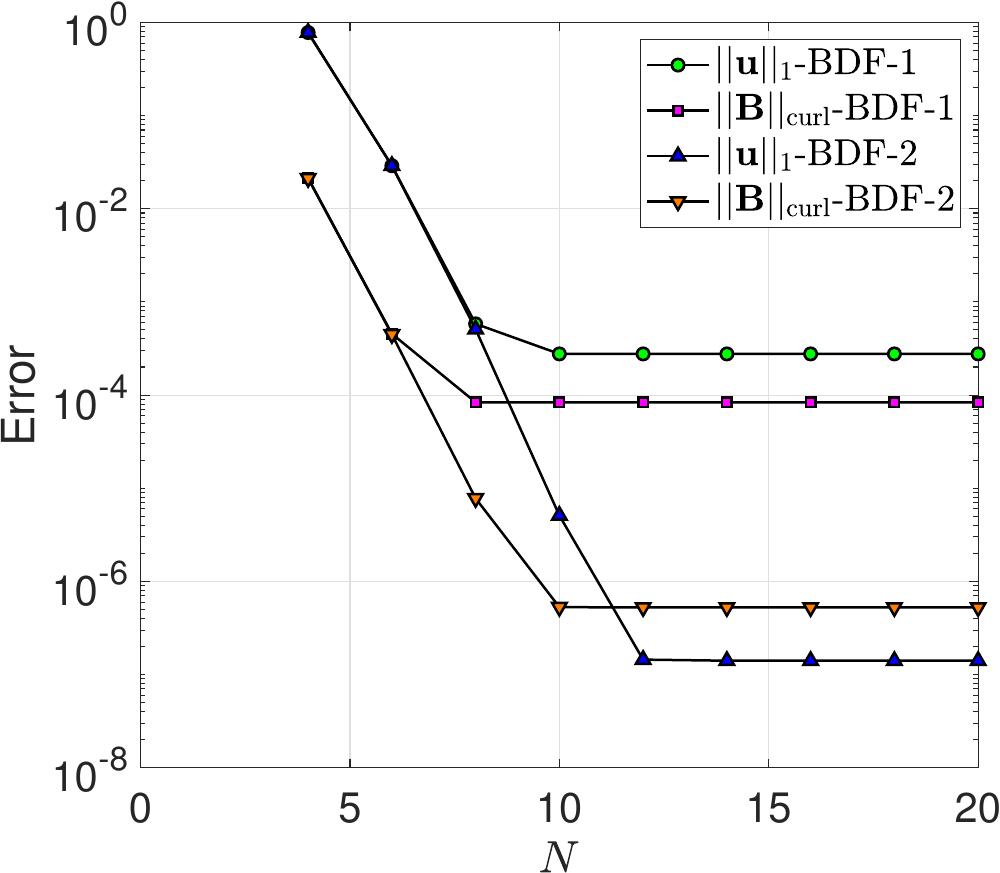}}\qquad
  \subfigure[Errors vs $\tau$ ]{ \includegraphics[scale=.38]{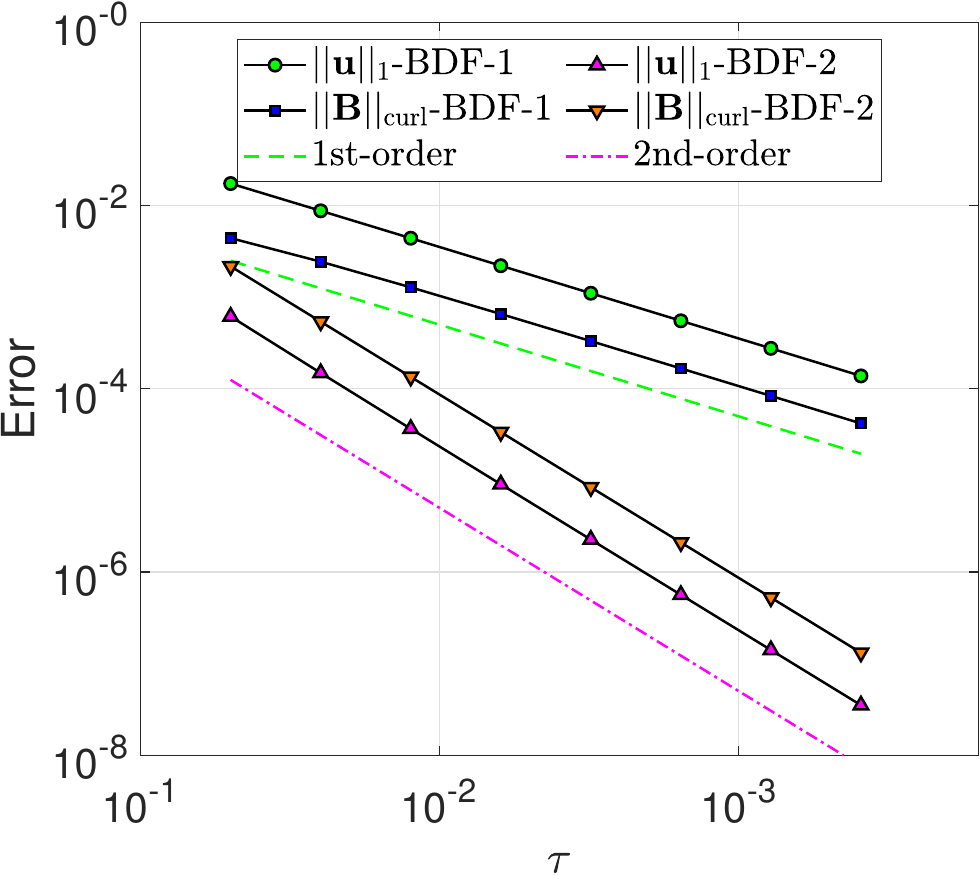}}
 \caption{\small Convergence test of DF-BDF-1-2 schemes in 3D: (a) errors versus polynomial order $N$ (fixed $\tau=3.125\times 10^{-3}$ and $T=0.1$); (b) errors versus time step sizes $\tau$ (fixed polynomial order $N=16$ and $T=1$).
 } 
   \label{fig_BDFerr_N3d}
\end{center}
\end{figure}

\begin{figure}[!htbp]
\begin{center}
    \subfigure[Errors of $\bs u$ vs $\tau$]{ \includegraphics[scale=.38]{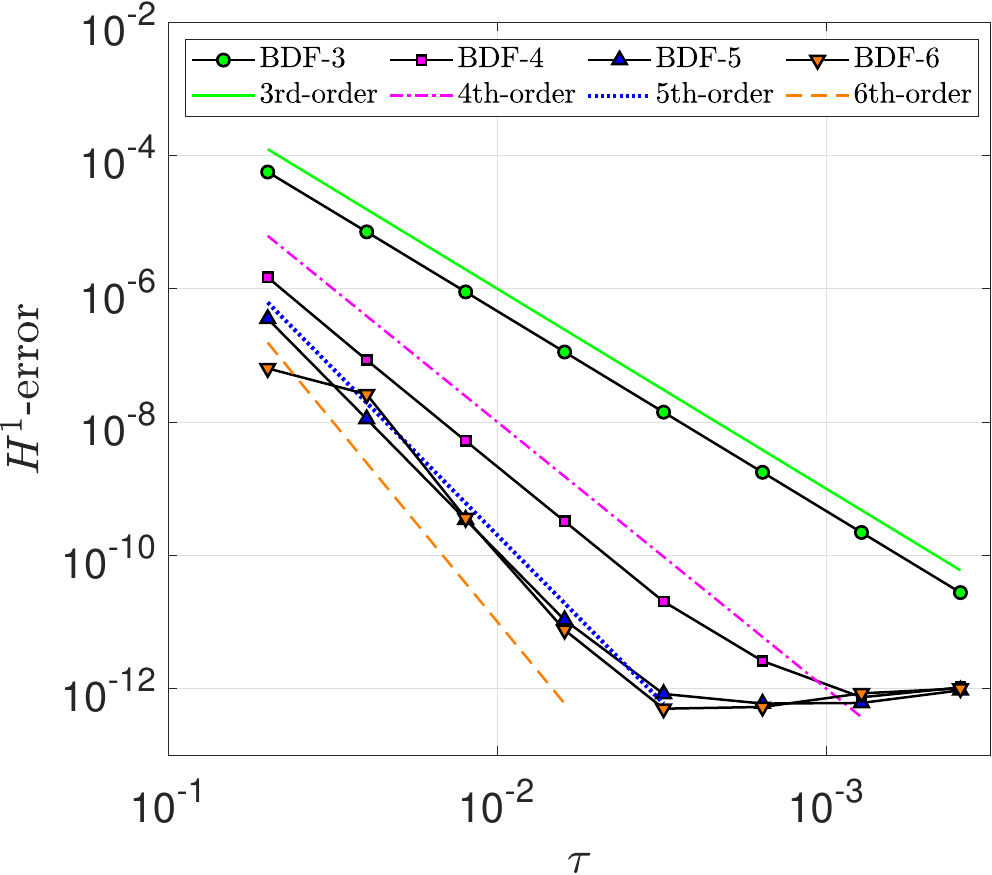}}\qquad
  \subfigure[Errors of $\bs B$ vs $\tau$ ]{ \includegraphics[scale=.38]{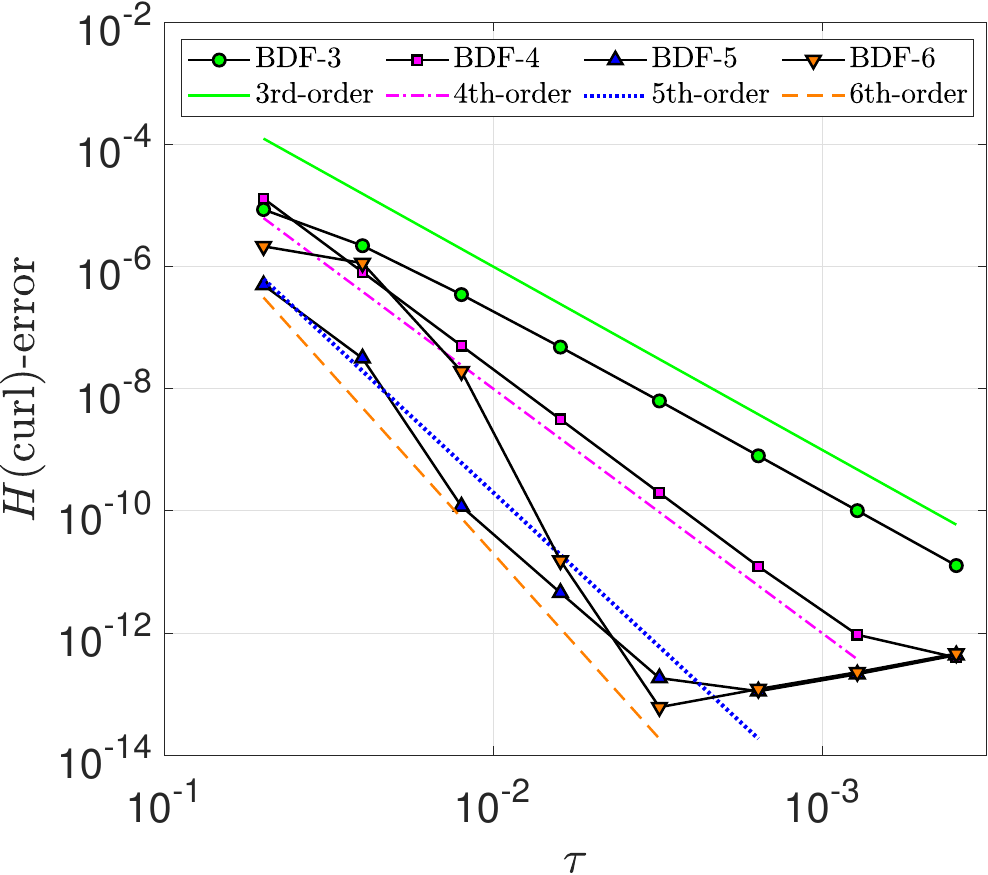}} 
\caption{\small Temporal convergence of DF-BDF-$k$ ($k=3,\cdots, 6$) schemes in 3D: (a) $\bs H^1$-errors of $\bs u$ and (b) $\bs H({\rm curl})$-errors of $\bs B$ versus time step sizes $\tau$ (fixed polynomial order $N=16$ and $T=1$.} 
   \label{fig_BDF36err_N3d}
\end{center}
\end{figure}

 \begin{figure}[!htbp]
\begin{center}
  \subfigure[Iteration numbers]{ \includegraphics[scale=.38]{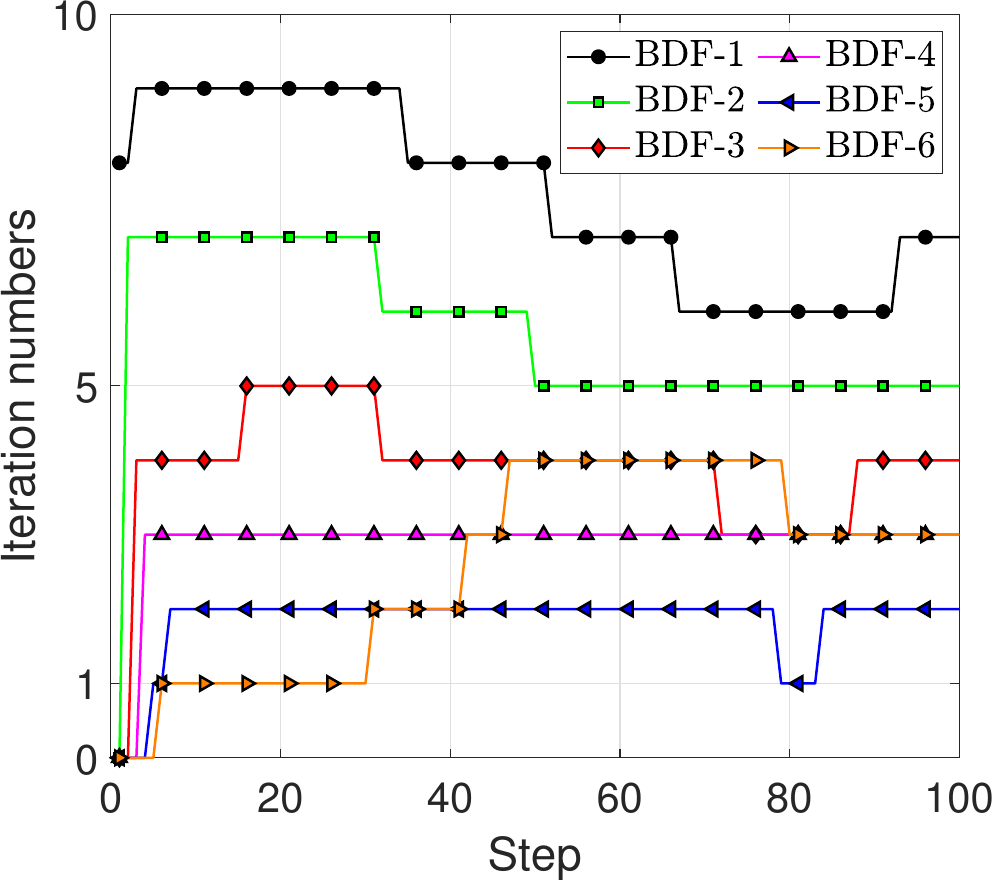}} \qquad
  \subfigure[$H^1$-error long time stability]{ \includegraphics[scale=.38]{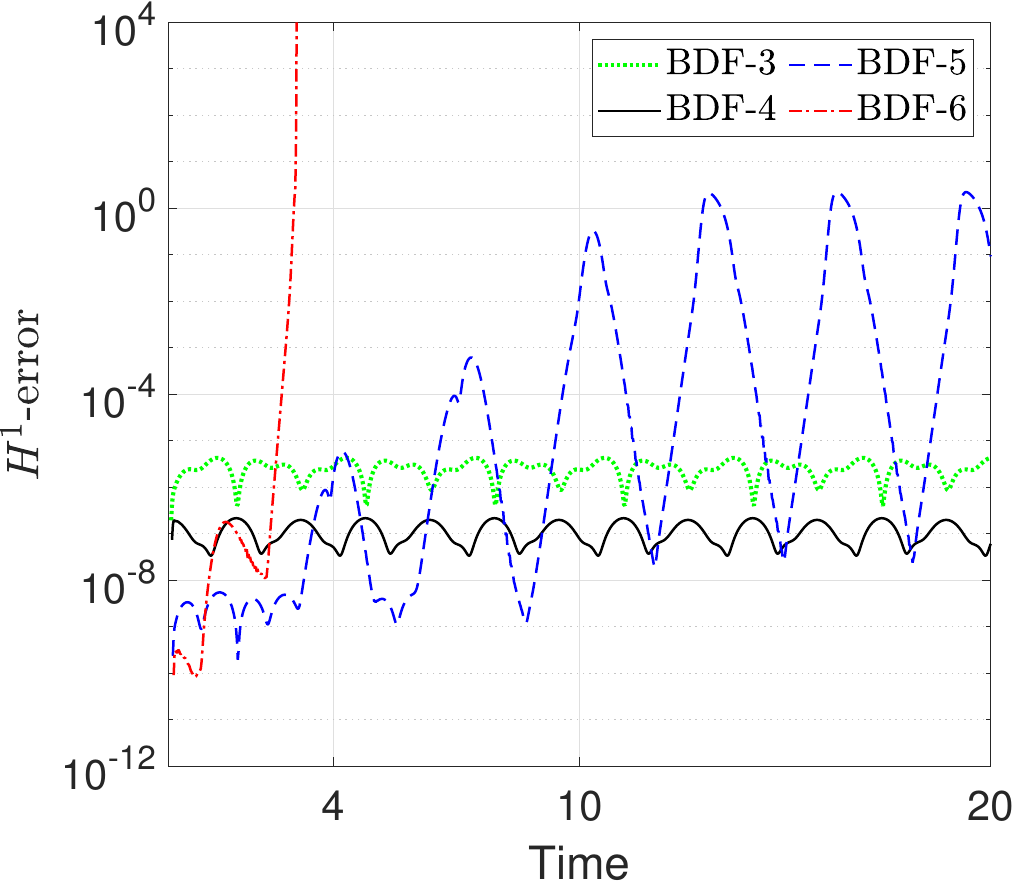}} 
\caption{\small Efficiency and stability of DF-BDF-$k$ schemes in 3D: (a) number of iterations of the sub-iteration method ; (b) time history of errors of $\bs u$. Results are obtained with $\tau=0.02$ and $N=16$.} 
   \label{fig_iterStep_N3d}
\end{center}
\end{figure}

\begin{figure}[!htbp]
\begin{center}
    \subfigure[Re=100]{ \includegraphics[scale=.28]{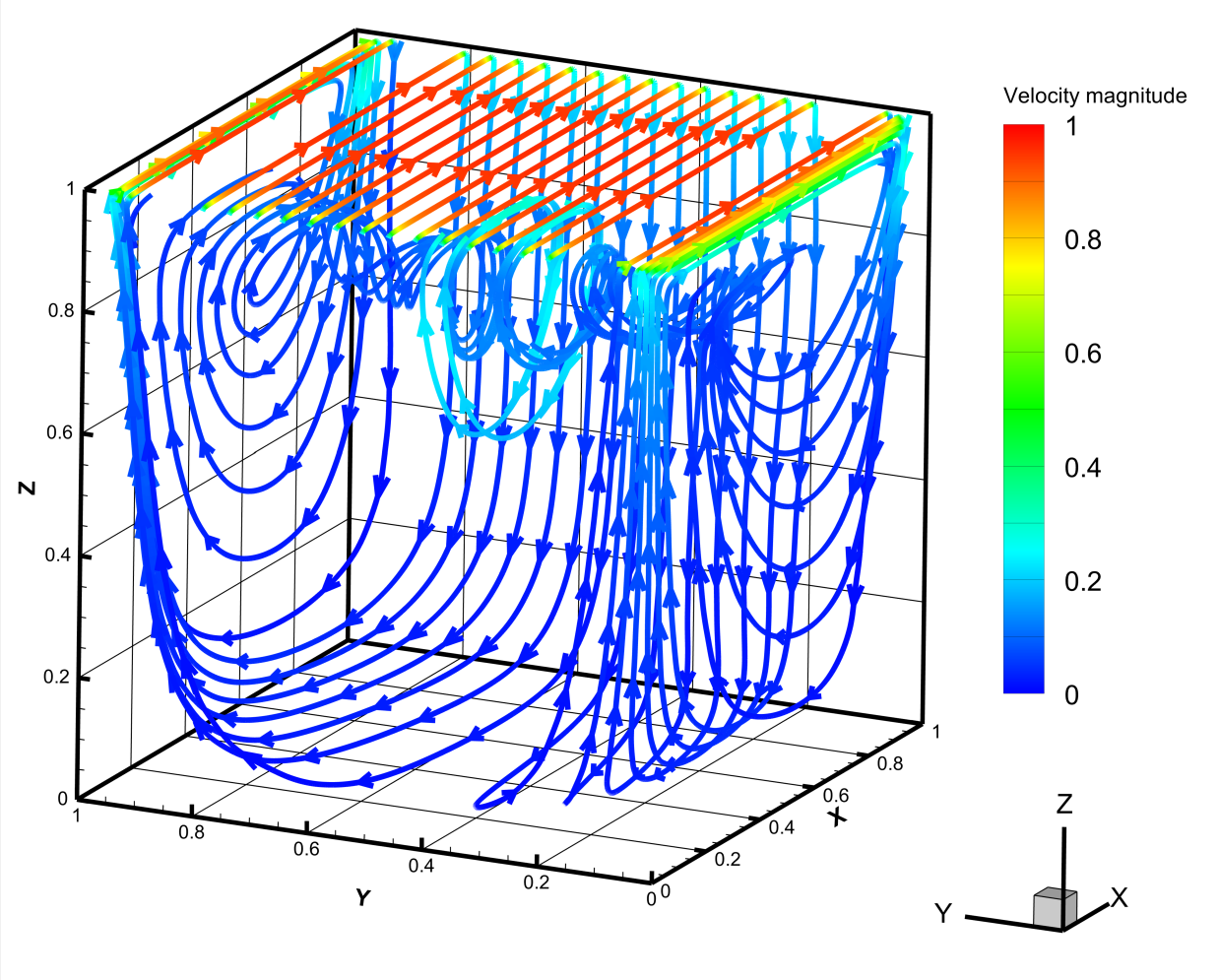}},\quad    
     \subfigure[Re=100]{ \includegraphics[scale=.28]{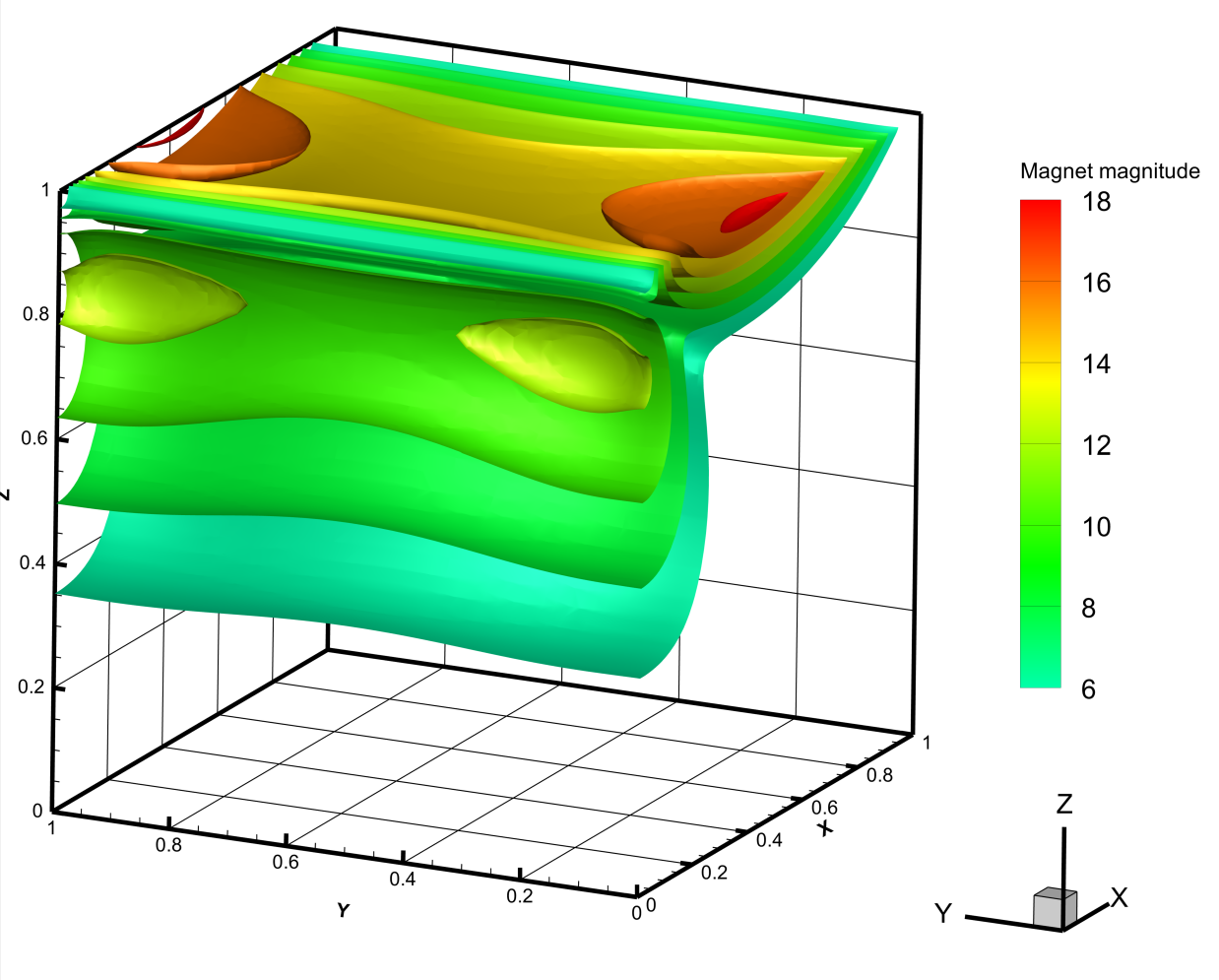}}\\
      \subfigure[Re=1000]{ \includegraphics[scale=.28]{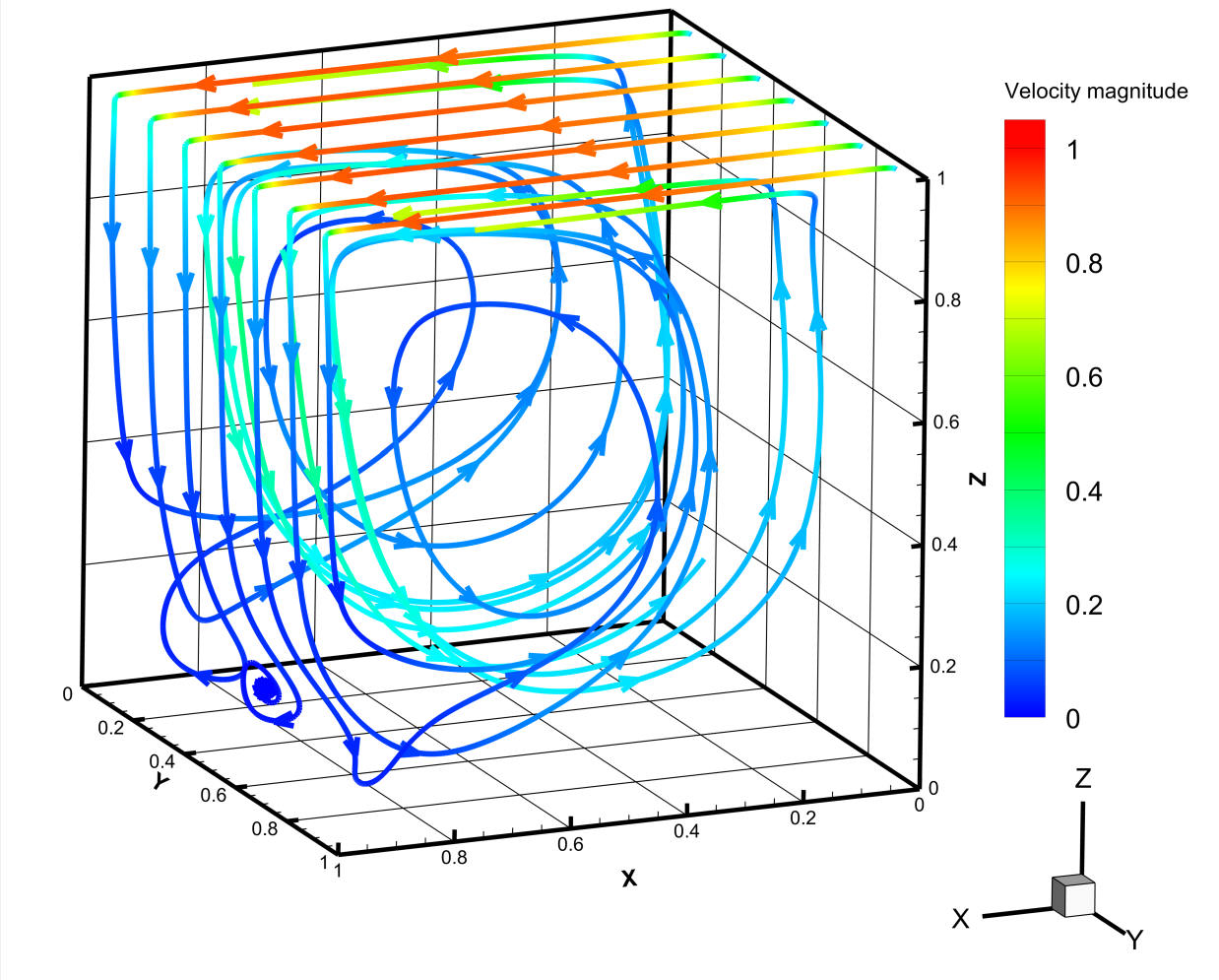}}\quad
  \subfigure[Re=1000]{ \includegraphics[scale=.28]{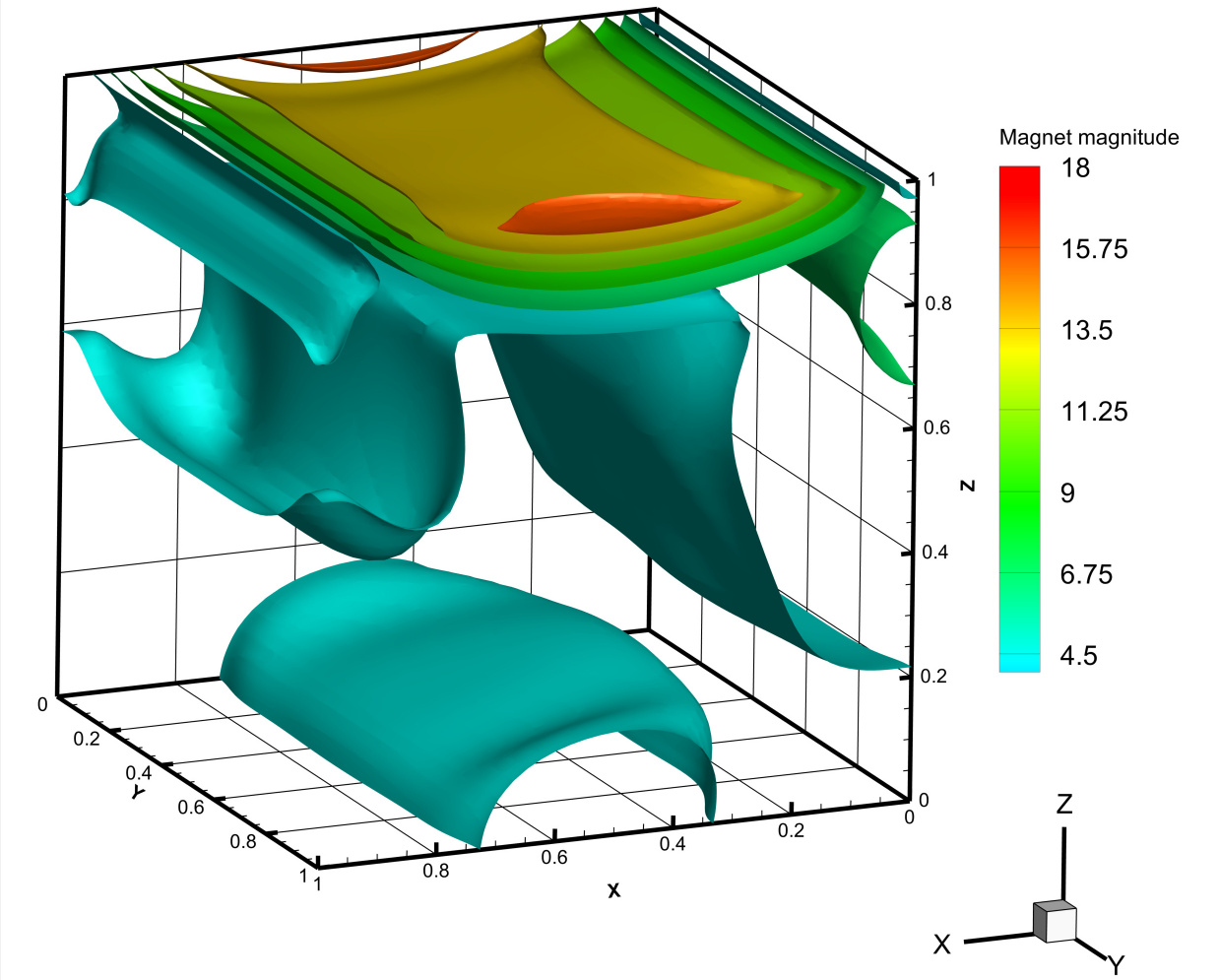}}
    \caption{ Steady state solutions of driven cavity flow problem in 3D. The streamline of velocity field  and the iso-surface of magnetic intensity for (a)-(b) ${\rm Re}=100$ and (c)-(d) ${\rm Re}=1000$, respetively.  The simulation results are obtained with DF-BDF-4 scheme with $N=50$, $\tau=10^{-2}$, ${\rm Rem}=100$ and ${\rm Ha}=\sqrt{10}$.} 
   \label{figs: Cavity3D}
\end{center}
\end{figure}

%
%

\end{example} 




\begin{example}[\bf Efficiency and stability of DF-BDF-$k$ schemes in 3D.]\label{ex: ex7} 

In order to find the efficiency and stability of the proposed DF-BDF-$k$ schemes in three dimensions, we adopt a large time step size $\tau=0.02$ with a fixed $N=16$ and depict in Figure \ref{fig_iterStep_N3d} (a)-(b) the iteration numbers of the sub-iteration method and the numerical errors of the velocity field for long time simulations. Similar with the 2D case, the iteration numbers decrease as $k$ increases from 1 to 4 and suddenly increase for $k=5$ and 6. It takes only 3 iterations for DF-BDF-$4$ to converge. Moreover, the accuracy remains $10^{-7}$ for this scheme, while DF-BDF-$5$-$6$ schemes blow up, as expected.


\end{example}

%

\begin{example}[\bf Driven cavity flow problem in 3D.]\label{ex: ex8}

We conclude the numerical experiments with the driven cavity flow problem in 3D as benchmark. The boundary condition is set with 
\begin{equation}\label{eq: Drivenbd}
\bs u(\bs x,t)=
\begin{cases}
(1,0,0)^{\intercal}, & {\rm at}\, \Gamma_z,\\
(0,0,0)^{\intercal}, & {\rm at}\, \partial \Omega \backslash \Gamma_z,
\end{cases}
\quad \bs n \cdot \bs B=\bs n \cdot \bs B_0,\quad \bs n \times (\nabla \times \bs B)=0\;{\rm at}\;\partial \Omega,
\end{equation}
where $\Gamma_z:=\{(x_1,x_2,x_3)^{\intercal}| 0\leq x_1,x_2\leq 1, x_3=1\}$ and $\bs B_0 = (0,0,1)^\intercal$. The body force term $\bs f(\bs x,t)=0$ in equation \eqref{eq: mhd}. Similar with the two dimensional case, let us fix ${\rm Rem}=100$ and ${\rm Ha}=\sqrt{10}$ and adjust the Renolds number ${\rm Re}=100,1000.$  DF-BDF-4 scheme is employed with $N=50$ and time step size $\tau=10^{-2}$ for the simulation until steady state solutions are obtained.  In Figure \ref{figs: Cavity3D} (a)-(b), we plot the streamline of velocity field and the iso-surface of magnetic intensity measured by $\sqrt{B_1^2+B_2^2+B_3^2}$ for ${\rm Re}=100$. We observe that the center of the streamline is close to the corner $x_1=x_3=1$ and the magnetic intensity is relative large at the faces $x_1=0$ and $x_3=1$. As ${\rm Re}$ increases to 1000, the center of the streamline approaches the center of the region and eddies occur in the lower corner of the domain. These phenomena are consistent with the two dimensional simulations. 


\end{example}

\section{Concluding Remarks}\label{sect: conclude}
In this paper, we presented a systematic approach of devising $\bs H_0({\rm div})$- and $\bs H^1_0$-conforming exact divergence-free spectral basis functions for computational domains diffeomorphic to square or cube in two and three dimensions. These are achieved by an interplay of the orthogonality and derivative property of generalised Jacobi polynomials, several de Rham complexes and the property of contravariant Piola transformation. We employed the proposed divergence-free bases to the approximation of the incompressible and resistive MHD equations based on Galerkin formulation, which decouples the velocity and pressure with the help of the divergence-free test function. Moreover, we proposed DF-BDF-$1$ and DF-BDF-$2$ fully-discretized schemes, which have the distinctive feature of provable unconditional energy stability and exact preservation of divergence-free constraints for both velocity and magnetic fields point-wisely at the same time. These schemes are linear-coupled and can be solved efficiently by a sub-iteration method. Under each iteration step, one only needs to solve two decoupled linear equation for velocity and magnetic fields, respectively. Numerical experiments demonstrate that  they usually take only a few iterations to converge for DF-BDF-$k$ schemes ($k=2,\cdots ,4$), even for large time step sizes. Furthermore, we proposed efficient solution algorithms by exploiting the decoupled nature of the governing equations and the sparsity of the linear algebraic systems resulting from the discretization with the divergence-free basis.  
The methodology presented herein can shed light on developing exact divergence-free spectral-element method for MHD equations in complex geometries, which we shall report in a separate work. It is also worth to explore the possibility to develop unconditionally energy-stable schemes for higher order temporal discretization (order larger than 2) with the help of the proposed divergence-free spectral bases. 
\vskip 10pt

\noindent\underline{\large\bf Acknowledgments}
\vskip 6pt
L. Qin and Z. Yang acknowledge the support from the National Natural Science Foundation of China (No. 12101399), the Shanghai Sailing Program (No. 21YF1421000), and the Strategic Priority Research Program of Chinese Academy of Sciences grant (No. XDA25010402) and the Key Laboratory of Scientific and Engineering Computing (Ministry of Education). H. Li acknowledges the support from the National Natural Science Foundation of China (Nos. 11871455 and 11971016).

\begin{appendix}

\section{  Proof of Proposition  \ref{prop: p3} }\label{ap: b}
\renewcommand{\theequation}{A.\arabic{equation}}
With the help of the de Rham complex \eqref{eq: 3dcompl2} in three dimensions, one can start from the conforming approximation $H_N^1(\Lambda^3)$ of $H^1(\Lambda^3)$ 
\begin{equation}\label{eq: 3dHN1}
H_{N}^1(\Lambda^3):={\rm span}\big( \big\{P_{m}^{(-1,-1)}(\xi_1)P_{n}^{(-1,-1)}(\xi_2)P_{l}^{(-1,-1)}(\xi_3),\;\; 0\leq m,n,l\leq N  \big \} \big),
\end{equation}
and the derivative relation in equation \eqref{eq: derivativep}, one arrives at the conforming approximation space $\bs H_N({\rm div},\Lambda^3)$ of $\bs H({\rm div},\Lambda^3)$:
\begin{equation}\label{eq: hdiv3d}
\begin{aligned}
\bs H_N({\rm div},\Lambda^3)={\rm span}\Big(& \Big\{ P_{m}^{(-1,-1)}(\xi_1)L_{n}(\xi_2)L_{l}(\xi_3)\bs e_1,\;\; 0\leq m\leq N,\; 0\leq n,l\leq N-1\Big\},\\
&\Big\{ L_{m}(\xi_1)P_{n}^{(-1,-1)}(\xi_2)L_{l}(\xi_3)\bs e_2,\;\; 0\leq m,l\leq N-1, \; 0\leq n\leq N \Big\}, \\
&\Big\{ L_{m}(\xi_1)L_{n}(\xi_2)P_{l}^{(-1,-1)}(\xi_3)\bs e_3,\;\; 0\leq m,n\leq N-1,\; 0\leq l\leq N \Big\} \Big).
\end{aligned}
\end{equation}
Once again, we impose the essential boundary condition and decompose $\bs H_N({\rm div},\Lambda^3)$ in terms of the interior modes, face modes, edge modes and vertex modes as follows:
\begin{equation*}
\begin{aligned}
& \bs H_N({\rm div},\Lambda^3)\cap \bs H_0({\rm div},\Lambda^3)={\rm span}\Big( \\
& \Big\{ P_{m+1}^{(-1,-1)}(\xi_1)L_{n}(\xi_2)L_{l}(\xi_3)\bs e_1,\;L_{m}(\xi_1)P_{n+1}^{(-1,-1)}(\xi_2)L_{l}(\xi_3)\bs e_2, \;\;L_{m}(\xi_1)L_{n}(\xi_2)P_{l+1}^{(-1,-1)}(\xi_3)\bs e_3\Big\}_{m,n,l=1}^{N-1},\\
                        & \Big\{ P_{n+1}^{(-1,-1)}(\xi_2) L_l(\xi_3)\bs e_2,\;\; L_n(\xi_2) P_{l+1}^{(-1,-1)}(\xi_3)\bs e_3 \Big\}_{n,l=1}^{N-1}, \\
                        & \Big\{ P_{m+1}^{(-1,-1)}(\xi_1) L_l(\xi_3)\bs e_1,\;\; L_m(\xi_1) P_{l+1}^{(-1,-1)}(\xi_3) \bs e_3 \Big\}_{m,l=1}^{N-1}, \\
                        & \Big\{ P_{m+1}^{(-1,-1)}(\xi_1) L_n(\xi_2)\bs e_1,\;\; L_m(\xi_1) P_{n+1}^{(-1,-1)}(\xi_2)\bs e_2    \Big\}_{m,n=1}^{N-1}, \\
                        & \Big\{  P_{m+1}^{(-1,-1)}(\xi_1)\bs e_1 \Big\}_{m=1}^{N-1},\;\;  \Big\{  P_{n+1}^{(-1,-1)}(\xi_2)\bs e_2 \Big\}_{n=1}^{N-1},\;\; \Big\{ P_{l+1}^{(-1,-1)}(\xi_3)\bs e_3 \Big\}_{l=1}^{N-1} \Big).
\end{aligned}
\end{equation*}
Expanding a function $\bs v(\bs \xi)$ in $\bs H_N({\rm div},\Lambda^3)\cap \bs H_0({\rm div},\Lambda^3)$, utilizing the derivative property \eqref{eq: derivativep} and imposing divergence-free condition, one has
\begin{equation}\label{eq:3dBexp}
\begin{aligned}
\bs v(\bs \xi)&=\sum_{m,n,l=1}^{N-1} \Big(  \hat v_{m,n,l}^1P_{m+1}^{(-1,-1)}(\xi_1)L_{n}(\xi_2)L_{l}(\xi_3)\bs e_1+ \hat v_{m,n,l}^2  L_{m}(\xi_1)P_{n+1}^{(-1,-1)}(\xi_2)L_{l}(\xi_3)\bs e_2 \\
& \qquad \qquad \qquad+\hat v_{m,n,l}^3 L_{m}(\xi_1)L_{n}(\xi_2)P_{l+1}^{(-1,-1)}(\xi_3)\bs e_3  \Big)\\
&+\sum_{n,l=1}^{N-1} \Big(   \hat v_{fx,n,l}^1 P_{n+1}^{(-1,-1)}(\xi_2) L_l(\xi_3) \bs e_2 +\hat v_{fx,n,l}^2 L_n(\xi_2) P_{l+1}^{(-1,-1)}(\xi_3)\bs e_3  \Big) \\
&+\sum_{m,l=1}^{N-1} \Big(   \hat v_{fy,m,l}^1 P_{m+1}^{(-1,-1)}(\xi_1) L_l(\xi_3) \bs e_1 +\hat v_{fy,m,l}^2 L_m(\xi_1) P_{l+1}^{(-1,-1)}(\xi_3)\bs e_3  \Big)\\
&+\sum_{m,n=1}^{N-1} \Big(   \hat v_{fz,m,n}^1 P_{m+1}^{(-1,-1)}(\xi_1) L_n(\xi_2) \bs e_1 +\hat v_{fz,m,n}^2 L_m(\xi_1) P_{n+1}^{(-1,-1)}(\xi_2)\bs e_2  \Big)\\
&+\sum_{m=1}^{N-1} \hat v_{e,m}^1 P_{m+1}^{(-1,-1)}(\xi_1)\bs e_1 +\sum_{n=1}^{N-1} \hat v_{e,n}^2 P_{n+1}^{(-1,-1)}(\xi_2)\bs e_2 +\sum_{l=1}^{N-1} \hat v_{e,l}^3 P_{l+1}^{(-1,-1)}(\xi_3)\bs e_3 ,
\end{aligned}
\end{equation}
and
\begin{equation}
\begin{aligned}
\nabla \cdot \bs v(\bs \xi)&=\frac{1}{2} \sum_{m,n,l=1}^{N-1}\big( m\hat v_{m,n,l}^1+ n\hat v_{m,n,l}^2+l\hat v_{m,n,l}^3 \big)L_{m}(\xi_1)L_{n}(\xi_2)L_{l}(\xi_3)\\
&+\frac{1}{2} \sum_{n,l=1}^{N-1} \big(n \hat v_{fx,n,l}^1 +l \hat v_{fx,n,l}^2 \big) L_n(\xi_2) L_l(\xi_3)\\
&+\frac{1}{2} \sum_{m,l=1}^{N-1} \big( m \hat v^1_{fy,m,l} +l \hat v^2_{fy,m,l} \big) L_m(\xi_1)L_l(\xi_3)\\
&+\frac{1}{2} \sum_{m,n=1}^{N-1} \big( m \hat v_{fz,m,n}^1 +n \hat v^2_{fz,m,n} \big)L_m(\xi_1) L_n(\xi_2)\\&+\frac{1}{2} \Big( \sum_{m=1}^{N-1} m \hat v_{e,m}^1 L_m(\xi_1)+ \sum_{n=1}^{N-1} n \hat v_{e,n}^2 L_n(\xi_2) +\sum_{l=1}^{N-1}  l \hat v_{e,l}^3 L_l(\xi_3)    \Big)=0.
\end{aligned}
\end{equation}
By the orthogonality of $\{L_m(\xi)\}_{m=0}^{\infty}$ in $L^2$- inner product, one arrives at the following linear system for the expansion coefficients 
\begin{align}
& m\hat v_{m,n,l}^1+ n\hat v_{m,n,l}^2+l\hat v_{m,n,l}^3=0, \quad 1\leq m,n,l\leq N-1, \label{eq: apphdivl1}\\
& n \hat v_{fx,n,l}^1 +l \hat v_{fx,n,l}^2=0,\quad 1\leq n,l \leq N-1,\label{eq: apphdivl2}\\
& m \hat v^1_{fy,m,l} +l \hat v^2_{fy,m,l}=0,\quad 1\leq m,l \leq N-1,\label{eq: apphdivl3}\\
& m \hat v_{fz,m,n}^1 +n \hat v^2_{fz,m,n}=0,\quad 1\leq m,n \leq N-1,\label{eq: apphdivl4}\\
& \hat v_{e,m}^1=\hat v_{e,m}^2=\hat v_{e,m}^3=0,\quad 1\leq m \leq N-1. \label{eq: apphdivl5}
\end{align}
Let us treat equation \eqref{eq: apphdivl1} first.  Since $(m,n,l)^{\intercal}\neq \bs 0\in \mathbb{R}^3$, the null space of $(m,n,l)^{\intercal}$ consists of two non trivial vectors $\bs \zeta_{m,n,l}^1$ and $\bs \zeta_{m,n,l}^2$ in $\mathbb{R}^3$. We properly scale $\bs \zeta_{m,n,l}^1$ and $\bs \zeta_{m,n,l}^2$ for the sake of numerical consideration as
\begin{equation}\label{eq: 3dzeta12}
\begin{aligned}
&\bs \zeta_{m,n,l}^1=\Big(\frac{\sqrt{2(2m+1)}}{m}  \sqrt{\frac{2n+1}{2}}  \sqrt{\frac{2l+1}{2}},-\sqrt{\frac{2m+1}{2}} \frac{\sqrt{2(2n+1)}}{n} \sqrt{\frac{2l+1}{2}},0\Big)^{\intercal},\\
&  \bs \zeta_{m,n,l}^2=\Big( \frac{\sqrt{2(2m+1)}}{m}  \sqrt{\frac{2n+1}{2}}  \sqrt{\frac{2l+1}{2}}, 0,   -\sqrt{\frac{2m+1}{2}} \sqrt{\frac{2n+1}{2}} \frac{\sqrt{2(2l+1)}}{l}  \Big)^{\intercal}.
\end{aligned}
\end{equation}
Thus, one can obtain that
\begin{equation}\label{eq:3dBcoe}
(\hat v_{m,n,l}^1,\hat v_{m,n,l}^2,\hat v_{m,n,l}^3)^{\intercal}=\tilde v_{m,n,l}^1  \bs \zeta_{m,n,l}^1 + \tilde v_{m,n,l}^2  \bs \zeta_{m,n,l}^2.
\end{equation}
Equations \eqref{eq: apphdivl2}-\eqref{eq: apphdivl4} can be treated similarly as the two-dimensional case, i.e. one has
\begin{align}
& \hat v^1_{fx,n,l}=  \frac{\sqrt{2(2n+1)}}{n} \sqrt{\frac{2l+1}{2}} \tilde v_{fx,n,l} , \quad \hat v^2_{fx,n,l}=- \sqrt{\frac{2n+1}{2}} \frac{\sqrt{2(2l+1)}}{l} \tilde v_{fx,n,l},  \label{eq: fa1}\\
& \hat v^1_{fy,m,l}=  \frac{\sqrt{2(2m+1)}}{m} \sqrt{\frac{2l+1}{2}} \tilde v_{fy,m,l},\quad \hat v^2_{fy,m,l}=- \sqrt{\frac{2m+1}{2}} \frac{\sqrt{2(2l+1)}}{l} \tilde v_{fy,m,l} , \label{eq: fa2}\\
& \hat v^1_{fz, m,n}=  \frac{\sqrt{2(2m+1)}}{m} \sqrt{\frac{2n+1}{2}} \tilde v_{fz,m,n},\quad  \hat v^2_{fz, m,n}=- \sqrt{\frac{2m+1}{2}} \frac{\sqrt{2(2n+1)}}{n} \tilde v_{fz,m,n}. \label{eq: fa3}
\end{align}

Inserting equations \eqref{eq:3dBcoe}-\eqref{eq: fa3} back into equation \eqref{eq:3dBexp} and utilizing the notations in equations \eqref{eq:divBb1}-\eqref{eq:hdive3}, one arrives at
\begin{equation*}
\begin{aligned}
\bs v(\bs \xi)=& \sum_{m,n,l=1}^{N-1} \Big\{      \tilde v_{m,n,l}^{1} \bs \Phi_{m,n,l}^1(\bs \xi)+ \tilde v_{m,n,l}^{2} \bs \Phi_{m,n,l}^2(\bs \xi)   \Big\} \\
+&\sum_{n,l=1}^{N-1}  \tilde v_{fx,n,l} \bs \Phi^{x}_{n,l}(\bs \xi) + \sum_{m,l=1}^{N-1}  \tilde v_{fy,m,l} \bs \Phi^{y}_{m,l}(\bs \xi)+ \sum_{m,n=1}^{N-1} \tilde v_{fz,m,n}  \bs \Phi^{z}_{m,n}(\bs \xi),
\end{aligned}
\end{equation*}
which leads to the desired divergence-free basis in Proposition \ref{prop: p3} with the help of Lemma \ref{eq: piola}.

\section{  Proof of Proposition  \ref{prop: prop4} }\label{ap:d}
\renewcommand{\theequation}{B.\arabic{equation}}

The procedure for deriving the divergence-free $\bs H_0^1$-conforming basis is similar with the above appendix, but much more involved. One starts by the conforming approximation $H^2_N(\Lambda^3)$ of $H^2(\Lambda^3)$ 
\begin{equation}\label{eq:h2N}
H^2_N(\Lambda^3)={\rm span}\big( \big\{P_m^{(-2,-2)}(\xi_1)P_n^{(-2,-2)}(\xi_2)P_l^{(-2,-2)}(\xi_3),\;\;0\leq m,n,l\leq N  \big\}  \big)
\end{equation}
and resort to the three-dimensional Stokes complex \eqref{eq: stokescom3d} to obtain $\mathbb{X}_N\in \bs H^1(\Lambda^3)$
\begin{equation*}
\begin{aligned}
\mathbb{X}_N={\rm span}\Big(&  P_{m}^{(-2,-2)}(\xi_1)(P_{n}^{(-2,-2)})'(\xi_2)(P_{l}^{(-2,-2)})'(\xi_3)\bs e_1, \\
& (P_{m}^{(-2,-2)})'(\xi_1)P_{n}^{(-2,-2)}(\xi_2)(P_{l}^{(-2,-2)})'(\xi_3)\bs e_2,  \\
 & (P_{m}^{(-2,-2)})'(\xi_1)(P_{n}^{(-2,-2)})'(\xi_2) P_{l}^{(-2,-2)}(\xi_3)\bs e_3,\quad  0\leq m,n,l\leq N \Big).
\end{aligned}
\end{equation*}
Equivalently, $\mathbb{X}_N$ can be decomposed into discrete subspaces $\mathbb{X}_{\rm in}$, $\big(\mathbb{X}_{\rm fx}, \mathbb{X}_{\rm fy}, \mathbb{X}_{\rm fz}\big)$, $\big(\mathbb{X}_{\rm ex}, \mathbb{X}_{\rm ey}, \mathbb{X}_{\rm ez}\big)$ and $\mathbb{X}_{\rm v}$,  accounting for the interior modes, face modes, edge modes and vertex modes, respectively as follows:
\begin{equation}
\mathbb{X}_N=\mathbb{X}_{\rm in}\oplus \mathbb{X}_{\rm fx}\oplus \mathbb{X}_{\rm fy}\oplus \mathbb{X}_{\rm fz}\oplus \mathbb{X}_{\rm ex} \oplus \mathbb{X}_{\rm ey} \oplus \mathbb{X}_{\rm ez} \oplus \mathbb{X}_{\rm v},
\end{equation}
where
\begin{subequations}
\mathleft
\begin{equation}\label{eq: u3din}
\begin{aligned}
\mathbb{X}_{\rm in}={\rm span}\Big( \Big\{  & P_{m+3}^{(-2,-2)}(\xi_1)P_{n+2}^{(-1,-1)}(\xi_2)P_{l+2}^{(-1,-1)}(\xi_3)\bs e_1, \\
& P_{m+2}^{(-1,-1)}(\xi_1)P_{n+3}^{(-2,-2)}(\xi_2)P_{l+2}^{(-1,-1)}(\xi_3)\bs e_2,  \\
 & P_{m+2}^{(-1,-1)}(\xi_1)P_{n+2}^{(-1,-1)}(\xi_2) P_{l+3}^{(-2,-2)}(\xi_3)\bs e_3  \Big\}_{m,n,l=1}^{N-3} \Big),
\end{aligned}
\end{equation}
\begin{equation}\label{eq: u3dfx}
\begin{aligned}
\mathbb{X}_{\rm fx}={\rm span}\Big( \Big\{ &  p_3(\xi_1)P_{n+2}^{(-1,-1)}(\xi_2)P_{l+2}^{(-1,-1)}(\xi_3)\bs e_1, \;\; p_2(\xi_1) P_{n+3}^{(-2,-2)}(\xi_2) P_{l+2}^{(-1,-1)}(\xi_3)\bs e_2,  \\
 & p_2(\xi_1)P_{n+2}^{(-1,-1)}(\xi_2) P_{l+3}^{(-2,-2)}(\xi_3)\bs e_3  \Big\}_{n,l=1}^{N-3} \Big),
\end{aligned}
\end{equation}
\begin{equation}\label{eq: u3dfy}
\begin{aligned}
\mathbb{X}_{\rm fy}={\rm span}\Big( \Big\{ &P_{m+3}^{(-2,-2)}(\xi_1)  p_2(\xi_2)P_{l+2}^{(-1,-1)}(\xi_3)\bs e_1, \;\; P_{m+2}^{(-1,-1)}(\xi_1) p_3(\xi_2) P_{l+2}^{(-1,-1)}(\xi_3)\bs e_2,  \\
 &P_{m+2}^{(-1,-1)}(\xi_1) p_2(\xi_2)  P_{l+3}^{(-2,-2)}(\xi_3) \bs e_3  \Big\}_{m,l=1}^{N-3} \Big),
\end{aligned}
\end{equation}
\begin{equation}\label{eq: u3dfz}
\begin{aligned}
\mathbb{X}_{\rm fz}={\rm span}\Big( \Big\{ &  P_{m+3}^{(-2,-2)}(\xi_1) P_{n+2}^{(-1,-1)}(\xi_2)p_2(
\xi_3)\bs e_1, \;\;  P_{m+2}^{(-1,-1)}(\xi_1) P_{n+3}^{(-2,-2)}(\xi_2) p_2(\xi_3) \bs e_2,  \\
 & P_{m+2}^{(-1,-1)}(\xi_1) P_{n+2}^{(-1,-1)}(\xi_2)p_3(\xi_3) \bs e_3  \Big\}_{m,n=1}^{N-3} \Big),
\end{aligned}
\end{equation}
\begin{equation}\label{eq: u3dex}
\begin{aligned}
\mathbb{X}_{\rm ex}={\rm span}\Big( \Big\{ &  P_{m+3}^{(-2,-2)}(\xi_1) p_2(\xi_2)p_2(\xi_3)\bs e_1, \;\;  P_{m+2}^{(-1,-1)}(\xi_1) p_3(\xi_2) p_2(\xi_3) \bs e_2,  \\
 & P_{m+2}^{(-1,-1)}(\xi_1) p_2(\xi_2)p_3(\xi_3) \bs e_3  \Big\}_{m=1}^{N-3} \Big),
\end{aligned}
\end{equation}
\begin{equation}\label{eq: u3dey}
\begin{aligned}
\mathbb{X}_{\rm ey}={\rm span}\Big( \Big\{ &  p_3(\xi_1)P_{n+2}^{(-1,-1)}(\xi_2)p_2(\xi_3)\bs e_1, \;\; p_2(\xi_1) P_{n+3}^{(-2,-2)}(\xi_2) p_2(\xi_3)\bs e_2,  \\
 & p_2(\xi_1)P_{n+2}^{(-1,-1)}(\xi_2) p_3(\xi_3)\bs e_3  \Big\}_{n=1}^{N-3} \Big),
\end{aligned}
\end{equation}
\begin{equation}\label{eq: u3dez}
\begin{aligned}
\mathbb{X}_{\rm ez}={\rm span}\Big( \Big\{ &p_3(\xi_1)  p_2(\xi_2)P_{l+2}^{(-1,-1)}(\xi_3)\bs e_1, \;\; p_2(\xi_1) p_3(\xi_2) P_{l+2}^{(-1,-1)}(\xi_3)\bs e_2,  \\
 &p_2(\xi_1) p_2(\xi_2)  P_{l+3}^{(-2,-2)}(\xi_3) \bs e_3  \Big\}_{l=1}^{N-3} \Big),
\end{aligned}
\end{equation}
\begin{equation}\label{eq: u3dv}
\mathbb{X}_{\rm v}={\rm span}\Big( \Big\{ p_3(\xi_1) p_2(\xi_2)p_2(\xi_3)\bs e_1, \;\; p_2(\xi_1) p_3(\xi_2) p_2(\xi_3)\bs e_2,  
 p_2(\xi_1) p_2(\xi_2)  p_3(\xi_3) \bs e_3  \Big\} \Big),
\end{equation}
\end{subequations}
where $p_k(\xi)$ is an arbitrary polynomial in $\mathbb{P}_k$,  with $\mathbb{P}_k$ be the polynomial space with polynomial degree less than or equal to $k$.

Observe that
\begin{equation}\label{eq: c2h01}
\mathbb{P}_2(\Lambda) \cap H_0^1(\Lambda)={\rm span}\big( P_2^{(-1,-1)}(\xi) \big).
\end{equation}
Thus, in order to obtain a divergence-free basis satisfying the Dirichlet boundary condition,  one needs to replace $p_2(\xi)$ with $P_2^{(-1,-1)}(\xi)$. Moreover, from the fact that
\begin{equation}\label{eq: c20}
\{0\}=\big\{f(\xi) |\; f(\xi)\in \mathbb{P}_3(\Lambda) \cap H_0^1(\Lambda) ,\;\; f'(\xi)=P_2^{(-1,-1)}(\xi) \big\},
\end{equation}
the terms involving $p_2$ and $p_3$ could not be solenoidal and satisfy the Dirichlet boundary condition at the same time. With the help of the observations in equations \eqref{eq: c2h01}-\eqref{eq: c20}, one verify that  
\begin{equation*}
\{\bs v(\bs \xi)\in \bs H_0^1(\Lambda^3),\;\; \nabla \cdot \bs v=0 \} \cap \mathbb{X}_{\rm ex} \oplus \mathbb{X}_{\rm ey} \oplus \mathbb{X}_{\rm ez} \oplus \mathbb{X}_{\rm v}=\{ 0\}.
\end{equation*} 
Thus the edge modes and vertex modes will not be included in the desired divergence-free basis. 
Consequently, one can expanding $\bs v(\xi)\in \mathbb{X}_N \cap \big\{ \bs v\in \bs H_0^1(\Lambda^3),\;\; \nabla \cdot \bs v=0 \big\}$ and evaluate its divergence using the derivation relation \eqref{eq: derivativep}  as follows:
\begin{equation}\label{eq: u3ddivexp}
\begin{aligned}
\bs v(\bs \xi)=\sum_{m,n,l=1}^{N-3} \Big( &  \hat v_{m,n,l}^1 P_{m+3}^{(-2,-2)}(\xi_1)P_{n+2}^{(-1,-1)}(\xi_2)P_{l+2}^{(-1,-1)}(\xi_3)\bs e_1\\
+&\hat v_{m,n,l}^2 P_{m+2}^{(-1,-1)}(\xi_1)P_{n+3}^{(-2,-2)}(\xi_2)P_{l+2}^{(-1,-1)}(\xi_3)\bs e_2\\
+&\hat v_{m,n,l}^3 P_{m+2}^{(-1,-1)}(\xi_1)P_{n+2}^{(-1,-1)}(\xi_2) P_{l+3}^{(-2,-2)}(\xi_3)\bs e_3\Big)\\
+\sum_{n,l=1}^{N-3}\Big(& \hat v_{fx,n,l}^1 P_2^{(-1,-1)}(\xi_1) P_{n+3}^{(-2,-2)}(\xi_2) P_{l+2}^{(-1,-1)}(\xi_3)\bs e_2\\
+& \hat v_{fx,n,l}^2 P_2^{(-1,-1)}(\xi_1)P_{n+2}^{(-1,-1)}(\xi_2) P_{l+3}^{(-2,-2)}(\xi_3)\bs e_3   \Big)\\
+\sum_{m,l=1}^{N-3}\Big(& \hat v_{fy,m,l}^1 P_{m+3}^{(-2,-2)}(\xi_1)  P_2^{(-1,-1)}(\xi_2)P_{l+2}^{(-1,-1)}(\xi_3)\bs e_1     \\
+& \hat v_{fy,m,l}^2  P_{m+2}^{(-1,-1)}(\xi_1) P_2^{(-1,-1)}(\xi_2)  P_{l+3}^{(-2,-2)}(\xi_3) \bs e_3     \Big)\\
+\sum_{n,l=1}^{N-3}\Big(& \hat v_{fz,n,l}^1  P_{m+3}^{(-2,-2)}(\xi_1) P_{n+2}^{(-1,-1)}(\xi_2)P_2^{(-1,-1)}(\xi_3)\bs e_1      \\
+& \hat v_{fz,n,l}^2  P_{m+2}^{(-1,-1)}(\xi_1) P_{n+3}^{(-2,-2)}(\xi_2) P_2^{(-1,-1)}(\xi_3) \bs e_2     \Big)
\end{aligned}
\end{equation} 
and
\begin{equation}
\begin{aligned}
\nabla \cdot \bs v(\bs \xi)&=\frac{1}{2}\sum_{m,n,l=1}^{N-3} \big( m\hat v_{m,n,l}^1 +n \hat v_{m,n,l}^2+ l\hat v_{m,n,l}^3  \big)P_{m+2}^{(-1,-1)}(\xi_1)P_{n+2}^{(-1,-1)}(\xi_2)P_{l+2}^{(-1,-1)}(\xi_3)\\
&+\frac{1}{2}\sum_{n,l=1}^{N-3} \big( n\hat v_{fx,n,l}^1+l \hat v_{fx,n,l}^2 \big)P_2^{(-1,-1)}(\xi_1) P_{n+2}^{(-1,-1)}(\xi_2) P_{l+2}^{(-1,-1)}(\xi_3)\\
&+\frac{1}{2}\sum_{m,l=1}^{N-3} \big( m\hat v_{fy,m,l}^1+l \hat v_{fy,m,l}^2 \big)  P_{m+2}^{(-1,-1)}(\xi_1)  P_2^{(-1,-1)}(\xi_2)P_{l+2}^{(-1,-1)}(\xi_3)\\
&+\frac{1}{2}\sum_{m,n=1}^{N-3} \big( m\hat v_{fz,m,n}^1+n \hat v_{fz,m,n}^2 \big)  P_{m+2}^{(-1,-1)}(\xi_1) P_{n+2}^{(-1,-1)}(\xi_2)P_2^{(-1,-1)}(\xi_3).
\end{aligned}
\end{equation}
Following the same procedure, one can utilize the property that $\{P^{(-1,-1)}_m(\xi)\}_{m=4}^{\infty}$ is orthogonal w.r.t the singular weight function $(1-\xi)^{-1}(1+\xi)^{-1}$ on $\Lambda$ and the fact that $\nabla \cdot \bs u=0$ to arrive at a linear system of expansion coefficients as in equations \eqref{eq: apphdivl1}-\eqref{eq: apphdivl4}. Thus, with the same derivation, we arrives at the desired divergence free basis in Proposition \ref{prop: prop4}.

\section{Backward differentiation formulas}\label{ap: bdfk}
\renewcommand{\theequation}{C.\arabic{equation}}

The $k$-step backward differentiation formula is give by
\begin{equation}\label{eq: BDFgeneral}
\frac{\partial \chi}{\partial t}\Big|_{t=t_{n+1}}=\frac{{\gamma \chi^{n+1}-\hat \chi}}{\tau},
\end{equation}
where
\begin{equation}\label{eq: BDF3}
\begin{aligned}
& \hat \chi=
\begin{cases}
 3\chi^n-\dfrac{3}{2} \chi^{n-1}+\dfrac{1}{3} \chi^{n-2}, & k=3\\[5pt]
 4\chi^n-3\chi^{n-1}+\dfrac{4}{3}\chi^{n-2}-\dfrac{1}{4}\chi^{n-3}, & k=4 \\[5pt]
 5\chi^n-5\chi^{n-1}+\dfrac{10}{3}\chi^{n-2} -\dfrac{15}{12}\chi^{n-3}+\dfrac{1}{5}\chi^{n-4}, & k=5\\
 6\chi^n-\dfrac{45}{6}\chi^{n-1}+\dfrac{20}{3}\chi^{n-2} -\dfrac{45}{12}\chi^{n-3}+\dfrac{6}{5}\chi^{n-4}-\dfrac{1}{6}\chi^{n-5}, & k=6
\end{cases},\quad 
\gamma=
\begin{cases}
\dfrac{11}{6}, & k=3\\[5pt]
\dfrac{25}{12}, & k=4\\[5pt]
\dfrac{137}{60}, & k=5\\[5pt]
\dfrac{147}{60}, & k=6
\end{cases}\,, \\
\end{aligned}
\end{equation}
and the $k$-th order explicit approximation of $\chi$ at time $t_{n+1}$ takes the form 
\begin{equation}\label{eq: app3}
\tilde{\chi}^{n+1}=
\begin{cases}
 3\chi^n-3 \chi^{n-1}+  \chi^{n-2}, & k=3\\
 4\chi^n-6 \chi^{n-1}+4 \chi^{n-2}-  \chi^{n-3}, & k=4\\
 5\chi^n-10\chi^{n-1}+10\chi^{n-2}- 5\chi^{n-3}+ \chi^{n-4}, & k=5\\
 6\chi^n-15\chi^{n-1}+20\chi^{n-2}-15\chi^{n-3}+6\chi^{n-4}-  \chi^{n-5}, & k=6
\end{cases}\;.
\end{equation}

\end{appendix}

\bibliography{refpapers}

\end{document}